\documentclass[10pt]{article}
\usepackage{latexsym}
\usepackage{indentfirst}
\usepackage{amssymb,amsfonts,amsmath,amsthm}
\usepackage{graphicx}
\usepackage{mathrsfs}
\usepackage{array}
\usepackage{color}
\usepackage{epstopdf}
\usepackage[all]{xy}
\usepackage{hyperref}
\usepackage{url}
\usepackage{appendix}
 
\usepackage{geometry}
\newtheorem{thm}{Theorem}[section]
\newtheorem{prop}[thm]{Proposition}
\newtheorem{lem}[thm]{Lemma}

\newtheorem{cor}[thm]{Corollary}
\newtheorem{exam}[thm]{Example}

\newtheorem{fact}[thm]{Fact}
\newtheorem{rmk}[thm]{Remark}
\newtheorem{dfn}[thm]{Definition}
\newtheorem{cons}[thm]{Construction}

\numberwithin{equation}{section}

\newcommand{\frakM}{{\mathfrak M}}

\newcommand{\frakS}{{\mathfrak S}}

\newcommand{\bE}{{\mathbb E}}
\newcommand{\bF}{{\mathbb F}}

\newcommand{\bL}{{\mathbb L}}
\newcommand{\bM}{{\mathbb M}}

\newcommand{\bQ}{{\mathbb Q}}

\newcommand{\bZ}{{\mathbb Z}}

\newcommand{\calC}{{\mathcal C}}
\newcommand{\calD}{{\mathcal D}}

\newcommand{\calF}{{\mathcal F}}

\newcommand{\calH}{{\mathcal H}}
\newcommand{\calI}{{\mathcal I}}

\newcommand{\calM}{{\mathcal M}}

\newcommand{\calO}{{\mathcal O}}

\newcommand{\calR}{{\mathcal R}}

\newcommand{\calT}{{\mathcal T}}

\newcommand{\calX}{{\mathcal X}}
\newcommand{\calY}{{\mathcal Y}}

\newcommand{\Ext}{{\mathrm{Ext}}}           
\newcommand{\Gal}{{\mathrm{Gal}}}           
\newcommand{\Hom}{{\mathrm{Hom}}}           
\newcommand{\Lan}{{\mathrm{Lan}}} 
\newcommand{\Map}{\mathrm{Map}}
\newcommand{\Mod}{{\mathrm{Mod}}}           
\newcommand{\nil}{\mathrm{nil}}
\newcommand{\Pro}{{\mathrm{Pro}}}

\newcommand{\Spf}{{\mathrm{Spf}}}           
\newcommand{\Spec}{{\mathrm{Spec}}}         
\newcommand{\Vect}{{\mathrm{Vect}}}         

\newcommand{\PreStk}{{\mathrm{PreStk}}}

\newcommand{\GL}{{\mathrm{GL}}}             

\newcommand{\cl}{{\mathrm{\rm cl}}}             
\newcommand{\cyc}{{\mathrm{cyc}}}           

\newcommand{\perf}{\mathrm{perf}}           

\newcommand\rightthreearrow{\substack{\rightarrow\\[-1em] \rightarrow \\[-1em] \rightarrow}
}

\usepackage{relsize}
\usepackage[bbgreekl]{mathbbol}
\DeclareSymbolFontAlphabet{\mathbb}{AMSb} 
\DeclareSymbolFontAlphabet{\mathbbl}{bbold}
\newcommand{\Prism}{{\mathlarger{\mathbbl{\Delta}}}} 

\usepackage{titletoc}
\titlecontents{section}[1.72em]{\vspace{1pt}}%
    {\thecontentslabel.\enspace}
    {}
    {\titlerule*[0.5pc]{.}\contentspage}%
\titlecontents{subsection}[1.72em]{\vspace{1pt}}%
    {\thecontentslabel.\enspace}
    {}
    {\titlerule*[0.5pc]{.}\contentspage}%
\titlecontents{subsubsection}[1.72em]{\vspace{1pt}}%
    {\thecontentslabel.\enspace}
    {}
    {\titlerule*[0.5pc]{.}\contentspage}%

\begin{document}
\title{Classicality of derived Emerton--Gee stack }
\author{Yu Min\footnote{Affilication: Department of Mathematics, Imperial College London, UK. Email: y.min@imperial.ac.uk}}
\date{}
\maketitle

\begin{abstract}
    We construct a derived stack $\calX$ of Laurent $F$-crystals on $(\calO_K)_{\Prism}$, where $\calO_K$ is the ring of integers of a finite extension $K$ of $\bQ_p$. We first show that its underlying classical stack $^{\rm cl}\calX$ coincides with the Emerton--Gee stack $\calX_{\rm EG}$, i.e. the moduli stack of \'etale $(\varphi,\Gamma)$-modules. Then we prove that the derived stack $\calX$ is classical in the sense that  when restricted to truncated animated rings, $\calX$ is equivalent to the sheafification of the left Kan extension of $\calX_{\rm EG}$ along the inclusion from the classical commutative rings to animated rings.
\end{abstract}

\tableofcontents

\section{Introduction}
Recent years have seen a lot of progress on the categorification of the (arithmetic) local Langlands correspondence, e.g. \cite{Hell20}, \cite{Zhu20},\cite{DHKM20}, \cite{FS21},  \cite{BZCHN20}, \cite{EGH22}. In particular, various methods have been used to construct the moduli stacks of Langlands parameters. When $l\neq p$, the moduli stack of local Langlands parameters turns out to be a quotient of a local complete intersection ring, which implies the derived stack of local Langlands parameters in this case is indeed classical. When $l=p$, the moduli stack of local Galois representations is no longer the right geometric object to consider. Instead, one needs to look at the moduli stack $\calX_{\rm EG}$ of \'etale $(\varphi,\Gamma)$-modules, constructed by Emerton--Gee in \cite{EG22}. It is then tempting to ask if there is a well-defined derived version of $\calX_{\rm EG}$ and if it is also classical. 

On the other hand, the derived stack of global Langlands parameters is rarely classical. In order to study a global-to-local map of moduli stacks of Langlands parameters, it is then desirable to have a derived version of the Emerton--Gee stack. And this derived extension should have reasonable (co)tangent complexes as suggested by \cite[Remark 9.2.3]{EGH22}.

In this paper, we will propose a construction of a derived Emerton--Gee stack in the $\GL_d$-case and investigate its classicality. Note that among all the derived extensions of the Emerton--Gee stack, there is an initial one, i.e. its left Kan extension from the classical commutative rings to animated rings. By classicality of a derived stack, we mean essentially it is the left Kan extension of its underlying classical stack. But usually, the left Kan extension of a stack, which is constructed abstractly, does not enjoy natural moduli interpretations. So in other words, our task is really to find a good moduli interpretation of the left Kan extension of the Emerton--Gee stack.

Let us first revisit the Emerton--Gee stack. Let $K$ be a finite extension of $\bQ_p$ and $\calO_K$ be its ring of integers. By \cite{BS22} and \cite{Wu21}, we now know that there are natural equivalences between the category $\Vect((\calO_K)_{\Prism},\calO_{\Prism}[\frac{1}{\calI_{\Prism}}]^{\wedge}_p)^{\varphi=1}$ of Laurent $F$-crystals on the absolute prismatic site $(\calO_K)_{\Prism}$ of $\calO_K$, the category of \'etale $(\varphi,\Gamma)$-modules and the category of finite free $\bZ_p$-representations of the absolute Galois group $G_K:=\Gal(\bar K/K)$. Instead of defining a moduli stack of \'etale $(\varphi,\Gamma)$-modules as in \cite{EG22}, one can try to construct a moduli stack of Laurent $F$-crystals. Indeed, our first main result says that the two stacks turn out to be equivalent, which generalises \cite{Wu21} to the non-trivial coefficient case.

\begin{thm}[Theorem \ref{main1}]\label{Intro-main1}
    The moduli stack $^{\rm cl}\calX$ of Laurent $F$-crystals on $(\calO_K)_{\Prism}$, i.e. the functor sending each $p$-nilpotent ring $R$ to the groupoid $\Vect((\calO_K)_{\Prism},\calO_{\Prism,R}[\frac{1}{\calI}])^{{\varphi}=1,\simeq}$ of Laurent $F$-crystals with coefficient in $R$, is equivalent to the Emerton--Gee stack $\calX_{\rm EG}$. Here the sheaf $\calO_{\Prism,R}[\frac{1}{\calI}]$ is defined by $\calO_{\Prism,R}[\frac{1}{\calI}](A,I):=(A\widehat\otimes^{\bL}_{\bZ_p}R)[\frac{1}{I}]$ for any $(A,I)\in (\calO_K)_{\Prism}$.
\end{thm}

The basic idea of proving this theorem is to use the absolute perfect prismatic site as a bridge. When the coefficient rings are finite type $\bZ_p$-algebras, there will be no difference between Laurent $F$-crystals on $(\calO_K)^{\rm perf}_{\Prism}$ and those on $(\calO_K)_{\Prism}$. In this case, both of them can be proved to be the same as \'etale $(\varphi,\Gamma)$-modules. To extend the equivalence from finite type algebras to arbitrary algebras, the limit-preserving property of the moduli stack of \'etale $\varphi$-modules will be of great importance to us. We remark that it is essential to use the Breuil--Kisin prism, which is related to the Kummer extension instead of the cyclotomic extension. This is because the cyclotomic prism is only well-defined in the unramified case. So in order to deal with the ramified case, the Breuil--Kisin prism is necessary.


Now we turn to our initial question, i.e. construct a good derived extension of the Emerton--Gee stack. Naturally, one needs to figure out how to define derived \'etale $(\varphi,\Gamma)$-modules. The subtlety lies in the continuous derived $\Gamma$-action. Note that the prismatic site or more precisely, the perfect prismatic site (or $v$-site) provides an intrinsic way to deal with continuous $\Gamma$-action. This suggests we stay in the realm of prismatic site to define derived stack of Laurent $F$-crystals and regard it as the derived Emerton--Gee stack in light of Theorem \ref{Intro-main1}.

Allowing $R$ to be any $p$-nilpotent animiated ring in Theorem \ref{Intro-main1}, we get a derived stack $\calX$ of Laurent $F$-crystals. Here comes the question: is $\calX$ classical? Our second main result provides an affirmative answer to it.

\begin{thm}[Theorem \ref{main2}]\label{intro-main2}
    The derived stack $\calX$ of Laurent $F$-crystals is classical up to nilcompletion, i.e. when restricted to truncated animated rings, the derived stack $\calX$ is equivalent to $(\Lan\calX_{\rm EG})^{\#}$, the (\'etale) sheafification of the left Kan extension of the Emerton--Gee stack along the inclusion from classical commutative rings to animated rings.
\end{thm}

\begin{rmk}
    The nilcompletion appearing in the above theorem is not essential as we mostly care about nilcomplete stacks in practice. Taking the nilcompletion is essentially due to that the left Kan extension of a prestack might not be nilcomplete in general (see Remark \ref{remark-nil}).
\end{rmk}


In fact, Theorem \ref{intro-main2} is infinitesimally guaranteed by \cite[Theorem 1.1]{BIP23},  which states that the framed local Galois deformation ring is a local complete intersection ring. By \cite[Lemma 7.5]{GV18}, this means the deformation functor of local Galois representations is classical. That being said, to extend this to the whole derived stack instead of just at the infinitesimal level is not as straightforward as the $l\neq p$ case. Namely, the Emerton--Gee stack is only a Noetherian formal algebraic stack instead of an algebraic stack, and the derived stack of Laurent $F$-crystals is only a derived stack abstractly defined, which is not a derived algebraic stack, either. This will cause essential troubles in proving the classicality.

Our basic tool is a global derived deformation theory, developed in \cite{GR17} (see \ref{appendix} for some subtleties in our context). The philosophy behind the global deformation theory is that the Postnikov tower of any animated ring is actually a successive extension of square-zero extensions, which is determined by the cotangent complex. In some good cases, comparing two derived stacks will be reduced to comparing their pro-cotangent complexes. This is roughly how we finally prove Theorem \ref{intro-main2}. More precisely, we have the following criterion.

\begin{prop}[{\cite[Chapter 1, Proposition 8.3.2]{GR17}}]
    Let $\calF_1\to \calF_2$ be a map of derived prestacks admitting a deformation theory. Suppose there exists a commutative diagram,
    \begin{equation*}
        \xymatrix@=0.6cm{
        & ^{\rm cl}\calF_0\ar[ld]^{g_1}\ar[rd]^{g_2}&\\
        \calF_1\ar[rr]^{f}&&\calF_2
        }
    \end{equation*}
where $g_1,g_2$ are nilpotent embeddings, and $^{\rm cl}\calF_{0}$ is a classical prestack. Suppose also that for any discrete commutative ring $R\in {Nilp}_{\bZ_p}$, and a map $x_0\in {^{\rm cl}\calF_{0}}(R)$ and $x_i=g_i\circ x_0$ where $i=1,2$, the induced map
    \[
    T^*_{x_2}(\calF_2)\to T^*_{x_1}(\calF_1)
    \]
    is an isomorphism. Then $f$ is an isomorphism.
\end{prop}

Our first task then is to show both $\calX$ and $(\Lan\calX_{\rm EG})^{\#}$ admit a deformation theory (see Definition \ref{dfn-deformation}), which reduces the comparison of $\calX$ and $(\Lan\calX_{\rm EG})^{\#}$ to the comparison of their pro-cotangent complexes at classical points by Theorem \ref{Intro-main1} and Proposition \ref{GR-criterion}. We can further reduce the comparison to classical points of finite type by proving both $\calX$ and $(\Lan\calX_{\rm EG})^{\#}$ are locally almost of finite type (Lemma \ref{geometrically finite type} and \ref{laft-2}). However, pro-complexes are known to be technically difficult to deal with. For example, they do not satisfy the faithfully flat descent in general and one can not apply derived Nakayama lemma to them. This seems to make it impossible to reduce the comparison of pro-cotangent complexes at the points valued in finite type algebras to those valued in finite fields, which . Fortunately, we are able to prove the following result.

\begin{prop}[Proposition \ref{key-a} and \ref{key-b}]
 For any point $e:\Spec(R)\to \calX$ (resp.  $e:\Spec(R)\to (\Lan\calX_{\rm EG})^{\#}$) with $R$ being a classical finite type algebra over $\bZ/p^a$ for some positive integer $a$, the pro-cotangent complex $T^*_e(\calX)$ (resp. $T^*_e((\Lan\calX_{\rm EG})^{\#})$) is indeed a complex of $R$-modules.
\end{prop}

In particular, the cotangent complex $T^*_e(\calX)$ is the shifted  Herr complex associated to the adjoint \'etale $(\varphi,\Gamma)$-module corresponding to the point $e$. This justifies the obstruction theory considered in \cite{EG22}.

Now we are able to reduce the comparison to classical points valued in finite fields. The new difficulty is that we have to deal with two different left Kan extensions: one is the left Kan extension from discrete Artinian local rings to Artinian local animated rings, the other is the left Kan extension from discrete commutative rings to animated rings. The former one is concerned with \cite[Lemma 7.5]{GV18}. The latter is what we care about. In order to apply \cite[Theorem 1.1]{BIP23}, we have to introduce the work of \cite{Zhu20} on derived representations with general coefficients as a bridge. In the meantime, it is necessary to clarify the relation between derived representations and derived Laurent $F$-crystals. The whole proof of the classicalty then culminates in the technical Proposition \ref{X-F}.

\begin{rmk}
    We will also give an ad-hoc definition of derived \'etale $(\varphi,\Gamma)$-modules in the unramified case in Subsection \ref{3.4}, which might be easier to work with than Laurent $F$-crystals in practice.
\end{rmk}

\begin{rmk}
    It seems possible to generalise the main results to the case of perfect complexes, i.e. the derived stack of Laurent $F$-crystals of perfect complexes is classical when restricted to truncated animated rings. But we do not pursue this generality here.
\end{rmk}


In some sense, Theorem \ref{intro-main2} completes the picture that the moduli stacks of local Langlands parameters are all classical in the $\GL_d$-case. To further study the case of general groups, a common strategy is to use the Tannakian formalism. Note that we do not know whether Theorem \ref{intro-main2} will hold true for any group $G$ (although we expect it to be the case if $G$ is a reductive group). In the cases where the local Galois deformation ring is a local complete intersection, one still needs to use the nice geometric properties of the $G$-version of the Emerton--Gee stack to prove Theorem \ref{intro-main2}. We plan to investigate the derived Emerton--Gee stack for general groups in future work. Note that there are already some works in this direction, e.g. \cite{Lin23},\cite{Lee23}.

\subsection*{Convention}
\begin{enumerate}
    \item All the \'etale $\varphi$-modules and \'etale $(\varphi,\Gamma)$-modules are projective of rank $d$ so we will omit $d$ in the notations.
    \item  For any $\infty$-category $\calC$, we use $\calC^{\simeq}$ to denote its core, i.e. its largest Kan subcomplex and use ${\rm Ho}(\calC)$ to denote its homotopy category.
    \item For any simplicial (resp. cosimplicial) diagram, we will omit degeneracy (resp. codegeneracy) maps for simplicity.
\end{enumerate}

\subsection*{Acknowledgments}
It is a great pleasure to thank Toby Gee for numerous helpful discussions regarding this work, as well as his many comments and corrections. We would like to thank Zhouhang Mao for answering some questions. We also thank the referees for their careful reading and valuable feedbacks.

The author has received funding from the European Research Council (ERC) under the European Union's Horizon 2020 research and innovation programme (grant agreement No. 884596).



\section{$^{\rm cl}\calX\simeq \calX_{\rm EG}$}

Let $K$ be a finite extension of $\bQ_p$ with perfect residue field $k$ and ring of integers $\calO_K$. As we have mentioned in the introduction, Bhatt--Scholze and Wu have proved that there is an equivalence between the category of Laurent $F$-crystals on the absolute prismatic site $(\calO_K)_{\Prism}$ and the category of \'etale $(\varphi,\Gamma)$-modules. There are two directions to generalise this result. One is to consider the geometric family, i.e.  consider more general schemes instead of just $\calO_K$. This has already been done in \cite{BS23} and part of it has also been done in \cite{MW21}. 

The other direction is to consider the arithmetic family. This is what we will do in this paper. Note that the arithmetic families of \'etale $(\varphi,\Gamma)$-modules is just the Emerton--Gee stack $\calX_{\rm EG}$ in \cite{EG22}. We mainly want to extend the above equivalence to more general coefficients. More precisely, we will construct the derived stack $\calX$ of Laurent $F$-crystals and compare its underlying classical stack $^{\rm cl}\calX$ to $\calX_{\rm EG}$.


\subsection{The derived stack $\calX$}

Let $\text{Nilp}_{\bZ_p}$ denote the category of $p$-nilpotent commutative rings, i.e. algebras over $\bZ/p^a$ for some positive integer $a$,  and $\textbf{Nilp}_{\bZ_p}$ denote $p$-nilpotent animated rings, i.e. the animation of $\text{Nilp}_{\bZ_p}$. Let $\textbf{Ani}$ denote the $\infty$-category of anima (i.e. $\infty$-category of $\infty$-groupoids). For more details about animation and anima, we refer to \cite[Chapter 5]{KS}. 

We first recall the definitions of prestack and derived prestack.

\begin{dfn}[Prestack/stack]
    By a (classical) prestack, we mean a functor $\calF: \text{Nilp}_{\bZ_p}\to \textbf{Ani}$. If $\calF$ satisfies the fppf descent, we call it a (classical) stack. Let $^{\rm cl}\PreStk$ denote the $\infty$-category of prestacks.
\end{dfn}

\begin{rmk}
    The (pre)stack here is also called higher (pre)stack in the literatures.
\end{rmk}

\begin{dfn}[Derived prestack/stack]
By a derived prestack, we mean a functor $\calF: \textbf{Nilp}_{\bZ_p}\to \textbf{Ani}$. If $F$ satisfies the fppf descent, we call it a derived stack. We use $\PreStk$ to denote the $\infty$-category of derived prestacks. Let $\textbf{Nilp}_{\bZ_p}^{\leq n}$ denote the $\infty$-category of $n$-truncated animated rings, i.e. animated ring $R$ with $\pi_i(R)=0$ for all $i>n$. We can define the $\infty$-category $\PreStk^{\leq n}$ of $n$-truncated prestacks, i.e. functors $\calF:\textbf{Nilp}_{\bZ_p}^{\leq n}\to \textbf{Ani}$. In particular, $\PreStk^{\leq 0}={^{\rm cl}\PreStk}$.
\end{dfn}

For any derived prestack $\calF$, one can get its underlying classical prestack $^{\rm cl}\calF$ by precomposing the natural inclusion $\text{Nilp}_{\bZ_p}\to \textbf{Nilp}_{\bZ_p}$.

 Now for any animated ring $R\in \textbf{Nilp}_{\bZ_p}$, we can define an $\infty$-presheaf $\calO_{\Prism,R}[\frac{1}{\calI_{\Prism}}]$ on the absolute prismatic site $(\calO_K)_{\Prism}$ (cf. \cite[Remark 4.7]{BS22},\cite[Definition 2.3]{BS23}) by 
\[
\calO_{\Prism,R}[\frac{1}{\calI_{\Prism}}]((A,I)):=(A\otimes^{\bL}_{\bZ_p}R)^{\wedge}_{I}[\frac{1}{I}].
\]

The $\infty$-presheaf $\calO_{\Prism,R}[\frac{1}{\calI_{\Prism}}]$ is actually an $\infty$-sheaf by the following theorem due to Drinfeld and Mathew. 
\begin{thm}[Drinfeld--Mathew, {\cite[Theorem 5.8]{Mat22}}]\label{Thm-descent}
Let $R$ be a connective $E_{\infty}$-ring, and $I$ be a finitely generated ideal of $\pi_0(R)$. Then the following functors defined on the $\infty$-category of connective $E_{\infty}$-$R$-algebras
\begin{enumerate}
    \item $S\mapsto D_{\rm perf}({\rm Spec}(S^{\wedge}_I)-V(IS))$;
    \item $S\mapsto \Vect({\rm Spec}(S^{\wedge}_I)-V(IS))$
\end{enumerate}
  are sheaves for the $I$-completely flat topology.  
\end{thm}

To see $\calO_{\Prism,R}[\frac{1}{\calI_{\Prism}}]$ is a sheaf, let $(A,I)\to (B,I)$ be a $(p,I)$-completely faithfully flat map. Then we need to show
\[
(A\otimes^{\bL}_{\bZ_p}R)^{\wedge}_{I}[\frac{1}{I}]\simeq \varprojlim_n(B^n\otimes^{\bL}_{\bZ_p}R)^{\wedge}_{I}[\frac{1}{I}]
\]
where $B^n$ is the $(n+1)$-fold $(p,I)$-completed tensor product of $B$ over $A$. It suffices to prove the above equivalence by regarding all the terms as anima. Note that the mapping space $\Map((B^n\otimes^{\bL}_{\bZ_p}R)^{\wedge}_{I}[\frac{1}{I}],(B^n\otimes^{\bL}_{\bZ_p}R)^{\wedge}_{I}[\frac{1}{I}])$ in $\calD((B^n\otimes^{\bL}_{\bZ_p}R)^{\wedge}_{I}[\frac{1}{I}])$, the derived $\infty$-category of $(B^n\otimes^{\bL}_{\bZ_p}R)^{\wedge}_{I}[\frac{1}{I}]$-modules, is equivalent to $(B^n\otimes^{\bL}_{\bZ_p}R)^{\wedge}_{I}[\frac{1}{I}]$ itself. And by the above theorem of Drinfeld--Mathew, we have
\[
\Map((A\otimes^{\bL}_{\bZ_p}R)^{\wedge}_{I}[\frac{1}{I}],(A\otimes^{\bL}_{\bZ_p}R)^{\wedge}_{I}[\frac{1}{I}])\simeq \varprojlim_n\Map((B^n\otimes^{\bL}_{\bZ_p}R)^{\wedge}_{I}[\frac{1}{I}],(B^n\otimes^{\bL}_{\bZ_p}R)^{\wedge}_{I}[\frac{1}{I}])
\]
which then gives the desired equivalence.

Let $\Vect((\calO_K)_{\Prism},\calO_{\Prism,R}[\frac{1}{\calI_{\Prism}}])$ denote the $\infty$-category of vector bundles\footnote{Following \cite[Notation 2.1]{BS23}, by a vector bundle over a ringed topos $(\calT,\calO)$, we mean an $\calO$-module $\bM$ such that there exists a cover $\{U_i\}$ of $\calT$ and finite projective modules $P_i$ such that $\bM|_{U_i}\simeq P_i\otimes_{\calO(U_i)}\calO_{U_i}$. Note that this is different from the definition of finite locally free module over a ringed topos.} over $\calO_{\Prism,R}[\frac{1}{\calI_{\Prism}}]$. Then we have the following description by the $I$-completely faithfully flat descent.

\begin{prop}[{\cite[Proposition 2.8]{BS23}}]\label{Vector bundle}
    There is an equivalence of $\infty$-categories
    \[
    \varprojlim_{(A,I)\in (\calO_K)_{\Prism}}\Vect((A\otimes^{\bL}_{\bZ_p}R)^{\wedge}_{I}[\frac{1}{I}])\simeq \Vect((\calO_K)_{\Prism},\calO_{\Prism,R}[\frac{1}{\calI_{\Prism}}])
    \]
\end{prop}

Note that as $R$ is $p$-nilpotent, the Frobenius on the $\delta$-ring $A$ induces a Frobenius map $\varphi$ on $\calO_{\Prism,R}[\frac{1}{\calI_{\Prism}}]$. Then we can define Laurent $F$-crystals with coefficients in $R$ on $(\calO_K)_{\Prism}$ and the corresponding derived prestack of Laurent $F$-crystals.

Let us start with some discussions about derived \'etale $\varphi$-modules.

\begin{dfn}[Derived \'etale $\varphi$-modules]
    Let $A$ be an animated ring equipped with an endomorphism $\varphi:A\to A$. Let $\Vect(A)$ be the $\infty$-category of finite projective $A$-modules (cf. \cite[Proposition 2.5.1 and 2.5.3, Theorem 2.5.2]{Lur09}) and ${\rm Isom}(\Vect(A))$ be the full subcategory of ${\rm Fun}(\Delta^1,\Vect(A))$ spanned by the isomorphisms. Then we define the $\infty$-category $\Vect(A)^{\varphi=1}$ of \'etale $\varphi$-modules over $A$ to be the (homotopy)\footnote{As the right vertical map is an isofibration, the pullback is the same as homotopy pullback.} pullback 
    \begin{equation*}
        \xymatrix@=1cm{
        \Vect(A)^{\varphi=1}\ar[rrrr]\ar[d]&&&& {\rm Isom}(\Vect(A))\ar[d]^{(ev_0,ev_1)}\\
        \Vect(A)\ar[rrrr]^{(\varphi^*,id)}&&&& \Vect(A)\times \Vect(A),
        }
    \end{equation*}
    where the functor $\varphi^*$ is the base change functor $(-)\otimes_{A,\varphi}A$. By construction, $\Vect(A)^{\varphi=1}$ is the $\infty$-category of pairs $(M,\varphi_M)$ where $M$ is a finite projective $A$-module and $\varphi_M:\varphi^*M\to M$ is an isomorphism of $A$-modules. By \cite[Proposition 7.6.5.18]{Lur23}, $\Vect(A)^{\varphi=1}$ is also equivalent to the homotopy pullback
      \begin{equation*}
        \xymatrix@=1cm{
        \Vect(A)^{\varphi=1}\ar[rrrr]\ar[d]&&&& \Vect(A)\ar[d]^{(id,id)}\\
        \Vect(A)\ar[rrrr]^{(\varphi^*,id)}&&&& \Vect(A)\times \Vect(A),
        }
    \end{equation*}
    which by \cite[Proposition 7.6.5.20]{Lur23} is equivalent to the homotopy equalizer of 
    \begin{equation*}
        \xymatrix{
        \Vect(A)\ar@<.5ex>[r]^{\varphi^*}\ar@<-.5ex>[r]_{id} & \Vect(A),
        }
    \end{equation*}
 similar to the classical case.
\end{dfn}

\begin{cons}[Dual of \'etale $\varphi$-module]
    For any $(M,\varphi_M)\in \Vect(A)^{\varphi=1}$, we can construct its dual $(M^{\vee},\varphi_{M^{\vee}})$ as follows: $M^{\vee}:=\Hom_{\calD(A)}(M,A)$ and $\varphi_{M^{\vee}}$ is the homotopy inverse of the natural map \[
    \Hom_{\calD(A)}(M,A)\xrightarrow{\simeq} \Hom_{\calD(A)}(\varphi^*M,A)\simeq \varphi^*\Hom_{\calD(A)}(M,A)\]where $\calD(A)$ is the (stable) $\infty$-category of $A$-modules and the second isomorphism follows from \cite[Corollary 2.5.6]{Lur09}.
\end{cons}

\begin{dfn}[Frobenius invariants]\label{Frobenius invariants}
    For any $(M,\varphi_M)\in \Vect(A)^{\varphi=1}$, by abuse of notation, we also write $\varphi_M:M\to M$ as the composite $M=M\otimes_{A,id}A\to \varphi^*M\xrightarrow{\varphi_M}M$ and let $M^{\varphi_M=1}$ denote the equalizer of
  \begin{equation*}
        \xymatrix{
        M\ar@<.5ex>[r]^{\varphi_M}\ar@<-.5ex>[r]_{id} & M
        }
    \end{equation*}
in $\calD(\bZ)$.
\end{dfn}
 We have the following lemma concerning the mapping space between derived \'etale $\varphi$-modules.
\begin{lem}\label{mapping space}
    For any two objects $(M,\varphi_M)$ and $(N,\varphi_N)$ in $\Vect(A)^{\varphi=1}$, we have 
    \[{\rm Map}((M,\varphi_M),(N,\varphi_N))\simeq \tau_{\geq 0}(M^{\vee}\otimes N)^{\varphi_{M^{\vee}}\otimes \varphi_N=1}.\]
\end{lem}

\begin{proof}
    By the description of $\Vect(A)^{\varphi=1}$ as the homotopy equalizer, we get an equalizer diagram
    \begin{equation*}
        \xymatrix{
        \Map((M,\varphi_M),(N,\varphi_N))\ar[r]&\Map(M,N)\ar@<.5ex>[r]^{\varphi^*}\ar@<-.5ex>[r]_{id} & \Map(M,N).
        }
    \end{equation*}

    Note that $\Map(M,N)\simeq \tau_{\geq 0}\Hom(M,N)\footnote{Here $\Hom(M,N)$ is the internal Hom object in $\calD(A)$ whose connective cover, i.e. $\tau_{\geq 0}\Hom(M,N)$, gives the mapping space in $\calD(A)$, which is an $\infty$-groupoid.}\simeq \Hom(M,N)\simeq M^{\vee}\otimes N$. Then the map $\varphi^*:\Map(M,N)\to \Map(M,N)$ defined as
    \[
    \Map(M,N)\xrightarrow{f} \Map(\varphi^*M,\varphi^*N)\xrightarrow{g} \Map(M,N),
    \]
    where $f$ is induced by the pullback along $A\xrightarrow{\varphi}A$ and $g$ is induced by the equivalences $M\xrightarrow{\varphi^{-1}_M} \varphi^*M$ and $\varphi^*N\xrightarrow{\varphi_N} N$, is induced by
    \[
    M^{\vee}\otimes N\to (\varphi^*M^{\vee})\otimes (\varphi^*N)\simeq \varphi^*(M^{\vee}\otimes N)\to M^{\vee}\otimes N.
    \]
    Now this lemma follows from Definition \ref{Frobenius invariants}.
\end{proof}

Now we are ready to define Laurent $F$-crystals and the corresponding derived prestack.
\begin{dfn}[Laurent $F$-crystal]
For any animated ring $R\in \textbf{Nilp}_{\bZ_p}$, we define the $\infty$-category of Laurent $F$-crystals as 
\[
\Vect((\calO_K)_{\Prism},\calO_{\Prism,R}[\frac{1}{\calI_{\Prism}}])^{\varphi=1}:=\varprojlim_{(A,I)\in (\calO_K)_{\Prism}}\Vect((A\otimes^{\bL}_{\bZ_p}R)^{\wedge}_{I}[\frac{1}{I}])^{\varphi=1},
\]
where $\Vect((A\otimes^{\bL}_{\bZ_p}R)^{\wedge}_{I}[\frac{1}{I}])^{\varphi=1}$ denote the $\infty$-category of \'etale $\varphi$-modules over $(A\otimes^{\bL}_{\bZ_p}R)^{\wedge}_{I}[\frac{1}{I}]$.

    By Proposition \ref{Vector bundle}, a Laurent $F$-crystal is simply a pair $(\bM,\varphi_{\bM})$ where $\bM$ is a vector bundle over $\calO_{\Prism,R}[\frac{1}{\calI_{\Prism}}]$ and $\varphi_{\bM}:\varphi^*\bM\to \bM$ is an isomorphism of $\calO_{\Prism,R}[\frac{1}{\calI_{\Prism}}]$-modules.

   Let $\Vect((\calO_K)_{\Prism},\calO_{\Prism,R}[\frac{1}{\calI_{\Prism}}])^{\varphi=1,\simeq}$ (resp. $\Vect((A\otimes^{\bL}_{\bZ_p}R)^{\wedge}_{I}[\frac{1}{I}])^{\varphi=1,\simeq}$) denote the core of $\Vect((\calO_K)_{\Prism},\calO_{\Prism,R}[\frac{1}{\calI_{\Prism}}])^{\varphi=1}$ (resp. $\Vect((A\otimes^{\bL}_{\bZ_p}R)^{\wedge}_{I}[\frac{1}{I}])^{\varphi=1}$, i.e. the largest Kan complex contained. As taking core commutes with limit, we then get
\[
\Vect((\calO_K)_{\Prism},\calO_{\Prism,R}[\frac{1}{\calI_{\Prism}}])^{\varphi=1,\simeq}:=\varprojlim_{(A,I)\in (\calO_K)_{\Prism}}\Vect((A\otimes^{\bL}_{\bZ_p}R)^{\wedge}_{I}[\frac{1}{I}])^{\varphi=1,\simeq}.
\]

\end{dfn}

\begin{dfn}
The derived prestack $ \calX_{\calO_K}$ of Laurent $F$-crystals
is defined as 
\[
\calX_{K}(R):=\Vect((\calO_K)_{\Prism},\calO_{\Prism,R}[\frac{1}{\calI}])^{{\varphi}=1,\simeq}
\]
for any $R\in \textbf{Nilp}_{\bZ_p}$. As $K$ is fixed throughout this paper, we simply write $\calX$ instead of $ \calX_{K}$.
\end{dfn}

Again, by Theorem \ref{Thm-descent} due to Drinfeld and Mathew, we see that $\calX$ is in fact a derived stack. Before proving this, we recall the definition of the Breuil--Kisin prism.
\begin{dfn}[Breuil--Kisin prism]
Let $\pi$ be a fixed uniformiser of $\calO_K$ and $E$ is the Eisenstein polynomial of $\pi$. Let $\{\pi_{n}\}_n$ be a compatible system of $p$-power roots of $\pi$. We call $(\frakS=W(k)[[u]],(E))\in (\calO_K)_{\Prism}$ the Breuil--Kisin prism and its perfection $(A_{\infty}:=W(\calO^{\flat}_{\hat K_{\infty}}),(E))$ the perfect Breuil--Kisin prism, where $\hat K_{\infty}$ is the $p$-adic completion of $\bigcup_n K(\pi_n)$.
\end{dfn}

\begin{prop}\label{stack}
The derived prestack $\calX$ is a derived stack for the fppf topology.
\end{prop}

\begin{proof}
    Let $R_1\to R_2$ be a faithfully flat map in $\textbf{Nilp}_{\bZ_p}$. Then we need to prove
    \begin{equation*}
    \xymatrix@=0.5cm{
\calX(R_1)\ar@<.0ex>[r]^-{\simeq} & \varprojlim (\calX(R_2)\ar@<-.3ex>[r]\ar@<.3ex>[r] & \calX(R_2^1)\ar@<-.6ex>[r]\ar@<-.0ex>[r]\ar@<.6ex>[r]&\cdots \calX(R_2^i)\cdots)&}
    \end{equation*}
where $\{R_2^i\}_{i\geq 0}$ is the \v Cech nerve associated with $R_1\to R_2$.

    Note that $(\frakS,(E))$ is a final cover of the topos ${\rm Shv}((\calO_K)_{\Prism})$. So we have 
     \begin{equation}\label{1-groupoid}
    \xymatrix@=0.5cm{
    \Vect((\calO_K)_{\Prism},\calO_{\Prism,R}[\frac{1}{\calI_{\Prism}}])^{\varphi=1,\simeq}\ar@<.0ex>[r]^-{\simeq} & \varprojlim (\Vect(\frakS_R[\frac{1}{E}])^{\varphi=1,\simeq}\ar@<-.3ex>[r]\ar@<.3ex>[r] & \Vect(\frakS^1_R[\frac{1}{E}])^{\varphi=1,\simeq}\ar@<-.6ex>[r]\ar@<-.0ex>[r]\ar@<.6ex>[r]&\cdots)}
  \end{equation}
    for any $R\in \textbf{Nilp}_{\bZ_p}$, where $(\frakS^{\bullet},(E))$ denotes the \v Cech nerve associated with $(\frakS,(E))$ and we simply write $\frakS_R^n=(\frakS^n\otimes^{\bL}_{\bZ_p}R)^{\wedge}_{E}$. 
For any $(\frakS^n,(E))$, we see that $\frakS^n\otimes^{\bL}_{\bZ_p}R_1\to \frakS^n\otimes^{\bL}_{\bZ_p}R_2$ is ($E$-completely) faithfully flat. So we can apply Theorem \ref{Thm-descent} and get
\[
\Vect(\frakS^n_{R_1}[\frac{1}{E}])^{\varphi=1,\simeq}\xrightarrow{\simeq}\varprojlim_i\Vect(\frakS^n_{R^i_2}[\frac{1}{E}])^{\varphi=1,\simeq}
\]
Now the proposition follows from the fact that limit commutes with limit.
\end{proof}

We conclude this subsection with a key lemma which will be used frequently.
\begin{lem}[{\cite[Proposition A.1]{HP22}}]\label{truncated-tot}
  Let $\calF:\Delta\to \textbf{Ani}$ be a cosimplicial anima such that all $\calF([i])$'s are $n$-truncated. Then there is an equivalence
  \[
  \varprojlim_{\Delta}\calF\xrightarrow{\simeq} \varprojlim_{\Delta_{\leq n+1}}\calF|_{\Delta_{\leq n+1}}
  \]
\end{lem}

\subsection{The underlying classical stack $^{\rm cl}\calX$}

In this subsection, we will study the underlying classical stack $^{\rm cl}\calX$ of $\calX$ and compare it with the Emerton--Gee stack $\calX_{\rm EG}$ constructed in \cite{EG22}.

Let's first briefly recall the construction of the Emerton--Gee stack and some of its basic properties. We fix a compatible system $\{\zeta_{p^n}\}_n$ of primitive $p$-power roots of unity in $\bar K$ and write $K_{\cyc}:=\bigcup_n K(\zeta_{p^n})$. Let $k_{\infty}$ denote the residue field of $K(\zeta_{p^{\infty}})$. Write $A_K^+=W(k_{\infty})[[T_K]]$, where $T_K$ is a chosen lift in $A_K$ of the uniformiser of the field of norm of $K(\zeta_{p^{\infty}})$ and $A_K=A_K^+[\frac{1}{T_K}]^{\wedge}_p$. There is a $\varphi$-action on $A_K$ as well as a compatible $\Gamma:=\Gal(K(\zeta_{p^{\infty}})/K)$-action. The subring $A_K^+$ is not necessarily stable under the $\varphi$-action. If $K$ is absolutely unramified, then $A_K^+$ is $\varphi$-stable.

\begin{dfn}[\'Etale $(\varphi,\Gamma)$-modules]
    Let $R$ be a $p$-nilpotent ring. We say $M$ is an \'etale $(\varphi,\Gamma)$-module with $R$-coefficient if $M$ is a finite projective module over $A_{K,R}:=(A_K^+\otimes_{\bZ_p}R)^{\wedge}_{T_K}[\frac{1}{T_K}]$ equipped with compatible semilinear Frobenius action $\varphi_M:\varphi^*M\xrightarrow{\simeq}M$ and continuous $\Gamma$-action. Let ${\rm Mod}^{\varphi,\Gamma}(A_{K,R})$ denote the category of \'etale $(\varphi,\Gamma)$-modules with $R$-coefficient and ${\rm Mod}^{\varphi,\Gamma}(A_{K,R})^{\simeq}$ denote its underlying groupoid.
\end{dfn}

\begin{thm}[\cite{EG22}]
    Let $\calX_{\rm EG}$ be the classical prestack defined as $\calX_{\rm EG}(R):={\rm Mod}^{\varphi,\Gamma}(A_{K,R})^{\simeq}$. Then $\calX_{\rm EG}$ is a Noetherian formal algebraic stack. In particular, it is also an ind-algebraic stack locally of finite type.
\end{thm}

Now we turn to the derived stack $\calX$. Note that when $R$ is discrete, the sheaf $\calO_{\Prism,R}[\frac{1}{\calI_{\Prism}}]$ is not necessarily a discrete $\infty$-sheaf. But the $\infty$-groupoid  $\Vect((\calO_K)_{\Prism},\calO_{\Prism,R}[\frac{1}{\calI_{\Prism}}])^{\varphi=1,\simeq}$ is indeed a $1$-groupoid. In fact, in Equation \ref{1-groupoid}, the right hand side is an inverse limit of $1$-groupoids as each $\frakS^n_R[\frac{1}{E}]$ is discrete. So $^{\rm cl}\calX$ is an fppf stack valued in 1-groupoids in the classical sense.

 In order to compare $^{\rm cl}\calX$ and $\calX_{\rm EG}$, we need to construct a map between them in the first place. To achieve this, we recall some important prisms in $(\calO_K)_{\Prism}$.

\begin{dfn}[Cyclotomic prism]\label{cyl-prism}
Let $\hat K_{\cyc}^{\flat}$ be the tilt of the $p$-adic completion of $K(\zeta_{p^{\infty}})$ and define $\epsilon:=(1,\zeta_p,\zeta_{p^2},\cdots)\in \hat K_{\cyc}^{\flat}$. Write $q=[\epsilon]$ and $[p]_q:=\frac{q^p-1}{q-1}$. When $K$ is absolutely unramified, by \cite[Proposition 2.19]{Wu21}, $(\bar A_K^+,[p]_q)$\footnote{To avoid possible confusion, we use $\bar A_K^+$ to denote the ring $A_K^+$ defined in \cite[Section 2.1]{Wu21} and keep the notation $A_K^+$ for the ring appearing in \cite{EG22}.} is a prism in $(\calO_K)_{\Prism}$. We will call it the cyclotomic prism.  By \cite[Lemma 2.17]{Wu21}, its perfection is $(A_{\cyc}:=W(\calO^{\flat}_{\hat K_{\cyc}}),(\xi))$, which we call the perfect cyclotomic prism. When $K$ is ramified, there is no well-defined cyclotomic prism but the perfect cyclotomic prism $(A_{\cyc}:=W(\calO^{\flat}_{\hat K_{\cyc}}),(\xi))$ still makes sense.
\end{dfn}

Note that both the cyclotomic prism and the Breuil--Kisin prism are covers of the final object of the topos ${\rm Shv}((\calO_K)_{\Prism})$. Now we state the main result of this section.

\begin{thm}\label{main1}
There is a natural equivalence between $^{\rm cl}\calX$ and $\calX_{\rm EG}$.
\end{thm}

To prove this theorem, we will proceed in two steps. The first one is to compare \'etale $(\varphi,\Gamma)$-modules with  coefficients to Laurent $F$-crystals with coefficients on the absolute perfect prismatic site. There is a natural restriction map from $\Vect((\calO_K)_{\Prism},\calO_{\Prism,R}[\frac{1}{\calI_{\Prism}}])^{\varphi=1}$ to $\Vect((\calO_K)_{\Prism}^{\perf},\calO_{\Prism,R}[\frac{1}{\calI_{\Prism}}])^{\varphi=1}$, which will then link $\Vect((\calO_K)_{\Prism},\calO_{\Prism,R}[\frac{1}{\calI_{\Prism}}])^{\varphi=1}$ to ${\rm Mod}^{\varphi,\Gamma}(A_{K,R})$. But this step can only provide comparison when the coefficient rings are finite type $\bZ_p$-algebras. This is closely related to the ``non-Noetherian" nature of the perfect absolute prismatic site $(\calO_K)_{\Prism}^{\perf}$. To finally prove Theorem \ref{main1}, we will need the second step: prove limit-preserving property of $^{\rm cl}\calX$. In a more geometric language, we will prove $^{\rm cl}\calX$ is locally of finite presentation.

We begin with how Laurent $F$-crystals are related to ``non-Noetherian" \'etale $(\varphi,\Gamma)$-modules.

\begin{prop}\label{perfect}
For any  $\bZ/p^a$-algebra $R$, there is an equivalence of categories
\[
\Vect((\calO_K)^{\perf}_{\Prism},\calO_{\Prism,R}[\frac{1}{\calI_{\Prism}}])^{\varphi=1}\simeq {\rm Mod}^{\varphi,\Gamma}(A_{\cyc,R}[\frac{1}{\xi}])
\]
where $A_{\cyc,R}:=(A_{\cyc}\otimes_{\bZ_p}R)^{\wedge}_{\xi}$.

\end{prop}

\begin{proof}
Since $(A_{\cyc},(\xi))$ is a cover of the final object of the topos ${\rm Shv}((\calO_K)_{\Prism}^{\perf})$, we see that any Laurent $F$-crystal corresponds to a stratification $(\calM,\epsilon)$ with respect to the \v Cech nerve $(A_{\cyc}^{\bullet},(\xi))$ of $(A_{\cyc},(\xi))$ in $(\calO_K)^{\rm perf}_{\Prism}$, where $\calM$ is a finite projective $A_{\cyc,R}[\frac{1}{\xi}]$-module and $\epsilon$ is an isomorphism of $A^1_{\cyc,R}[\frac{1}{\xi}]$-modules: 
\[
\calM\otimes_{A_{\cyc,R}[\frac{1}{\xi}],p_0}A^1_{\cyc,R}[\frac{1}{\xi}]\cong \calM\otimes_{A_{\cyc,R}[\frac{1}{\xi}],p_1}A^1_{\cyc,R}[\frac{1}{\xi}]
\]
which satisfies the cocycle condition. Now this theorem follows from Lemma \ref{Continuous functions}.
\end{proof}

\begin{rmk}
    One can also replace ${\rm Mod}^{\varphi,\Gamma}(A_{\cyc,R}[\frac{1}{\xi}])$ by the category ${\rm Mod}^{\varphi,G_K}(A_{\inf,R}[\frac{1}{\xi}])$ of \'etale $(\varphi,G_K)$-modules over $A_{\inf,R}[\frac{1}{\xi}]$ in Theorem \ref{perfect}.
\end{rmk}

\begin{lem}\label{Continuous functions}
For any  $\bZ/p^a$-algebra $R$, there is an isomorphism of rings
\[
A^n_{\cyc,R}[\frac{1}{\xi}]\cong C(\Gamma^n, A_{\cyc,R}[\frac{1}{\xi}])
\]
for all $n\geq 0$. 
\end{lem}

\begin{proof}
Consider $(C(\Gamma^n,A_{\cyc}),(\xi))$ as a perfect prism in $(\calO_K)_{\Prism}^{\rm perf}$. Then there is a natural map \[
(A_{\cyc}^n,(\xi))\to (C(\Gamma^n,A_{\cyc}),(\xi))
\]
induced by the maps of prisms $\gamma_i:(A_{\cyc},(\xi))\to (C(\Gamma^n,A_{\cyc}),(\xi))$ defined as $(\gamma_i(x))(g)=({\rm pr}_i(g))(x)$, where ${\rm pr}_i$ is the $i$-th projection map $\Gamma^n\to \Gamma$ when $1\leq i\leq n$ and the map $\Gamma^n\to *$ when $i=0$.

    Note that by \cite[Lemma 5.3]{Wu21}, there is an isomorphism
    \[
    A^n_{\cyc}[\frac{1}{\xi}]/p^m\cong C(\Gamma^n,A_{\cyc}[\frac{1}{\xi}]/p^m)
    \]
    for any $m$.
    In particular the homotopy cofiber $K$ of the map $A^n_{\cyc}/p^m\to C(\Gamma^n,A_{\cyc})/p^m$ vanishes after inverting $\xi$. As both $A^n_{\cyc}/p^m$ and $C(\Gamma^n,A_{\cyc})/p^m$ are derived $\xi$-complete, so is the homotopy cofiber $K$. Then by \cite[Theorem 2.3]{Bha19}, we see that $K$ is killed by $\xi^s$ for some $s$.

    Now for any finite type $\bZ/p^a$-algebra $R$, the map 
    \[
    A^n_{\cyc}/p^a\otimes_{\bZ/p^a} R\to C(\Gamma^n, A_{\cyc}/p^a)\otimes_{\bZ/p^a} R
    \]
    becomes isomorphism after inverting $\xi$. We can consider the following exact sequence of derived $\xi$-complete $A_{\cyc}$-modules
    \[
    A^n_{\cyc}/p^a\widehat\otimes_{\bZ/p^a}^{\bL} R\to C(\Gamma^n, A_{\cyc}/p^a)\widehat\otimes_{\bZ/p^a}^{\bL} R\to K\widehat\otimes_{\bZ/p^a}^{\bL}R.
    \]

We claim that $K\widehat\otimes_{\bZ/p^a}^{\bL}R$ vanishes after inverting $\xi$. By \cite[\href{https://stacks.math.columbia.edu/tag/0CQE}{Tag 0CQE}]{Sta18} and \cite[\href{https://stacks.math.columbia.edu/tag/0923}{Tag 0923}]{Sta18}, it is easy to see this is true if $K\otimes^{\bL}_{\bZ/p^a}R$ is killed by some bounded power of $\xi$. To see the latter, first note that $H^0(K\otimes_{\bZ/p^a}^{\bL}R)=H^0(K)\otimes_{\bZ/p^a}R$ is killed by $\xi^s$. To investigate $H^{-1}(K\otimes_{\bZ/p^a}^{\bL}R)$, we examine the homology Künneth spectral sequence
\[
E^2_{p,q}={\rm Tor}^{\bZ/p^a}_p(H_q(K),R)\Longrightarrow H_{p+q}(K\otimes_{\bZ/p^a}^{\bL}R).
\]
Note that both $E^2_{0,1}=H_1(K)\otimes_{\bZ/p^a}R$ and $E^2_{1,0}={\rm Tor}^{\bZ/p^a}_1(H_0(K),R)$ are killed by $\xi^s$. So both $E^{\infty}_{0,1}$ and $E^{\infty}_{1,0}$ are killed by $\xi^s$. This implies that $H^{-1}(K\otimes_{\bZ/p^a}^{\bL}R)=H_1(K\otimes_{\bZ/p^a}^{\bL}R)$ is killed by $\xi^{2s}$. Hence $K\otimes_{\bZ/p^a}^{\bL}R$ is killed by $\xi^{2s}$.

We finally get
    \[
    A^n_{\cyc,R}[\frac{1}{\xi}]\cong (C(\Gamma^n,A_{\cyc}/p^a)\widehat\otimes R)[\frac{1}{\xi}].
    \]

    It remains to prove 
    \[
    C(\Gamma^n,A_{\cyc}/p^a)\widehat\otimes R\cong C(\Gamma^n,A_{\cyc}/p^a\widehat\otimes R).
    \]
    This follows from $A_{\cyc}/p^a$ is $\xi$-adic complete.
    
\end{proof}

Before moving on, we give a lemma that will be used frequently.

\begin{lem}\label{tot}
    Let $f:\calC\to\calD$ be a morphism between two cosimplicial $\infty$-categories. Suppose $f_i:\calC_{[i]}\to \calD_{[i]}$ is fully faithful for all $[i]$'s and $f_0:\calC_{[0]}\to \calD_{[0]}$ is an equivalence. Then $f:{\rm Tot}(\calC)\to {\rm Tot}(\calD)$ is an equivalence.
\end{lem}

\begin{proof}
   As all $f_i$'s are fully faithful, we know that $f$ must be fully faithful. Moreover since $f_0$ is an equivalence and all $f_i$'s are fully faithful, we can construct a map from ${\rm Tot}(\calD)$ to $\calC$ using $f_0^{-1}$ which induces a map $g:{\rm Tot}(\calD)\to {\rm Tot}(\calC)$. By the full faithfulness of all $f_i$'s again, we see that $f\circ g\simeq id:{\rm Tot}(\calD)\to {\rm Tot}(\calD)$. In particular, this implies that $f$ is essentially surjective. So we get $f$ is an equivalence.
\end{proof}

Now we need to compare Laurent $F$-crystals on the absolute prismatic site and those on the absolute perfect prismatic site. 

\begin{prop}\label{finitetype}
    Let $R$ be a finite type $\bZ/p^a$-algebra. There is a natural equivalence of categories
    \[
    \Vect((\calO_K)_{\Prism},\calO_{\Prism,R}[\frac{1}{\calI_{\Prism}}])^{\varphi=1}\xrightarrow{\simeq}\Vect((\calO_K)^{\perf}_{\Prism},\calO_{\Prism,R}[\frac{1}{\calI_{\Prism}}])^{\varphi=1}.
    \]
\end{prop}

In order to prove Proposition \ref{finitetype}, we need the following lemma due to Bhatt--Scholze.

\begin{lem}[{\cite[Lemma 9.2]{BS22}}]\label{F-invariant}
    Let $B$ be a $\bZ_p$-algebra with a Frobenius lift $\varphi$. Suppose there is an element $t\in B$ such that $\varphi(t)\in t^2B$. Let $\calD(B[F])$ be the $\infty$-category of pairs $(M,(\varphi)$ where $M\in \calD(B)$ and $\varphi:M\to \varphi_*M$ is a map. Let $\calD_{\rm comp}(B[F])$ be the full subcategory of $\calD(B[F])$ consisting of pairs $(M,\varphi)$ with $M$ derived $t$-adic complete. Then the functors $\calD_{\rm comp}(B[F]) \to\calD(\bZ_p)$ given by $(M,\varphi)\mapsto M^{\varphi=1}$ and $(M,\varphi)\mapsto (M[\frac{1}{t}])^{\varphi=1}$ commute with colimits.
\end{lem}
\begin{proof}
   The proof is the same as that of \cite[Lemma 9.2]{BS22}. For readers' convenience, we repeat it here. All colimits in this proof refer to the colimits in $\calD(B)$.

    Let $(M_i,\varphi_i)$ be a diagram in $\calD_{\rm comp}(B[F])$. As the fiber of the map $\varinjlim M_i\to \widehat{\varinjlim M_i}$ is uniquely $t$-divisible, we just need to prove the statement for the functor $(M,\varphi)\mapsto M^{\varphi=1}$. Next consider the fiber $\bar N$ of the map $N\to N/t$ for any $N\in \calD_{\rm comp}(B[F])$. Note that $\bar N$ is derived $t$-complete. As $\varphi(t)\in t^2B$, $\bar N$ and $N/t$ both have induced $\varphi$-actions. Moreover as the $\varphi$-action on $\bar N$ is topologically nilpotent, we have $\bar N^{\varphi=1}=0$. This implies $N^{\varphi=1}\xrightarrow{\simeq}(N/t)^{\varphi=1}$. Since both functors in the composition $\calD_{\rm comp}(B[F])\xrightarrow{N\to N/t}\calD(B[F])\xrightarrow{(-)^{\varphi=1}} \calD(\bZ_p)$ commute with colimits, the lemma follows. 

\end{proof}

\begin{proof}[Proof of Proposition \ref{finitetype}]

By using the Breuil--Kisin prism $(\frakS,(E))$ and its perfection $(A_{\infty},(E))$, we can write both categories as a 2-limit of categories 
      \begin{equation*}
    \xymatrix@=0.5cm{
    \Vect((\calO_K)_{\Prism},\calO_{\Prism,R}[\frac{1}{\calI_{\Prism}}])^{\varphi=1}\ar@<.0ex>[r]^-{\simeq} & \varprojlim \Vect(\frakS_R[\frac{1}{E}])^{\varphi=1}\ar@<-.3ex>[r]\ar@<.3ex>[r] & \Vect(\frakS^1_R[\frac{1}{E}])^{\varphi=1}\ar@<-.6ex>[r]\ar@<.0ex>[r]\ar@<.6ex>[r] &\Vect(\frakS^2_R[\frac{1}{E}])^{\varphi=1}}
  \end{equation*}
    and
       \begin{equation*}
    \xymatrix@=0.5cm{
    \Vect((\calO_K)^{\rm perf}_{\Prism},\calO_{\Prism,R}[\frac{1}{\calI_{\Prism}}])^{\varphi=1}\ar@<.0ex>[r]^-{\simeq} & \varprojlim \Vect(A_{\infty,R}[\frac{1}{E}])^{\varphi=1}\ar@<-.3ex>[r]\ar@<.3ex>[r] & \Vect(A^1_{\infty,R}[\frac{1}{E}])^{\varphi=1}\ar@<-.6ex>[r]\ar@<.0ex>[r]\ar@<.6ex>[r] &\Vect(A^2_{\infty,R}[\frac{1}{E}])^{\varphi=1}}
  \end{equation*}
    where $(\frakS^{\bullet},(E))$ (resp. $(A_{\infty}^{\bullet},(E))$) is the \v Cech nerve associated with $(\frakS,(E))$ in $(\calO_K)_{\Prism}$ (resp. $(A_{\infty},(E))$ in $(\calO_K)_{\Prism}^{\rm perf}$) and $\frakS^{\bullet}_R:=(\frakS^{\bullet}\otimes_{\bZ_p}R)^{\wedge}_{E}$ (resp. $A_{\infty,R}^{\bullet}:=(A_{\infty}\otimes_{\bZ_p}R)^{\wedge}_{E}$). As $R$ is $p$-nilpotent, we can replace all $E$ by $u$.

    By \cite[Proposition 2.6.12]{EG22}, the natural functor ${\rm Res}^0:\Vect(\frakS_{R}[\frac{1}{u}])^{\varphi=1}\to \Vect(A_{\infty,R}[\frac{1}{u}])^{\varphi=1}$ is an equivalence. By Lemma \ref{F-invariant}, the functor ${\rm Res}^i:\Vect(\frakS^i_{R}[\frac{1}{u}])^{\varphi=1}\to \Vect(A^i_{\infty,R}[\frac{1}{u}])^{\varphi=1}$ is fully faithful for all $i$. In fact, we claim that $M^{\varphi=1}\cong (M\otimes_{\frakS_R^i[\frac{1}{u}]}A_{\infty,R}^i[\frac{1}{u}])^{\varphi=1}$ for any $M\in\Vect(\frakS^i_{R}[\frac{1}{u}])^{\varphi=1} $. Then the full faithfulness follows from Lemma \ref{mapping space}. Combining the equivalence of ${\rm Res}^0$ and the full faithfulness of ${\rm Res}^i$ for all $i$, we can see the restriction functor ${\rm Res}: \Vect((\calO_K)_{\Prism},\calO_{\Prism,R}[\frac{1}{\calI_{\Prism}}])^{\varphi=1}\xrightarrow{\simeq}\Vect((\calO_K)^{\perf}_{\Prism},\calO_{\Prism,R}[\frac{1}{\calI_{\Prism}}])^{\varphi=1}$ is indeed an equivalence by Lemma \ref{tot}. Or an easy way to see this is that every stratification with respect to $A_{\infty,R}^{\bullet}[\frac{1}{E}]$ can now be descended to a stratification with respect to $\frakS_{R}^{\bullet}[\frac{1}{E}]$.

    Now we prove the claim. By \cite[Lemma 5.2.14]{EG21}, for any $(M,\varphi_M)\in\Vect(\frakS^i_{R}[\frac{1}{u}])^{\varphi=1}$, there exists $(N,\varphi_N)\in\Vect(\frakS^i_{R}[\frac{1}{u}])^{\varphi=1}$ such that $(M\oplus N, \varphi_M\oplus \varphi_N)\in\Vect(\frakS^i_{R}[\frac{1}{u}])^{\varphi=1}$ is a finite free \'etale $\varphi$-module. So without loss of generality, we may assume $M$ is finite free. In this case, we can always find a lattice $(M_0,\varphi)$ over $\frakS_R^i$ inside $M$.

    Now we want to apply Lemma \ref{F-invariant} to the pair $(\frakS,u)$ and $M_0$. Note that $\varphi(u)=u^p\in \frakS$. So now what we need is to prove 
    \[
    A_{\infty,R}^{i}\simeq (\varinjlim_{\varphi}\frakS^{i}_{R})^{\wedge}_{u,\rm derived}
    \]
    where the completion on the right hand side is the derived $u$-adic completion. Note that $(A_{\infty}^i,(E))$ is the perfection of $(\frakS^i,(E))$. By \cite[Lemma 3.9]{BS22}, we see that 
    \[
A_{\infty}^i\simeq(\varinjlim_{\varphi}\frakS^i)^{\wedge}_{(p,E)}\simeq(\varinjlim_{\varphi}\frakS^i)^{\wedge}_{(p,E),\rm derived}.
    \]
    Then by derived Nakayama lemma, we have
    \[
    A_{\infty}^i/p^a\simeq (\varinjlim_{\varphi}\frakS^i/p^a)^{\wedge}_{E,\rm derived}\simeq(\varinjlim_{\varphi}\frakS^i/p^a)^{\wedge}_{u,\rm derived}
    \]
which implies
\[
A_{\infty,R}^i=(A^i_{\infty}/p^a\otimes_{\bZ_p/p^a}R)^{\wedge}_{u}=(A^i_{\infty}/p^a\otimes^{\bL}_{\bZ_p/p^a}R)^{\wedge}_{u,\rm derived}\simeq ((\varinjlim_{\varphi}\frakS^i/p^a)\otimes_{\bZ_p/p^a}^{\bL}R)^{\wedge}_{u,\rm derived}\simeq (\varinjlim_{\varphi}\frakS^{i}_{R})^{\wedge}_{u,\rm derived}.
\]

Note that $\varphi^*M$ is equipped with a natural Frobenius map such that $\varphi_M:\varphi^*M\to M$ induces an isomorphism of \'etale $\varphi$-modules. In particular, we have $\varphi_M:(\varphi^*M)^{\varphi=1}\simeq M^{\varphi=1}$. On the other hand, the composite of the Frobenius equivariant maps $M\to \varphi^*M$ and $\varphi_M:\varphi^*M\to M$ is exactly the Frobenius semilinear map $\varphi_M:M\to M$. As $\varphi_M=id: M^{\varphi=1}\to M^{\varphi=1}$, we see that $M^{\varphi=1}\xrightarrow{\simeq}(\varphi^*M)^{\varphi=1}$. The same argument then shows that the natural map $(\varphi^n)^*M\to (\varphi^{n+1})^*M$ induces $((\varphi^n)^*M)^{\varphi=1}\simeq ((\varphi^{n+1})^*M)^{\varphi=1}$. This implies $M^{\varphi=1}\simeq \varinjlim_n ((\varphi^n)^*M)^{\varphi=1}$.

By Lemma \ref{F-invariant}, we finally get
\[
M^{\varphi=1}\simeq \varinjlim_n ((\varphi^n)^*M)^{\varphi=1}\simeq((\varinjlim_{\varphi}\frakS_R^i)\otimes_{\frakS_R^i}M_0)^{\wedge}_{u,\rm derived}[\frac{1}{u}])^{\varphi=1}\simeq (A_{\infty,R}^i\otimes_{\frakS^i_R}M_0[\frac{1}{u}])^{\varphi=1}.
\]

\end{proof}

\begin{rmk}\label{cyclo}
In the unramified case, we may also use the cyclotomic prism $(\bar A_K^+,[p]_q)$ instead of the Breuil--Kisin prism in the proof of Proposition \ref{finitetype}. In fact, let $a$ be a positive integer, we just need to find a suitable element $t\in \bar A_K^+/p^a$ with $\varphi(t)\in t^2\bar A_K^+/p^a$ so that we can apply Lemma \ref{F-invariant}.  As $\varphi([p]_q)=[p]_q^p+p\delta([p]_q)\in \bar A_K^+$, we have $\varphi([p]_q^m)\equiv [p]_q^{pm} \mod p^a$ for some large enough integer $m$. So we can simply choose $t$ to be the image of $[p]_q^m$ in $\bar A_K^+/p^a$.
\end{rmk}

Recall the following important proposition in \cite{EG22}.

\begin{prop}[{\cite[Proposition 2.7.8]{EG22}}]\label{Frobenius descent}
    For any finite type $\bZ/p^a$-algebra $R$, there is an equivalence of categories
    \[
    {\rm Mod}^{\varphi,\Gamma}(A_{\cyc,R}[\frac{1}{\xi}])\simeq {\rm Mod}^{\varphi,\Gamma}(A_{K,R})
    \]
\end{prop}

We then have the immediate corollary.
\begin{cor}\label{step1}
    For any finite type $\bZ/p^a$-algebra $R$, there is an equivalence of categories
    \[
    \Vect((\calO_K)_{\Prism},\calO_{\Prism,R}[\frac{1}{\calI_{\Prism}}])^{\varphi=1}\simeq {\rm Mod}^{\varphi,\Gamma}(A_{K,R})
    \]
\end{cor}
\begin{proof}
    Just combine Proposition \ref{perfect}, Proposition \ref{finitetype} and Proposition \ref{Frobenius descent}.
\end{proof}

 In order to extend this equivalence to all $\bZ/p^a$-algebras, we recall that the Emerton--Gee stack $\calX_{\rm EG}$ is limit preserving.

\begin{lem}[{\cite[Lemma 3.2.19]{EG22}}]
    The Emerton--Gee stack $\calX_{\rm EG}$ is limit preserving. In particular, $\calX_{\rm EG}(R)=\varinjlim_i\calX_{\rm EG}(R_i)$ where $R$ is any $\bZ/p^a$-algebra and $\{R_i\}$ is the filtered colimit system of all the finite type sub-$\bZ/p^a$-algebras of $R$.
\end{lem}

So to prove Theorem \ref{main1}, we just need to prove the classical stack $\calX$ of Laurent $F$-crystals is also limit preserving.

\begin{lem}\label{step2}
    For any $\bZ/p^a$-algebra $R=\varinjlim_i R_i$ where $\{R_i\}$ is a filtered colimit system of finite type $\bZ/p^a$-algebras, we have $^{\rm cl}\calX(R)=\varinjlim_i^{\rm cl}\calX(R_i)$.
\end{lem}

\begin{proof}
By using the Breuil--Kisin prism $(\frakS,(E))$ as the cover of the final object of the topos ${\rm Shv}((\calO_K)_{\Prism})$, we can write
\begin{equation*}
    \xymatrix@=0.5cm{
    \Vect((\calO_K)_{\Prism},\calO_{\Prism,R_i}[\frac{1}{\calI_{\Prism}}])^{\varphi=1}\ar@<.0ex>[r]^-{\simeq} & \varprojlim \Vect(\frakS_{R_i}[\frac{1}{E}])^{\varphi=1}\ar@<-.3ex>[r]\ar@<.3ex>[r] & \Vect(\frakS^1_{R_i}[\frac{1}{E}])^{\varphi=1}\ar@<-.6ex>[r]\ar@<.0ex>[r]\ar@<.6ex>[r] &\Vect(\frakS^2_{R_i}[\frac{1}{E}])^{\varphi=1}}
  \end{equation*}
for any $i$. We can also take the core of each category as taking core commutes with limits.

As filtered colimit commutes with finite limits in $1$-Groupoids, we then have 
\begin{equation*}
    \xymatrix@=0.5cm{
    \varinjlim_i {^{\rm cl}\calX}(R_i)\ar@<.0ex>[r]^-{\simeq} & \varprojlim (\varinjlim_i\Vect(\frakS_{R_i}[\frac{1}{E}])^{\varphi=1,\simeq})\ar@<-.3ex>[r]\ar@<.3ex>[r] & (\varinjlim_i\Vect(\frakS_{R_i}^1[\frac{1}{E}])^{\varphi=1,\simeq})\ar@<-.6ex>[r]\ar@<.0ex>[r]\ar@<.6ex>[r] &(\varinjlim_i\Vect(\frakS_{R_i}^2[\frac{1}{E}])^{\varphi=1,\simeq})}
  \end{equation*}

    There is a natural map $\varinjlim_i {^{\rm cl}\calX}(R_i)\to {^{\rm cl}\calX(R)}$ by base change.
    By \cite[Corollary 3.1.5 and 3.2.6]{EG22}, the moduli stack of \'etale $\varphi$-modules is limit preserving, i.e. we have $\varinjlim_i(\Vect(\frakS_{R_i}[\frac{1}{E}])^{\varphi=1,\simeq})\simeq \Vect(\frakS_R[\frac{1}{E}])^{\varphi=1,\simeq}$. By Lemma \ref{mapping space} and Lemma \ref{F-invariant}, we get the base change functor $\varinjlim_i\Vect(\frakS^j_{R_i}[\frac{1}{E}])^{\varphi=1}\to \Vect(\frakS_R^j[\frac{1}{E}])^{\varphi=1}$ is fully faithful for $j=1,2$. This means we finally have an equivalence of categories $^{\rm cl}\calX(R)=\varinjlim_i{^{\rm cl}\calX}(R_i)$ by Lemma \ref{tot}.
\end{proof}

\begin{proof}[Proof of Theorem \ref{main1}]
This follows from Corollary \ref{step1} and Lemma \ref{step2}.
\end{proof}

\begin{rmk}

Perpendicular to the cyclotomic case, one can  also try to construct a moduli stack of \'etale $(\varphi,\tau)$-modules in the Kummer case. Then by using similar arguments, one can prove it is equivalent to the stack $^{\rm cl}\calX$ as least when restricted to the finite type $\bZ_p$-algebras. We are not sure if the stack of \'etale $(\varphi,\tau)$-modules is limit-preserving or not. But in practice, one can always modify it to be limit-preserving.

\end{rmk}

\section{$\calX^{\rm nil}\simeq (\Lan\calX_{\rm EG})^{\#,\rm nil}$}

In this section, we shift our focus to the derived stack $\calX$ of Laurent $F$-crystals. As we have mentioned in the introduction, we hope the derived stack $\calX$ is indeed ``classical", which roughly means it is determined by its underlying stack. This will then complete the picture that ``derived stacks of local Langlands parameters are all classical".

The possibility of the classicality of $\calX$ is suggested by the following theorem.

\begin{thm}[{\cite[Theorem 1.1]{BIP23}}]\label{BIP}
   Let $k_f$ be a finite field of characteristic $p$. Let $\bar\rho:G_K\to GL_d(k_f)$ be a residual representation. Then the framed local Galois deformation ring $R^{\square}_{\bar \rho}$ is a local complete intersection.
\end{thm}

As we will discuss in detail later, this is equivalent to that the derived local Galois deformation functor constructed by \cite{GV18} is classical. In other words, if $\calX$ is the correct derived object to be considered, it is already infinitesimal classical at any point corresponding to finite fields.

Now let's be more precise about the notion of calssicality. 
\begin{dfn}\label{dfn-classicality}
    Let $\calF: \textbf{Nilp}_{\bZ_p}\to \textbf{Ani}$ be a derived prestack. We say $\calF$ is classical if $\calF$ is equivalent to the left Kan extension ${\rm Lan}{^{\rm cl}\calF}$ of its underlying classical prestack $^{\rm cl}\calF$ along the natural inclusion $ {Nilp}_{\bZ_p}\hookrightarrow \textbf{Nilp}_{\bZ_p}$. Equivalently, $\calF$ is classical if $\calF$ can be written as a colimit of $h_{A_i}$ where $A_i\in {Nilp}_{\bZ_p}$. If $\calF$ is a derived stack for the \'etale topology, we say $\calF$ is classical as a stack if it is equivalent to the sheafification $({\rm Lan}{^{\rm cl}\calF})^{\#}$.
\end{dfn}

From now on, by left Kan extension, we always mean left Kan extension along the inclusion $ {Nilp}_{\bZ_p}\hookrightarrow \textbf{Nilp}_{\bZ_p}$ unless specified otherwise. Our main result in this section is the following theorem.
\begin{thm}\label{main2}
    The derived stack $\calX$ is equivalent to the (\'etale) sheafification $({\rm Lan}{^{\rm cl}\calX})^{\#}$ of the left Kan extension ${\rm Lan}{^{\rm cl}\calX}$ when restricted to $\textbf{Nilp}_{\bZ_p}^{<\infty}$. In other words, $\calX$ and $({\rm Lan}{\calX_{\rm EG}})^{\#}$ are equivalent after nilcompletion.
\end{thm}

To prove Theorem \ref{main2}, we need a global derived deformation theory instead of just infinitesimal deformation theory. This global theory is established in \cite[Chapter 1]{GR17}. Although \cite{GR17} always assumes the base field is of characteristic 0, the foundational part of the deformation theory also works\footnote{Over a field of characteristic $0$, cdga's, animated rings and $\bE_{\infty}$-rings are all equivalent. But this is not true when the base ring is $\bZ_p$. In particular, there are two different concepts of cotangent complex. In \cite{Lur17}, the cotangent complex for animated rings is called algebraic cotangent complex, which is the notion used in this paper, and the one for $\bE_{\infty}$-rings is called topological cotangent complex. One of key lemmas in \cite[Chapter 1]{GR17} is Lemma 5.4.3, which is proved by using \cite[Theorem 7.4.3.1]{Lur17} concerning topological cotangent complex. In the setting of animated rings, we need to use \cite[Proposition 25.3.6.1]{Lur18}, which is somewhat weaker than its $\bE_{\infty}$-ring analogue but is still sufficient to prove \cite[Lemma 5.4.3]{GR17}. } for animated rings (see \ref{appendix} for some necessary modifications).

\subsection{Derived deformation theory}

We first introduce the necessary foundational part of the derived deformation theory. The basic spirit is behind the following fact.

\begin{fact}\label{truncation}
    For any animated ring $R$, there is a homotopy pullback for any $n$
\begin{equation}\label{core}
\xymatrix@C=0.6cm{
\tau_{\leq n+1}{R}\ar[d]\ar[rr]&&\tau_{\leq n}(R)\ar[d]\\
\tau_{\leq n}(R)\ar[rr]&&\tau_{\leq n}(R)\oplus \pi_{n+1}(R)[n+2]
}
\end{equation}
where $\tau_{\leq n}(R)\oplus \pi_{n+1}(R)[n+2]$ is the split square-zero extension, which we will discuss later.
\end{fact}

Suppose $\calF: \textbf{Nilp}_{\bZ_p}\to \textbf{Ani}$ is a derived prestack. The study of the deformation theory of $\calF$ can be divided into 3 parts:
\begin{enumerate}
    \item Find the relationship between $\calF(R)$ and $\{\calF(\tau_{\leq n}(R))\}_n$.
    \item The homotopy pullback \ref{core} is an example of square-zero extensions, which is controlled by (pro-)cotangent complex\footnote{The cotangent complex usually appears in the study of deformation theory of algebraic objects. In our case, the derived prestack $\calF:\textbf{Nilp}_{\bZ_p}\to \textbf{Ani}$ is more like a derived ``formal" prestack. So instead, pro-cotangent complex will naturally appear in our context.} in a certain way. So we need to study the (pro-)cotangent complex of $\calF$, which can be regarded as some ``linear-algebra" data.
    \item If we have a good understanding of the cotangent complex (or square zero extensions), then we need to understand how the derived prestack $\calF$ acts on the square-zero extensions.
\end{enumerate}

\subsubsection{Nilcomplete}
Let us start with the first part. Recall that all animated rings are Postnikov complete, i.e. $R\simeq \varprojlim_n\tau_{\leq n}R$. We then have the following definition.
\begin{dfn}[Nilcomplete]
A derived prestack $\calF$ is nilcomplete\footnote{In \cite{GR17}, the notion of nilcomplete is called convergent.} if for any $R\in \textbf{Nilp}_{\bZ_p}$, the map
\[
\calF(R)\to \varprojlim_n \calF(\tau_{\leq n}R)
\]
is an isomorphism. Let $\PreStk^{\rm nil}$ denote the $\infty$-category of nilcomplete derived prestacks. The inclusion $\PreStk^{\rm nil}\hookrightarrow \PreStk$ admits a left adjoint given by the nilcompletion
\[
\calF^{\rm nil}(R):=\varprojlim_n \calF(\tau_{\leq n}R)
\]

\begin{exam}
    All derived affine schemes are nilcomplete. In fact, more is true: all derived Artin stacks are nilcomplete. Note that this is not always true for ind-derived Artin stacks\footnote{As in \cite{GR17}, nilcomplete is often included in the definition of ind-schemes.} as filtered colimit does not necessarily commute with inverse limit. 
\end{exam}
 
\begin{rmk}\label{remark-nil}
    The left Kan extensions of prestacks in ${\PreStk}^{\leq n}$ along the inclusion $\textbf{Nilp}_{\bZ_p}^{\leq n}\to \textbf{Nilp}_{\bZ_p}$ does not necessarily lie in $\PreStk^{\rm nil}$. So the left Kan extension $\Lan{^{\cl}\calX}$ and its sheafification $({\rm Lan}{^{\rm cl}\calX})^{\#}$ might not be nilcomplete. This is why we need to take nilcompletion in Theorem \ref{main2}.
    
\end{rmk}
    
\end{dfn}

\subsubsection{Pro-cotangent complex}

Suppose $A$ is a discrete animated ring and $M$ is a discrete $A$-module. Then there is a classical construction $A\oplus M$ which is an $A$-algebra whose underlying $A$-module structure is the natural one and the ring structure is given by the following formula
\[
(a_1,m_1)\cdot (a_2,m_2)=(a_1a_2,a_2m_1+a_1m_2).
\]
Then we can obtain
\[
\Hom_{\Mod(A)}(\Omega_{A/\bZ},M)\cong{\rm Der}_{\bZ}(A,M)\cong \Hom_{A}(A,A\oplus M).
\]

For general $A\in \textbf{Nilp}_{\bZ_p}$ and animated $A$-module $M$, there is a similar story. Recall that the $\infty$-category $\Mod(A)$ of $A$-modules is defined to be the $\infty$-category $\Mod(A^{\circ})$ where $A^{\circ}$ is the underlying connective $\bE_{\infty}$-algebra of $A$ through the natural functor $\Theta: \textbf{Nilp}_{\bZ_p}\to \textbf{CAlg}_{\bZ_p}^{\leq 0}$ (cf. \cite[Construction 25.1.2.1]{Lur18}). Then the $\infty$-category $\Mod(A)^{\leq 0}$ of animated $A$-modules is the connective part $\Mod(A^{\circ})^{\leq 0}$. Now one can still define $A\oplus M$ and the derivations ${\rm Der}_{\bZ}(A,M)$ of $A$ with values in $M$. There is an isomorphism
\[
\Hom_{A}(L_{A/\bZ},M)\simeq {\rm Der}_{\bZ}(A,M),
\]
where $L_{A/\bZ}$ is the cotangent complex of $A$ over $\bZ$. For more information, we refer to \cite[Chapter 25]{Lur18}.

Now we present a theory of pro-cotangent complex of a derived prestack $\calF:\textbf{Nilp}_{\bZ_p}\to \textbf{Ani}$. Let $x\in \calF(R)$ for some $R\in \textbf{Nilp}_{\bZ_p}$. Then we define a functor $\Mod(R)^{\leq 0}\to \textbf{Ani}$ by
\[
M\in \Mod(R)^{\leq 0}\mapsto {\rm Fib}_x(\calF(R\oplus M)\to \calF(R)).
\]

Let $M_1\to M_2$ be a map in $\Mod(R)^{\leq 0}$ such that $H^0(M_1)\to H^0(M_2)$ is a surjection. Consider the fiber product $M:=0\times_{M_2}M_1$ in $\Mod(R)$. Then $M$ is also in $\Mod(R)^{\leq 0}$. So we also get a fiber product $R\oplus M\simeq R\times_{(R\oplus M_2)}(R\oplus M_1)$ as the functor $R\oplus (-)$ commutes with limits.

\begin{dfn}[Pro-cotangent space at a point]
We say $\calF$ admits a pro-cotangent space at $x\in \calF(R)$ if the diagram
\begin{equation}\label{fiber product}
    \xymatrix@=0.6cm{
    \calF(R\oplus M)\ar[rr]\ar[d]&& \calF(R\oplus M_1)\ar[d]\\
    \calF(R)\ar[rr]&&\calF(R\oplus M_2)
    }
\end{equation}
    induced by $R\oplus M\simeq R\times_{(R\oplus M_2)}(R\oplus M_1)$ is a homotopy pullback diagram for all $M_1\to M_2$ as above.
\end{dfn}

More generally, if $\calF$ admits a pro-cotangent space at $x\in\calF(R)$, then we can define a new functor $\Mod(R)^{-}\to \cal S$ by 
\[
M\in \Mod(R)^{\leq k}\mapsto \Omega^i({\rm Fib}_x(\calF(R\oplus M)\to \calF(R)) 
\]
for any $i\geq k$, where $\Mod(R)^{-}$ is the $\infty$-category of almost connective modules, i.e. all those $M$ such that $M[k]$ is connective for some $k$. Then this new functor is actually exact as the diagram \ref{fiber product} is a homotopy pullback. In particular, this means the new functor is pro-corepresentable by $T^*_x(\calF)\in {\rm Pro}(\Mod(R)^{-})$. We call it the pro-cotangent space to $\calF$ at $x$.

\begin{dfn}[Pro-cotangent space]
    We say a derived prestack $\calF$ admits pro-cotangent spaces if it admits a pro-cotangent space at all points.
\end{dfn}

\begin{exam}
    All derived schemes admit pro-cotangent spaces, even cotangent spaces, which means the pro-object is indeed in $\Mod(R)^{-}$ for any $R$-point. All derived ind-schemes also admit pro-cotangent spaces.
\end{exam}

\begin{exam}\label{affine}
    If $\calF=h_{A}$, the prestack representable by some $A\in \textbf{Nilp}_{\bZ_p}$, then for any point $x\in \calF(R)$, the pro-cotangent space $T^*_x(\calF)$ is exactly $L_{A/\bZ}\otimes_{A,x}R$.
\end{exam}

Next we want to discuss the nilcompleteness condition of pro-cotangent spaces, which is closely related to the nilcompleteness of derived prestacks.

\begin{dfn}
For any $R\in \textbf{Nilp}_{\bZ_p}$, we let ${\rm Pro}(\Mod(R)^{-})^{\rm nil}$ denote the full sub-category of ${\rm Pro}(\Mod(R)^{-})$ spanned by the objects $\calH$ viewed as a functor $\Mod(R)^{-}\to \textbf{Ani}$, satisfying
\[
\calH(M)\xrightarrow{\simeq}\varprojlim_n\calH(\tau_{\leq n}M).\footnote{Here we still use the homological convention $\tau_{\leq n}$.}
\]
\end{dfn}

\begin{lem}[{\cite[Chapter 1, Lemma 3.3.3]{GR17}}]
    Let $\calF$ be a nilcomplete derived prestack. If it admits a pro-cotangent space at $x\in \calF(R)$, then $T^*_x(\calF)\in {\rm Pro}(\Mod(R)^{-})^{\rm nil}$.
\end{lem}
\begin{proof}
    This follows from the definition of nilcompleteness of derived prestack.
\end{proof}

For any $R\in \textbf{Nilp}_{\bZ_p}$, there is another full subcategory ${\rm Pro}(\Mod(R)^{-})_{\rm laft}$ of ${\rm Pro}(\Mod(R)^{-})$, which is very important to us.
\begin{dfn}\label{cotangent complex laft}
    ${\rm Pro}(\Mod(R)^{-})_{\rm laft}$ is the full subcategory of ${\rm Pro}(\Mod(R)^{-})$ spanned by objects $\calH$ satisfying
    \begin{enumerate}
        \item $\calH\in {\rm Pro}(\Mod(R)^{-})^{\rm nil}$;
        \item For every $m\geq 0$, the induced functor $\Mod(R)^{[-m,0]}\to \textbf{Ani}$ commutes with filtered colimits.
    \end{enumerate}
\end{dfn}
The full subcategory ${\rm Pro}(\Mod(R)^{-})_{\rm laft}$ will be closely related to the geometric property of the derived prestack. To see this, we now recall some basic properties concerning $\textbf{Nilp}_{\bZ_p}$.

\begin{dfn}[{\cite[Proposition 3.1.5]{Lur09}},{\cite[Chapter 2, Corollary 1.5.6]{GR19}}]
    Let $R\in \textbf{Nilp}_{\bZ_p}^{\leq n}$ for some $n$. Then we say $R$ is of finite type (or finite presentation) over $\bZ_p$ if it satisfies the following equivalent conditions.
    \begin{enumerate}
        \item The functor $\Hom(R,-)$ commutes with filtered colimit when restricted to $\textbf{Nilp}_{\bZ_p}^{\leq n}$, i.e. $R$ is a compact object in $\textbf{Nilp}_{\bZ_p}^{\leq n}$.
        \item $\pi_0(R)$ is a finite type $\bZ_p$-algebra and each $\pi_m(R)$ is a finite generated module over $\pi_0(R)$.
    \end{enumerate}
    Let $\textbf{Nilp}_{\bZ_p,\rm ft}^{\leq n}$ denote the full subcategory consisting of animated rings satisfying the above condition. Then we have $\textbf{Nilp}_{\bZ_p}^{\leq n}\simeq {\rm Ind}(\textbf{Nilp}_{\bZ_p,\rm ft}^{\leq n})$.
\end{dfn}

Now we can define derived prestacks locally of almost finite type (or locally almost of finite presentation).

\begin{dfn}
    Let $\calF\in \PreStk^{\leq n}$. We say $F$ is locally of finite type over $\bZ_p$ if $\calF$ is the left Kan extension along the inclusion $\textbf{Nilp}_{\bZ_p,\rm ft}^{\leq n}\hookrightarrow \textbf{Nilp}_{\bZ_p}^{\leq n}$, or equivalently, if  $\calF$ preserves filtered colimit. We say $\calF\in \PreStk$ is locally almost of finite type over $\bZ_p$ if its restriction to $\textbf{Nilp}_{\bZ_p}^{\leq n}$ is locally of finite type over $\bZ_p$ for every $n\geq 0$. Let $\PreStk_{\rm laft}$ denote the $\infty$-category of derived prestacks locally almost of finite type.
\end{dfn}

Now we have the following unsurprising result.

\begin{lem}[{\cite[Chapter 1, Lemma 3.5.2]{GR17}}]\label{laft}
    Let $\calF$ be a derived prestack locally almost of finite type. Assume $\calF$ admits a pro-cotangent space at $x\in \calF(R)$. Then the pro-cotangent space $T^*_x(\calF)$ is in ${\rm Pro}(\Mod(R)^{-})_{\rm laft}$.
\end{lem}

\begin{proof}
    This follows from the definition of $\PreStk_{\rm laft}$ and that the functor $(R,M)\mapsto R\oplus M$ commutes with filtered colimit.
\end{proof}

We conclude this subsection with the functoriality of pro-cotangent space. Let $R_1\to R_2$ be a map in $\textbf{Nilp}_{\bZ_p}$. Then we have a natural pullback functor
\[
f^*: \Mod(R_1)^{-}\to \Mod(R_2)^{-}
\]
which induces a natural functor
\[
\Pro(f)^*:\Pro(\Mod(R_1)^{-})\to \Pro(\Mod(R_1)^{-}).
\]

Now let $x_1\in \calF(R_1)$ be an $R_1$-point and $x_2\in \calF(R_2)$ be the image of the point $x_1$ in $\calF(R_2)$. So we can obtain a map
\[
T_{x_2}^*(\calF)\to \Pro(f^*)(T_{x_1}^*(\calF))
\]
\begin{dfn}[Pro-cotangent complex]
   We say $\calF$ admits a (pro)-cotangent complex if it admits (pro)-cotangent spaces and for any $R_1\to R_2, x_1,x_2$ as above, the map $T_{x_2}^*(\calF)\to \Pro(f^*)(T_{x_1}^*(\calF))$ is an isomorphism. 
\end{dfn}

\begin{exam}
    All the derived Artin stacks admit (pro)-cotangent complexes. For derived affine schemes, this can be easily seen from Example \ref{affine}.
\end{exam}

\subsubsection{Infinitesimally cohesive}
As we have seen in Fact \ref{truncation}, there is a homotopy pullback diagram  connecting $\tau_{\leq n+1}R$ and $\tau_{\leq n}R$ for $R\in\textbf{Nilp}_{\bZ_p}$ and each $n\geq 0$. This homotopy pullback diagram is an example of more general square-zero extensions.

\begin{dfn}[Square-zero extension]
Let $R\in \textbf{Nilp}_{\bZ_p}$ and $M$ is a connective $R$-module. A square-zero extension of $R$ by $M$ is a homotopy pullback diagram
\begin{equation*}
    \xymatrix@=0.6cm{
    \tilde R\ar[rr]\ar[d]&& R\ar[d]^{\rm triv}\\
    R\ar[rr]^{s}&&R\oplus M[1]
    }
\end{equation*}
where ${\rm triv}: R\to R\oplus M[1]$ is the trivial section of the natural projection $R\oplus M[1]\to R$ and $s$ is any section of $R\oplus M[1]\to R$. We simply write the square-zero extension as the quadruple $(\tilde R,R,M,s)$. 
\end{dfn}
By the definition of the cotangent complex $L_{R/\bZ}$, we see that square-zero extensions are classified by $\Hom(L_{R/\bZ},M[1])$.

\begin{dfn}[Infinitesimally cohesive]
    Let $\calF$ be a derived prestack. We say $\calF$ is infinitesimally cohesive if for all square-zero extensions $(\tilde R,R,M,s)$, the following diagram is a homotopy pullback
    \begin{equation*}
        \xymatrix@=0.6cm{
        \calF(\tilde R)\ar[rr]\ar[d]&&\calF(R)\ar[d]^{\rm triv}\\
        \calF(R)\ar[rr]^{s}&&\calF(R\oplus M[1]).
        }
    \end{equation*}
\end{dfn}

\subsection{Deformation theory of $\calX$}

After making all the necessary preparations, we are now at a good stage to study the derived stack $\calX$. The three properties about derived prestacks we discussed in previous subsections can be integrated into the following definition.
\begin{dfn}[Deformation theory]\label{dfn-deformation}
  Let $\calF$ be a derived prestack. We say $\calF$ admits a deformation theory (resp. corepresentable deformation theory) if the following conditions are satisfied:
  \begin{enumerate}
      \item it is nilcomplete;
      \item it admits pro-cotangent complex (resp. cotangent complex\footnote{A classical prestack, which admits a cotangent complex and is infinitesimal cohesive, is also called having an obstruction theory.});
      \item it is infinitesimal cohesive.
  \end{enumerate}
\end{dfn}

There is a simple way to verify if a derived prestack admits a deformation theory.

\begin{prop}[{\cite[Chapter 1, Proposition 7.2.5]{GR17}}]\label{criterion}
    Suppose $\calF$ is a nilcomplete derived prestack. If for any homotopy pullback diagram of truncated animated rings
    \begin{equation*}
        \xymatrix@=0.6cm{
        S\ar[rr]\ar[d]&&S_2\ar[d]\\
        S_1\ar[rr]&&S_0\\
        }
    \end{equation*}
    with $S_1\to S_0$ a nilpotent embedding, which means $\pi_0(S_1)\to \pi_0(S_0)$ is surjective and its kernel is a nilpotent ideal, the induced diagram
    \begin{equation*}
        \xymatrix@=0.6cm{
        \calF(S)\ar[rr]\ar[d]&&\calF(S_2)\ar[d]\\
        \calF(S_1)\ar[rr]&&\calF(S_0)
        }
    \end{equation*}
    is also a homotopy pullback diagram. Then $\calF$ admits a deformation theory.
\end{prop}
\begin{proof}
    As $\calF$ is nilcomplete, it is sufficient to check the existence of pro-cotangent complexes and infinitesimal cohesiveness for truncated animated rings and modules by \cite[Chapter 1, Lemma 3.3.4 and 6.1.3]{GR17}. So considering truncated animated rings is sufficient for the arguments in the proof of \cite[Chapter 1, Proposition 7.2.5]{GR17}.
\end{proof}

We also have its converse.

\begin{prop}[{\cite[Chapter 1, Proposition 7.2.2]{GR17}}]\label{converse-criterion}
   Suppose $\calF$ admits a deformation theory. Then for any homotopy pullback diagram $S\simeq S_1\times_{S_0}S_2$ with $S_1\to S_0$ a nilpotent embedding, the induced map $\calF(S)\to \calF(S_1)\times_{\calF(S_0)}\calF(S_2)$ is an equivalence.
\end{prop}

We will actually use this criterion to prove $\calX^{\rm nil}$ admits a deformation theory. 
\begin{prop}\label{admitdeformation}
    The derived prestack $\calX^{\rm nil}$ admits a deformation theory.
\end{prop}

In order to prove Proposition \ref{admitdeformation}, we recall the following theorem due to Toen--Vaqui\'e.

\begin{thm}[\cite{TV07}, \cite{Yay22}]\label{perf}
    The derived (pre)stack $\textbf{Perf}^{[a,b]}:\textbf{Nilp}_{\bZ_p}\to \textbf{Ani}$ sending $R\in \textbf{Nilp}_{\bZ_p}$ to ${\rm Perf}(R)^{[a,b],\simeq}$ admits a deformation theory and is locally almost of finite presentation for any $a\leq b\in \bZ$. In particular, the derived (pre)stack $\bigsqcup_n {\rm BGL}_n$ of vector bundles admits a deformation theory and is locally almost of finite presentation.
\end{thm}
\begin{rmk}
    In fact, we only need the Artin stack ${\rm BGL}_n$. But for the possible generalisation to perfect complexes, we choose to present the above general theorem.
\end{rmk}
\begin{proof}[Proof of Proposition \ref{admitdeformation}]
Recall that for any animated ring $A\in \textbf{Nilp}_{\bZ_p}$ equipped with a Frobenius action $\varphi$, we can define the $\infty$-groupoid of \'etale $\varphi$-modules over $A$ by the homotopy pullback diagram
\begin{equation}\label{dfn of etale}
     \xymatrix{
\Vect(A)^{\varphi=1}\ar@<.ex>[r]&\Vect(A)\ar@<.5ex>[r]^{\varphi^*}\ar@<-.5ex>[r]_{id} & \Vect(A),}
\end{equation}
Using the Breuil--Kisin prism $(\frakS,(E))$ and Lemma \ref{truncated-tot}, for any $R\in\textbf{Nilp}_{\bZ_p}^{\leq n}$, we have
\[
\Vect((\calO_K)_{\Prism},\calO_{\Prism,R}[\frac{1}{\calI_{\Prism}}])^{\varphi=1,\simeq}\simeq \varprojlim (\Vect(\frakS_R[\frac{1}{E}])^{\varphi=1,\simeq}\rightrightarrows\Vect(\frakS_R^1[\frac{1}{E}])^{\varphi=1,\simeq}\cdots\Vect(\frakS_R^{n+2}[\frac{1}{E}])^{\varphi=1,\simeq}).
\]
  Now we want to use Proposition \ref{criterion}. Given any homotopy fiber product $R\simeq R_1\times_{R_0}R_2$ with $R_1\to R_0$ a nilpotent embedding, we want to prove
\begin{equation*}
    \xymatrix@=0.6cm{
    \Vect(\frakS_R^i[\frac{1}{E}])^{\varphi=1,\simeq}\ar[rr]\ar[d]&&\Vect(\frakS_{R_2}^i[\frac{1}{E}])^{\varphi=1,\simeq}
    \ar[d]\\
    \Vect(\frakS_{R_1}^i[\frac{1}{E}])^{\varphi=1,\simeq}\ar[rr]&&\Vect(\frakS_{R_0}^i[\frac{1}{E}])^{\varphi=1,\simeq}
    }
\end{equation*}
 is also a homotopy pullback diagram for any $0\leq i\leq n+2$. By diagram \ref{dfn of etale}, Theorem \ref{perf} and Proposition \ref{converse-criterion}, we just need to prove (1)
 \begin{equation*}
    \xymatrix@=0.6cm{
    \frakS_R^i[\frac{1}{E}]\ar[rr]\ar[d]&& \frakS_{R_2}^i[\frac{1}{E}]\ar[d]\\
     \frakS_{R_1}^i[\frac{1}{E}]\ar[rr]&& \frakS_{R_0}^i[\frac{1}{E}]
    } 
 \end{equation*}
 is also a homotopy pullback diagram and (2) $\frakS_{R_1}^i[\frac{1}{E}]\to \frakS_{R_0}^i[\frac{1}{E}]$ is a nilpotent embedding. 

 We first prove the first statement. In fact, it suffices to prove the natural map \[
 \iota:\frakS^i\otimes R\to(\frakS^i\otimes R_1)\times_{(\frakS^i\otimes R_0)}(\frakS^i\otimes R_2)\]  is an isomorphism. As the forgetful functor
 \[
 \tilde\Theta:\textbf{Nilp}_{\bZ_p}\xrightarrow{\Theta}\textbf{CAlg}_{\bZ_p}^{\leq 0}\to\Mod(\bZ_p)^{\leq 0}
 \]
 is conservative and preserves limits (cf. \cite[Proposition 25.1.2.2]{Lur18}), it remains to prove $\tilde\Theta(\iota)$ is an isomorphism. 

 Let $\tilde R$ denote the fiber product $R_1\times_{R_0}R_1$ in $\Mod(\bZ_p)$. Then we have $\tilde\Theta(R)$ is isomorphic to $\tau_{\geq 0}\tilde R$, the connective cover of $\tilde R$. As $\frakS^i$ is flat over $\bZ_p$, we have $\frakS^i\otimes \tau_{\geq 0}\tilde R\simeq \tau_{\geq 0}(\frakS^i\otimes\tilde R)$. Note that fiber sequence in $\Mod(\bZ_p)$ is also a cofiber sequence. We then get 
 \begin{equation*}
     \xymatrix@=0.6cm{
     \frakS^i\otimes\tilde R\ar[rr]\ar[d]&& \frakS^i\otimes R_2\ar[d]\\
      \frakS^i\otimes R_1\ar[rr]&& \frakS^i\otimes R_0
     }
 \end{equation*}
 is a homotopy pullback diagram in $\Mod(\bZ_p)$. Now by using $\tau_{\geq 0}$ commutes with limits, we see that $\tilde\Theta(\iota)$ is an isomorphism. So the first statement is true.

 Next we prove the second statement, i.e. $\frakS_{R_1}^i[\frac{1}{E}]\to \frakS_{R_0}^i[\frac{1}{E}]$ is a nilpotent embedding. This means we need to prove the map $\pi_0(\frakS^i\widehat\otimes R_1))\to \pi_0(\frakS^i\widehat\otimes R_0))$ is surjective and it kernel is nilpotent. By Lemma \ref{pi-complete}, we see that $\pi_0(\frakS^i\widehat\otimes R_j))\cong \frakS^i\widehat\otimes\pi_0(R_j)$ for $j=0,1$, from which the second statement follows. In fact, let $K$ be the kernel of $\pi_0(R_1)\to \pi_0(R_0)$. Then we have a short exact sequence
 \[
 0\to \frakS^i\widehat\otimes K\to \frakS^i\widehat\otimes \pi_0(R_1)\to \frakS^i\widehat\otimes \pi_0(R_0)\to 0
 \]
 where $\frakS^i\widehat\otimes K$ is nilpotent.

\end{proof}
\begin{lem}\label{pi-complete}
    Let $A$ be a flat $\bZ/p^a$-algebra for some $a>0$. Let $d\in A$ be a non zero-divisor such that $A/d^n$ is flat over $\bZ/p^a$ for all $n$. Let $R$ be an animated ring over $\bZ/p^a$. Then for each $i\geq 0$, we have $\pi_i(A\widehat\otimes R)\cong A\widehat\otimes\pi_i(R)$, where the completion is derived $d$-adic compeletion. 
\end{lem}
\begin{proof}
    By \cite[\href{https://stacks.math.columbia.edu/tag/0CQE}{Tag 0CQE}]{Sta18}, we have a short exact sequence
 \[
 0\to R^1\varprojlim_n \pi_{i+1}(A/d^n\otimes R)\to \pi_i(A\widehat\otimes R)\to \varprojlim_n \pi_i(A/d^n\otimes R)\to 0.
 \]
   As $R^1\varprojlim_n \pi_{i+1}(A/d^n\otimes R)=R^1\varprojlim_n A/d^n\otimes \pi_{i+1}(R)=0$, we then have   $\pi_i(A\widehat\otimes R)\cong \varprojlim_n \pi_i(A/d^n\otimes R)=A\widehat\otimes \pi_i(R)$.
\end{proof}

 We can also prove $(\Lan{^{\rm cl}\calX})^{\#,\rm nil}$ admits a deformation theory.
 \begin{thm}
    The derived prestack $(\Lan\calX_{\rm EG})^{\#,\rm nil}$ admits a deformation theory.
 \end{thm}

 \begin{proof}
     Recall that $\calX_{\rm EG}$ is an ind-Artin stack. Write $\calX_{\rm EG}=\varinjlim_i\calX_i$. As left Kan extension preserves colimits, we have $\Lan\calX_{\rm EG}=\varinjlim_i\Lan\calX_i$. Then $(\Lan\calX_{\rm EG})^{\#}\simeq (\varinjlim_i(\Lan\calX_i)^{\#})^{\#}$. By \cite[Chapter 2, Proposition 4.4.3]{GR19}, we see that each $(\Lan\calX_i)^{\#}$ is a derived Artin stack. By \cite[Chapter 1, Corollary 4.3.4]{GR19}, $(\Lan\calX_i)^{\#}(R)$ is in $\textbf{Ani}^{<\infty}$ for any truncated animated ring $R$. For any \'etale cover $R_1\to R_2$, if $R_1\in \textbf{Ani}^{<n}$, then $R_2$ is also in $\textbf{Ani}^{<n}$ as $\pi_i(R_2)\cong\pi_0(R_2)\otimes_{\pi_0(R_1)}\pi_i(R_1)$. Then by Lemma \ref{truncated-tot} and the fact that filtered colimit commutes with finite limit, we see $(\varinjlim_i(\Lan\calX_i)^{\#})(R_1)\simeq \varprojlim(\varinjlim (\Lan\calX_i)^{\#})(R^{\bullet}_2)$ where the right hand side is the limit of the \v Cech nerve. In other words, this means $\varinjlim_i(\Lan\calX_i)$ already satisfies the \'etale sheaf property for \'etale coverings in $\textbf{Ani}^{<\infty}$. So $(\varinjlim_i(\Lan\calX_i))^{\#}$ coincides with $\varinjlim_i(\Lan\calX_i)^{\#}$ when restricted to truncated animated rings, which in particular satisfies the assumptions in Proposition \ref{criterion} except being nilcomplete. Then $(\varinjlim_i(\Lan\calX_i))^{\#,\nil}$ satisfies all the assumptions in Proposition \ref{criterion}. Hence $(\Lan\calX_{\rm EG})^{\#,\nil}$ admits a deformation theory.

 \end{proof}
\subsection{Classicality of $\calX$}

Now we have proved both $\calX^{\rm nil}$ and $(\Lan\calX_{\rm EG})^{\#,\rm nil}$ admit a deformation theory. The next goal is to compare them. The following proposition will be our key tool to compare two derived prestacks. 
\begin{prop}[{\cite[Chapter 1, Proposition 8.3.2]{GR17}}]\label{GR-criterion}
    Let $\calF_1\to \calF_2$ be a map of derived prestacks admitting a deformation theory. Suppose there exists a commutative diagram,
    \begin{equation*}
        \xymatrix@=0.6cm{
        & ^{\rm cl}\calF_0\ar[ld]^{g_1}\ar[rd]^{g_2}&\\
        \calF_1\ar[rr]^{f}&&\calF_2
        }
    \end{equation*}
where $g_1,g_2$ are nilpotent embeddings, and $^{\rm cl}\calF_{0}$ is a classical prestack. Suppose also that for any discrete commutative ring $R\in {Nilp}_{\bZ_p}$, and a map $x_0\in {^{\rm cl}\calF_{0}}(R)$ and $x_i=g_i\circ x_0$ where $i=1,2$, the induced map
    \[
    T^*_{x_2}(\calF_2)\to T^*_{x_1}(\calF_1)
    \]
    is an isomorphism. Then $f$ is an isomorphism.
\end{prop}

In our case, we will take $^{\rm cl}\calF_0=\calX_{\rm EG}$, $\calF_1=(\Lan\calX_{\rm EG})^{\#,\rm nil}$ and $\calF_2=\calX^{\rm nil}$. This means we need to compare $T^*_x((\Lan\calX_{\rm EG})^{\#,\rm nil})$ and $T^*_x(\calX^{\rm nil})$ for any discrete $R\in {\rm Nilp}_{\bZ_p}$. It is not easy to compare pro-cotangent spaces for arbitrary classical rings. But Lemma \ref{laft} says that the situation becomes better if the involved derived prestacks are locally almost of finite type. This is what we will prove.

\begin{prop}\label{geometrically finite type}
    The derived prestack $\calX^{\rm nil}$ is locally almost of finite type.
\end{prop}

\begin{proof}
    By definition, this means we need to prove $\calX^{\rm nil}$ preserves filtered colimit when restricted to $\textbf{Nilp}_{\bZ_p}^{\leq n}$ for each $n\geq 0$.    Write $R=\varinjlim_iR_i\in\textbf{Nilp}_{\bZ_p}^{\leq n} $ as a filtered colimit. We then need to prove $\varinjlim_i\calX(R_i)\simeq\calX(R)$.

    Again by using the Breuil--Kisin prism $(\frakS,(E)$ and Lemma \ref{truncated-tot}, we have
    \[
    \Vect((\calO_K)_{\Prism},\calO_{\Prism,R_i}[\frac{1}{\calI_{\Prism}}])^{\varphi=1,\simeq}\simeq \varprojlim \Vect(\frakS_{R_i}[\frac{1}{E}])^{\varphi=1,\simeq}\rightrightarrows\Vect(\frakS_{R_i}^1[\frac{1}{E}])^{\varphi=1,\simeq}\cdots\Vect(\frakS_{R_i}^{n+2}[\frac{1}{E}])^{\varphi=1,\simeq}.
    \]
  
As filtered colimit commutes with finite limit in $\textbf{Ani}$, we have 
\[
\varinjlim_i\calX(R_i)\simeq \varprojlim (\varinjlim
_i\Vect(\frakS_{R_i}[\frac{1}{E}])^{\varphi=1,\simeq})\rightrightarrows(\varinjlim_i\Vect(\frakS^1_{R_i}[\frac{1}{E}])^{\varphi=1,\simeq})\cdots (\varinjlim_i\Vect(\frakS^{n+2}_{R_i}[\frac{1}{E}])^{\varphi=1,\simeq}).
\]
We claim that the functor $\varinjlim_i\Vect(\frakS^{j}_{R_i}[\frac{1}{E}])^{\varphi=1}\to \Vect(\frakS^{j}_{R}[\frac{1}{E}])^{\varphi=1}$ is fully faithful for every $0\leq j\leq n+2$. In fact, let $M_s\in \Vect(\frakS^{j}_{R_s}[\frac{1}{E}])^{\varphi=1}$ and $M_t\in \Vect(\frakS^{j}_{R_t}[\frac{1}{E}])^{\varphi=1}$. Then the mapping space between $M_s$ and $M_t$ in $\varinjlim_i\Vect(\frakS^{j}_{R_i}[\frac{1}{E}])^{\varphi=1}$ is $\varinjlim_{s,t\to k}{\rm Map}_k(M_{s,k},M_{t,k})$ where $M_{s,k}=M_s\otimes_{R_s}R_k$, $M_{s,k}=M_s\otimes_{R_s}R_k$ and ${\rm Map}_k(-,-)$ means the mapping space in $\Vect(\frakS^{j}_{R_k}[\frac{1}{E}])^{\varphi=1}$. Then we need to prove $\varinjlim_{s,t\to k}{\rm Map}_k(M_{s,k},M_{t,k})\simeq {\rm Map}(M_{s,R},M_{t,R})$ where $M_{s,R}=M_s\otimes_{R_s}R$, $M_{t,R}=M_t\otimes_{R_t}R$ and the mapping space is taken in $\Vect(\frakS^{j}_{R}[\frac{1}{E}])^{\varphi=1}$. Without loss of generality, we may assume $M_s$ and $M_t$ are both finite free. Choose a basis $\underline e$ of $M_s^{\vee}$. As $\varphi^*(\frakS_R^j \underline e)$ is finite free, we have $\Map(\varphi^*(\frakS_R^j \underline e),M_s^{\vee})=\Map(\varphi^*(\frakS_R^j \underline e),(\frakS_R^j\underline e)[\frac{1}{u}])=\varinjlim_a \Map(\varphi^*(\frakS_R^j\underline e),\frac{1}{u^a}\frakS^j_R\underline e)$. This means we can choose a $\varphi$-stable lattice $\overline M^{\vee}_s$ in $M_s^{\vee}$. In the same way, we choose a $\varphi$-stable lattice $\overline M_t$ in $M_t$. Note that $M_{s,k}^{\vee}\simeq M_{s}^{\vee}\otimes_{R_s}R_k$ and $M_{s,R}^{\vee}\simeq M_s^{\vee}\otimes_{R_s}R$. Now we just need to prove $\varinjlim_{s,t\to k}(M_{s,k}^{\vee}\otimes M_{t,k})^{\varphi=1}\simeq (M_{s,R}^{\vee}\otimes M_{t,R})^{\varphi=1}$. This follows from Lemma \ref{F-invariant}.

Next we claim that the functor $\varinjlim_i\Vect(\frakS^{}_{R_i}[\frac{1}{E}])^{\varphi=1,\simeq}\to \Vect(\frakS^{}_{R}[\frac{1}{E}])^{\varphi=1,\simeq}$ is also essentially surjective. By Lemma \ref{tot}, this will then  complete the proof of Proposition \ref{geometrically finite type}.

Note that for any $i$ and $n$, we have the following homotopy pullback diagram as in Fact \ref{truncation}
\begin{equation}
\xymatrix@C=0.6cm{
\tau_{\leq n+1}{R_i}\ar[d]\ar[rr]&&\tau_{\leq n}(R_i)\ar[d]\\
\tau_{\leq n}(R_i)\ar[rr]^-{s}&&\tau_{\leq n}(R_i)\oplus \pi_{n+1}(R_i)[n+2]
}
\end{equation}
By tensoring with $\frakS/E^m$ which is flat over $\bZ_p$, we get a homotopy pullback diagram compatible with $\varphi$-actions, which come from the $\varphi$-actions on $\frakS$,
\begin{equation}
    \xymatrix@C=0.6cm{
\tau_{\leq n+1}( \frakS/E^m\otimes R_i)\ar[d]\ar[rrrr]&&&&\tau_{\leq n}(\frakS/E^m\otimes R_i)\ar[d]^{\rm triv}\\
\tau_{\leq n}(\frakS/E^m\otimes R_i)\ar[rrrr]^-{s}&&&&\tau_{\leq n}(\frakS/E^m\otimes R_i)\oplus \pi_{n+1}(\frakS/E^m\otimes R_i)[n+2].
}
\end{equation}
Now by taking limit with respect to $m$ and inverting $E$, we get a homotopy pullback diagram compatible with $\varphi$-actions
\begin{equation}\label{Frob-diagram}
    \xymatrix@C=0.6cm{
\tau_{\leq n+1}{\frakS_{R_i}}[\frac{1}{E}]\ar[d]\ar[rrrr]&&&&\tau_{\leq n}(\frakS_{R_i})[\frac{1}{E}]\ar[d]^{\rm triv}\\
\tau_{\leq n}(\frakS_{R_i})[\frac{1}{E}]\ar[rrrr]^-{s}&&&&\tau_{\leq n}(\frakS_{R_i}[\frac{1}{E}])\oplus \pi_{n+1}(\frakS_{R_i}[\frac{1}{E}])[n+2].
}
\end{equation}
Here we have used the fact that taking truncation commutes with limits and the fact that $\pi_{n+1}(\frakS_{R_i})\cong \varprojlim_m \pi_{n+1}(\frakS/E^m\otimes R_i)$ by Lemma \ref{pi-complete}.

Note that Diagram \ref{Frob-diagram} is exactly the diagram obtained by applying Fact \ref{truncation} to $\frakS_{R_i}$, which is a square-zero extension. The point is that now we know Diagram \ref{Frob-diagram} (especially the section $s$) is compatible with $\varphi$-actions and the Frobenius on $\tau_{\leq n+1}{\frakS_{R_i}}[\frac{1}{E}]$ is induced by the homotopy pullback diagram. In particular, it induces a homotopy pullback diagram
\begin{equation}\label{pull-frob}
    \xymatrix@C=0.6cm{
\Vect(\tau_{\leq n+1}{\frakS_{R_i}}[\frac{1}{E}])^{\varphi=1,\simeq}\ar[d]\ar[rr]&&\Vect(\tau_{\leq n}(\frakS_{R_i}[\frac{1}{E}]))^{\varphi=1,\simeq}\ar[d]^{\rm triv}\\
\Vect(\tau_{\leq n}(\frakS_{R_i}[\frac{1}{E}])^{\varphi=1,\simeq}\ar[rr]^-{s}&&\Vect(\tau_{\leq n}(\frakS_{R_i}[\frac{1}{E}])\oplus \pi_{n+1}(\frakS_{R_i}[\frac{1}{E}])[n+2]))^{\varphi=1,\simeq}
}
\end{equation}
Now we want to prove the left vertical functor in Diagram \ref{pull-frob} is essentially surjective, which can be deduced from that the right vertical functor is essentially surjective. Consider the base change functor between the homotopy categories
\[
{\rm Ho}(\Vect(\tau_{\leq n}(\frakS_{R_i}[\frac{1}{E}])^{\varphi=1})\xrightarrow{\rm triv} {\rm Ho}(\Vect(\tau_{\leq n}(\frakS_{R_i}[\frac{1}{E}])\oplus \pi_{n+1}(\frakS_{R_i}[\frac{1}{E}])[n+2]))^{\varphi=1}).
\]
Given $(M,\varphi_M)\in {\rm Ho}(\Vect(\tau_{\leq n}(\frakS_{R_i}[\frac{1}{E}])\oplus \pi_{n+1}(\frakS_{R_i}[\frac{1}{E}])[n+2]))^{\varphi=1})$, by base changing along the natural projection $\tau_{\leq n}(\frakS_{R_i}[\frac{1}{E}])\oplus \pi_{n+1}(\frakS_{R_i}[\frac{1}{E}])[n+2]\to \tau_{\leq n}(\frakS_{R_i}[\frac{1}{E}])$, we get $(M_1,\varphi_{M_1})$. Further base changing along the trivial section ${\rm triv}: \tau_{\leq n}(\frakS_{R_i}[\frac{1}{E}])\to \tau_{\leq n}(\frakS_{R_i}[\frac{1}{E}])\oplus \pi_{n+1}(\frakS_{R_i}[\frac{1}{E}])[n+2]$, we get $(M_2,\varphi_{M_2})\in {\rm Ho}(\Vect(\tau_{\leq n}(\frakS_{R_i}[\frac{1}{E}])\oplus \pi_{n+1}(\frakS_{R_i}[\frac{1}{E}])[n+2]))^{\varphi=1})$. As the trivial section and the natural projection $\tau_{\leq n}(\frakS_{R_i}[\frac{1}{E}])\oplus \pi_{n+1}(\frakS_{R_i}[\frac{1}{E}])[n+2]\to \tau_{\leq n}(\frakS_{R_i}[\frac{1}{E}])$ induce the identity on $\pi_0$, we see that $(\pi_0(M),\varphi_{M})=(\pi_0(M_2),\varphi_{M_2})$. This must give an isomorphism between $(M,\varphi_M)$ and $(M_2,\varphi_{M_2})$ by Lemma \ref{lift lemma}, which finishes the proof of the essential surjectivity of the base change functor
\[\Vect(\tau_{\leq n+1}(\frakS_{R_i}[\frac{1}{E}])^{\varphi=1,\simeq}\to\Vect(\tau_{\leq n}(\frakS_{R_i}[\frac{1}{E}])^{\varphi=1,\simeq}.\]

We then get the essential surjectivity of the functor $\pi_0$
\[
\Vect((\frakS_{R_i}[\frac{1}{E}])^{\varphi=1,\simeq}\to\Vect(\pi_0(\frakS_{R_i}[\frac{1}{E}])^{\varphi=1,\simeq}.
\]




Let's now look at the following commutative diagram of base change functors
\begin{equation*}
    \xymatrix@=0.6cm{
    {\rm Ho}(\varinjlim_i\Vect(\frakS_{R_i}[\frac{1}{E}])^{\varphi=1,\simeq})\ar[rr]^{l_1}\ar[d]^{l_2}&& {\rm Ho}(\Vect(\frakS_{R}[\frac{1}{E}])^{\varphi=1,\simeq})\ar[d]^{l_3}\\
    \varinjlim_i\Vect(\pi_0(\frakS_{R_i}[\frac{1}{E}]))^{\varphi=1,\simeq}\ar[rr]^{l_4}&&\Vect(\pi_0(\frakS_{R}[\frac{1}{E}]))^{\varphi=1,\simeq}.
    }
\end{equation*}
By the limit-preserving property of classical stack of \'etale $\varphi$-modules, we know the functor $l_4$ is an equivalence. By previous discussions, we also know $l_1$ is fully faithful and $l_2,l_3$ are essentially surjective. So given $M$ in  ${\rm Ho}(\Vect(\frakS_{R}[\frac{1}{E}])^{\varphi=1,\simeq})$, we can get $N$ in ${\rm Ho}(\varinjlim_i\Vect(\frakS_{R_i}[\frac{1}{E}])^{\varphi=1,\simeq})$ such that $l_3\circ l_1(N)\simeq l_3(M)$. By Lemma \ref{lift lemma}, this implies $l_1$ is essentially surjective, which finally finishes the proof of Proposition \ref{geometrically finite type}.
\end{proof}
\begin{lem}\label{lift lemma}
    Let $A$ be an animated ring with a Frobenius action $\varphi$. Let $M,N\in \Vect(A)^{\varphi=1}$. If there is an isomorphism $f:\pi_0(M)\to \pi_0(N)$ in $\Vect(\pi_0(A))^{\varphi=1}$, then $f$ can be lifted to an isomorphism $\tilde f: M\to N$ in ${\rm Ho}(\Vect(A)^{\varphi=1})$.
\end{lem}
\begin{proof}
    Let $V$ be an \'etale $\varphi$-module in $\Vect(A)^{\varphi=1}$. We have an exact triangle
    \[
    V^{\varphi=1}\to V\xrightarrow{\varphi-1}V.
    \]
    Then we see the natural map $\pi_0(V^{\varphi=1})\to (\pi_0(V))^{\varphi=1}$ is surjective. This means $f$ can be lifted to a map $\tilde f:M\to N$ by taking $V=M^{\vee}\otimes N$. As a map between finite projective modules over $A$ is an isomorphism if and only if the induced map on $\pi_0$ is an isomorphism, we see that $\tilde f$ is an isomorphism.
\end{proof}

We can also prove a similar result for $(\Lan\calX_{\rm EG})^{\#,\rm nil}$.
\begin{prop}\label{laft-2}
    The derived prestack $(\Lan\calX_{\rm EG})^{\#,\rm nil}$ is locally almost of finite type.
\end{prop}

\begin{proof}
 Let $((\Lan\calX_{\rm EG})^{\#})^{\leq n}$ be the restriction of $(\Lan\calX_{\rm EG})^{\#}$ to $\textbf{Nilp}_{\bZ_p}^{\leq n}$. Let $(\Lan\calX_{\rm EG})^{\#,\leq n}_{\rm ft}$ be its restriction to $\textbf{Nilp}_{\bZ_p,\rm ft}^{\leq n}$. We need to prove  $((\Lan\calX_{\rm EG})^{\#})^{\leq n}$ is the left Kan extension of $(\Lan\calX_{\rm EG})^{\#,\leq n}_{\rm ft}$ along the inclusion $\textbf{Nilp}_{\bZ_p,\rm ft}^{\leq n}\to \textbf{Nilp}_{\bZ_p}^{\leq n}$. By \cite[Chapter 2, Corollary 2.5.7]{GR19}, we have $((\Lan\calX_{\rm EG})^{\#})^{\leq n}\simeq ((\Lan\calX_{\rm EG})^{\leq n})^{\#_{n}}$ where $\#_{n}$ means the sheafification functor on the level of $\textbf{Nilp}_{\bZ_p}^{\leq n}$ (cf. \cite[Chapter 2, Section 2.5.1]{GR19}).
 
 As $\calX_{\rm EG}$ is locally of finite type, we know $(\Lan\calX_{\rm EG})^{\nil}$ is locally almost of finite type by \cite[Chapter 2, Corollary 1.7.8]{GR19}. Let $(\Lan\calX_{\rm EG})^{\leq n}_{\rm ft}$ be the restriction of $(\Lan\calX_{\rm EG})^{\leq n}$ to $\textbf{Nilp}_{\bZ_p,\rm ft}^{\leq n}$. Then $(\Lan\calX_{\rm EG})^{\leq n}$ is the left Kan extension of $(\Lan\calX_{\rm EG})^{\leq n}_{\rm ft}$ along the inclusion $\textbf{Nilp}_{\bZ_p,\rm ft}^{\leq n}\to \textbf{Nilp}_{\bZ_p}^{\leq n}$. By \cite[Chapter 2, Crollary 2.7.5(b) and 2.7.10]{GR19}, we see $((\Lan\calX_{\rm EG})^{\leq n})^{\#_{n}}$ is the left Kan extension of $(\Lan\calX_{\rm EG})^{\#,\leq n}_{\rm ft}$ along the inclusion $\textbf{Nilp}_{\bZ_p,\rm ft}^{\leq n}\to \textbf{Nilp}_{\bZ_p}^{\leq n}$.
\end{proof}

Now as both $(\Lan\calX_{\rm EG})^{\#,\rm nil}$ and $\calX^{\rm nil}$ are locally almost of finite type, it suffices to reduce Proposition \ref{GR-criterion} to classical points which are valued in finite type algebras. We make the following claim.

\textbf{Claim}:
  the pro-cotangent complexes of both $(\Lan\calX_{\rm EG})^{\#,\rm nil}$ and $\calX^{\rm nil}$ at classical points valued in finite type algebras are indeed complexes (instead of being just pro-complexes).

We first prove the claim for $(\Lan\calX_{\rm EG})^{\#,\rm nil}$. Let us begin with some discussion about the pro-cotangent complexs of formal schemes. Let $\Spf(R)$ be a classical Noetherian affine formal scheme with $I$-adic topology. Let $(x_1,\cdots,x_n)$ be a set of generators of $I$. The left Kan extension of $\Spf(R)$ is then the filtered colimit $\varinjlim_n \Lan \Spec(R/I^m)$. Note that $\Lan \Spec(R/I^m)$ is again representable by $\Spec(R/I^n)$. 

On the other hand, we can define animated rings $R_m=R\otimes^{\bL}_{f_m,\bZ[y_1,\cdots,y_n]}\bZ$ where the map $f_m:\bZ[y_1,\cdots,y_n]\to R$ sends each $y_i$ to $x_i^{2^m}$. Then we can consider the filtered colimit $\varinjlim_m R_n$ as derived prestacks.

\begin{prop}[{\cite[Proposition 2.1.4]{HP23}}]
    The natural morphism 
    \[
    \varinjlim_m \Spec(R/I^m)\to \varinjlim_m \Spec(R_m)
    \]
    is an equivalence.
\end{prop}

We still use $\Spf(R)$ to denote the above derived formal scheme. Now let $A$ be a finite type $\bZ/p^a$-algebra and $x: \Spec(A)\to \Spf(R)$ be an $A$-point of $\Spf(R)$. Then we have $\Map(\Spec(A),\Spf(R))=\varinjlim_m \Map(\Spec(A),\Spec(R_m))$.

In particular, we have $T^*_x(\Spf(R))=``\varprojlim_m"T^*_x(\Spec(R_m))=``\varprojlim_m" L_{R_m/\bZ}\otimes^{\bL}_{R_m}A$
as pro-systems by the definition of pro-cotangent complex.  Now we have the cotangent complex $L_{R/\bZ}$ of $R$. Consider its derived $I$-adic completion $\widehat L_{R/\bZ}$. We have a natural pro-system 
$``\varprojlim_m" L_{R/\bZ}\otimes^{\bL}_{R}R_m$. For the given point $x$ as above, we have the following result.

\begin{lem}
    Regarding $\widehat L_{R/\bZ}\otimes^{\bL}_RA$ as a constant pro-complex, the natural morphism of pro-complexes $\widehat L_{R/\bZ}\otimes^{\bL}_RA\to ``\varprojlim_m" L_{R/\bZ}\otimes^{\bL}_{R}R_m\otimes^{\bL}_{R_m}A$ is an equivalence.
\end{lem}
\begin{proof}
    This is clear as \[L_{R/\bZ}\otimes^{\bL}_{R}R_m\otimes^{\bL}_{R_m}A\simeq L_{R/\bZ}\otimes^{\bL}_RR/I^m\otimes^{\bL}_{R/I^m}A\simeq \widehat L_{R/\bZ}\otimes^{\bL}_RR/I^m\otimes^{\bL}_{R/I^m}A\simeq \widehat L_{R/\bZ}\otimes^{\bL}_RA.\]
\end{proof}

Consider the maps $\bZ\to R\to R_m$. This induces an exact triangle
\[
L_{R/\bZ}\otimes^{\bL}_RR_m\to L_{R_m/\bZ}\to L_{R_m/R},
\]
which induces an exact triangle of pro-complexes
\[
``\varprojlim_m"L_{R/\bZ}\otimes^{\bL}_RR_m\otimes^{\bL}_{R_m}A\to ``\varprojlim_m" L_{R_m/\bZ}\otimes^{\bL}_{R_m}A\to ``\varprojlim_m"L_{R_m/R}\otimes^{\bL}_{R_m}A
\]
We claim the last term vanishes.
\begin{lem}
    The pro-complex $``\varprojlim_m"L_{R_m/R}\otimes^{\bL}_{R_m}A$ is isomorphic to $0$.
\end{lem}
\begin{proof}
    Note that $R_m=R\otimes^{\bL}_{f_m,\bZ[y_1,\cdots,y_n]}\bZ$. So \[L_{R_m/R}\simeq L_{\bZ/\bZ[y_1,\cdots,y_n]}\otimes^{\bL}_{\bZ[y_1,\cdots,y_n],f_m}R\simeq L_{\bZ/\bZ[y_1,\cdots,y_n]}\otimes^{\bL}_{\bZ}R_m\simeq (y_1,\cdots,y_n)/(y_1,\cdots,y_n)^2[1]\otimes^{\bL}_{\bZ}R_m.\]

    The transition map $R_m\to R_{m-1}$ is induced by the map $\bZ[y_1,\cdots,y_n]\to \bZ[y_1,\cdots,y_n]$ sending $y_i$ to $y_i^2$. So the transition map
    \[
    L_{R_m/R}\otimes^{\bL}_{R_m}A\to L_{R_{m-1}/R}\otimes^{\bL}_{R_{m-1}}A
    \]
    turns into
    \[
    (y_1,\cdots,y_n)/(y_1,\cdots,y_n)^2[1]\otimes^{\bL}_{\bZ}A\to (y_1,\cdots,y_n)/(y_1,\cdots,y_n)^2[1]\otimes^{\bL}_{\bZ}A,
    \]
    which sends $y_i$ to $y_i^2$. Hence the transition map is $0$. So the pro-system $``\varprojlim_m"L_{R_m/R}\otimes^{\bL}_{R_m}A$ is $0$.
\end{proof}

\begin{cor}\label{pro-cotangent of formal scheme}
    The pro-cotangent complex $T^*_x(\Spf(R))$ is isomorphic to $\widehat L_{R/\bZ}\otimes^{\bL}_RA$, which is a complex of $A$-modules.
\end{cor}
\begin{proof}
    This follows from the above lemmas.
\end{proof}

Now we are ready to prove the claim about the pro-cotangent complex of $(\Lan\calX_{\rm EG})^{\#}$. 

\begin{prop}\label{key-a}
   Let $e:\Spec(B)\to (\Lan\calX_{\rm EG})^{\#}$ be a classical point with $B$ being a finite type algebra over $\bZ/p^a$ for some $a$. Then the pro-cotangent complex $T^*_{e}((\Lan\calX_{\rm EG})^{\#})$ is a complex of $B$-modules.
\end{prop}
\begin{proof}
    
Recall that $\calX_{\rm EG}$ is a Noetherian formal algebraic stack. This means there is a smooth surjective morphism $U\to \calX_{\rm EG}$ where $U$ is a Noetherian affine formal scheme over $\Spf(\bZ_p)$.

Let $U^{\bullet}$ be the \v Cech nerve of the morphism $U\to \calX_{\rm EG}$. Then the \'etale sheafification $(\Lan(U^{\bullet}))^{\#}$ of the left Kan extension of $U^{\bullet}$ is a groupoid object, whose geometric realization in the $\infty$-category of derived stacks is exactly $(\Lan\calX_{\rm EG})^{\#}$. And $(\Lan U)^{\#}\to (\Lan\calX_{\rm EG})^{\#}$ is again smooth surjective by \cite[Chapter 2, Proposition 4.4.3]{GR19}.

 As $(\Lan U)^{\#}\to (\Lan\calX_{\rm EG})^{\#}$ is smooth surjective, there exists an \'etale covering $\{\Spec(B_j)\}_{j\in J}$ of $\Spec(B)$, with $J$ a finite set, such that each $e_j:\Spec(B_j)\to (\Lan\calX_{\rm EG})^{\#}$ factors through $(\Lan U)^{\#}\to (\Lan\calX_{\rm EG})^{\#}$. Then we have an exact triangle for each $j\in J$
\[
T^*_{e_j}((\Lan\calX_{\rm EG})^{\#})\to T^*_{e_j}((\Lan U)^{\#})\to T^*_{e_j}((\Lan U)^{\#}/(\Lan\calX_{\rm EG})^{\#})
\]
where $T^*_{e_j}((\Lan U)^{\#}/(\Lan\calX_{\rm EG})^{\#})$ is the relative pro-cotangent complex (cf. \cite[Chapter 1, Section 2.4.3]{GR17}) at the point $e_j$. As $(\Lan U)^{\#}$ is indeed a Noetherian affine formal scheme, we know $T^*_{e_j}((\Lan U)^{\#})$ is a complex of $B_j$-modules by Corollary \ref{pro-cotangent of formal scheme}. Moreover as $(\Lan U)^{\#}\to (\Lan\calX_{\rm EG})^{\#}$ is smooth surjective, the pro-complex $T^*_{e_j}((\Lan U)^{\#}/(\Lan\calX_{\rm EG})^{\#})$ is also a complex of $B_j$-modules. Therefore, we see $T^*_{e_j}((\Lan\calX_{\rm EG})^{\#})$ is also a complex of $B_j$-modules. 
By the next Lemma \ref{etale covering}, the pro-cotangent complex $T^*_{e}((\Lan\calX_{\rm EG})^{\#})$ is then a complex of $B$-modules.
\end{proof}

\begin{lem}\label{etale covering}
    Let $``\varprojlim_{i\in I}"T_i$ be a pro-complex of $A$-modules. Let $\{B_j\}_{j\in J}$ be an \'etale covering of $A$ where $J$ is a finite set. Suppose that for each $j\in J$, the pro-complex $``\varprojlim_{i\in I}"T_i\otimes_AB_j$ is isomorphic to $P_j$, a complex of $B_j$-modules. Then the pro-complex $``\varprojlim_{i\in I}"T_i$ is indeed a complex of $A$-modules.
\end{lem}

\begin{proof}
    For each $j\in J$, the isomorphism $``\varprojlim_{i\in I}"T_i\otimes_AB_j\simeq P_j$ implies there exists an $i_j\in I$ such that for any $T_k$ living over $T_{i_j}$, we have $T_k\otimes_AB_j\simeq P_j$. Then there exists an $i\in I$ such that for any $T_k$ living over $T_i$, we have $T_i\otimes_AB_j\simeq P_j$ for all $j\in J$. In particular, this means for any $T_k$ living over $T_i$, we always have $T_k\simeq T_i$ by the faithfully flat descent. This means $``\varprojlim_{i\in I}"T_i$ is indeed a complex of $A$-modules.
\end{proof}

Next we move to prove the claim for $\calX^{\rm nil}$. 

\begin{prop}\label{key-b}
   Let $e:\Spec(C)\to \calX^{\nil}$ be a classical point with $C$ being a finite type algebra over $\bZ/p^a$ for some $a$. Then the pro-cotangent complex $T^*_{e}(\calX^{\nil})$ is a complex of $C$-modules, which is exactly the shift by $(-1)$ of the dual of the Herr complex associated to the adjoint of the \'etale $(\varphi,\Gamma)$-module corresponding to $e$.
\end{prop}
\begin{proof}

 Let $M\in \Mod(C)^{[-n,0]}$ such that $M\simeq \tau^{\geq -n}\tilde M$ for a perfect complex $\tilde M$. Then $C\oplus M\in \textbf{Nilp}_{\bZ_p,\rm ft}^{\leq n}$. By the definition of the pro-cotangent complex, we have the pullback diagram
\begin{equation*}
    \xymatrix{
    \Map(T^*_e(\calX^{\nil}),M))\ar[r]\ar[d] & {*}_e\ar[d]\\
    \calX^{\nil}(C\oplus M)\ar[r]& \calX^{\nil}(C).
    }
\end{equation*}

Note that the cotangent complex of the stack $\rm BGL_d$ at a point $\Spec(D)\to \rm BGL_d$ is given by $ad(M_D)[-1]$, where $M_D$ is the corresponding finite projective module over $D$ and $ad(M_D):=M_D\otimes M_D^{\vee}$. Then by Proposition \ref{derivedperf} in the next subsection, we can use the perfect prismatic site and get
\[
\calX^{\nil}(C)\simeq \varprojlim_i {\rm BGL_d}(A^i_{\cyc,C}[\frac{1}{\xi}])^{\varphi=1}
\]
and
\[
\calX^{\nil}(C\oplus M)\simeq \varprojlim_i {\rm BGL_d}(A^i_{\cyc,C\oplus M}[\frac{1}{\xi}])^{\varphi=1}.
\]

Note that $A^i_{\cyc,C\oplus M}[\frac{1}{\xi}])=A^i_{\cyc,C}[\frac{1}{\xi}]\oplus (M\otimes_CA^i_{\cyc,C}[\frac{1}{\xi}])$. This implies
\[
\Map(T^*_e(\calX^{\nil}),M))\simeq \varprojlim_i \Map(T^*_{e_i}({\rm BGL_d}),M\otimes_CA^i_{\cyc,C}[\frac{1}{\xi}])^{\varphi=1}
\]
where $e_i:\Spec(A^i_{\cyc,C}[\frac{1}{\xi}])\to {\rm BGL_d}$ is the point corresponding to $e$.

Also, we know that 
\[
\varprojlim_i \Map(T^*_{e_i}({\rm BGL_d}),M\otimes_CA^i_{\cyc,C}[\frac{1}{\xi}])^{\varphi=1}\simeq \varprojlim_i\tau_{\geq 0}(ad(M_i)[1]\otimes_CM)^{\varphi=1}\simeq \tau_{\geq 0}\varprojlim_i (ad(M_i)[1]\otimes_CM)^{\varphi=1},
\]
where $M_i$ is the corresponding object in ${\rm BGL_d}(A^i_{\cyc,C}[\frac{1}{\xi}])^{\varphi=1}$. In particular, we have 
\[ad(M_i)=ad(M_0)\otimes_{A_{\cyc,C}[\frac{1}{\xi}]} A^i_{\cyc,C}[\frac{1}{\xi}]\cong C(\Gamma^i,ad(M_0)).\]

Recall that the homotopy limit of cosimplicial chain complexs is quasi-isomorphic to the totalisation of its corresponding double complex. Write $C(\Gamma^{\bullet},ad(M_0))$ the complex associated to the cosimplicial module $\{ad(M_i)\}_i$. As each term $ad(M_i)$ is flat over $C$, we see that
\[
\varprojlim_i (ad(M_i)\otimes_CM)\simeq C(\Gamma^{\bullet},ad(M_0))\otimes_C M
\]
and therefore
\[
\varprojlim_i (ad(M_i)\otimes_CM)^{\varphi=1}\simeq C(\Gamma^{\bullet},ad(M_0))^{\varphi=1}\otimes_C M.
\]

Next we are going to study the complex $C(\Gamma^{\bullet},ad(M_0))^{\varphi=1}$. We claim this complex is (a shift of) the Herr complex of the \'etale $(\varphi,\Gamma)$-module $M_C$ corresponding to the point $e: \Spec(C)\to \calX^{\nil}$. Note that $M_C$ corresponds to the \'etale $(\varphi,\Gamma)$-modules $M_0$ over $A_{\cyc,C}[\frac{1}{\xi}]$ under the equivalence in Proposition \ref{Frobenius descent}. To prove the claim, we have to resort to the ``basic" case introduced in \cite{EG22}, in which we can always find $(\varphi,\Gamma)$-stable lattices.

Let us recall the following definition.

\begin{dfn}[\cite{EG22} Definition 3.2.3]
    If $K$ is a finite extension of $\bQ_p$, we set $K^{\rm basic}:=K\cap K_0(\zeta_{p^{\infty}})$, where $K_0$ is the maximal unramified extension of $\bQ_p$ inside $K$.
\end{dfn}
The advantage of working with $K^{\rm basic}$ is that $A^+_{K^{\rm basic},C}$ is $(\varphi,\Gamma)$-stable. Now the \'etale $(\varphi,\Gamma)$-module $M_C$ of rank $d$ over $A_{K,C}$ can be regarded as an \'etale $(\varphi,\Gamma)$-module of rank $d[K:K^{\rm basic}]$ over $A_{K^{\rm basic},C}$ (cf. the discussion after \cite[Definition 3.2.3]{EG22}). By the same proof of \cite[Lemma 5.1.5]{EG22}, we can find a $(\varphi,\Gamma)$-stable $A^+_{K^{\rm basic},C}$-lattice $ad(\frakM_C)$ in $ad(M_C)$. Let $(e_1,\cdots,e_s)$ be a set of generators of $ad(\frakM_C)$ over $A^+_{K^{\rm basic},C}$.  Note that $A^+_{K^{\rm basic},C}$ is actually a subring of $A_{\cyc,C}$ (in fact, $A^+_{K^{\rm basic},C}=A^+_{K_0,C}\subset A_{\cyc,C}$ by the discussion after \cite[Deinition 2.1.12]{EG22}). Therefore, the $A_{\cyc,C}$-submodule $ad(\frakM_0)$ of $ad(M_0)$ generated by $(e_1,\cdots,e_s)$ is a $(\varphi,\Gamma)$-stable lattice.

Note that $A_{\cyc,C}[\frac{1}{\xi}]$ is a Noetherian Banach ring by \cite[Proposition 2.4.3]{EG22}. Then the forgetful functor from finite Banach modules over $A_{\cyc,C}[\frac{1}{\xi}]$ to finite $A_{\cyc,C}[\frac{1}{\xi}]$-modules is an equivalence of categories by \cite[remark 2.2.11]{KL15}. In particular, $ad(\frakM_0)$ is an open subgroup of $ad(M_0)$ (this is true for any lattice). As $ad(M_0)$ is finite Banach module, we have 
\[
ad(M_0)\cong \varprojlim_iad(M_0)/T^iad(\frakM_0).
\]
Note that each $\xi^iad(\frakM_0)$ is again open in  $ad(M_0)$. Hence $ad(M_0)/\xi^iad(\frakM_0)$ has discrete topology. By construction, we know $\frakM_0$ is $\Gamma$-stable. As the $\Gamma$-action on $A_{\cyc}$ preserves the ideal $(\xi)$, we know each $\xi^i\frakM_0$ is also $\Gamma$-stable. So the isomorphism $ad(M_0)\cong \varprojlim_iad(M_0)/\xi^iad(\frakM_0)$ actually writes the $\Gamma$-module $ad(M_0)$ as an inverse limit of discrete $\Gamma$-modules. In particular, we can apply \cite[Lemma 7.3]{BMS18} to get
\[
C(\Gamma^{\bullet},ad(M_0))\simeq [ad(M_0)\xrightarrow{\gamma-1}ad(M_0)].
\]
This implies 
\[
C(\Gamma^{\bullet},ad(M_0))^{\varphi=1}\simeq [ad(M_0)\xrightarrow{\gamma-1}ad(M_0)]^{\varphi=1}\simeq {\rm Herr}(ad(M_0))
\]
where ${\rm Herr}(ad(M_0))$ is the Herr complex associated with $ad(M_0)$.

Recall that $M_C$ is the \'etale $(\varphi,\Gamma)$-module over $A_{\cyc,C}[\frac{1}{\xi}]$ corresponding to $M_0$. We have a natural map 
\[
{\rm Herr}(ad(M_C))\to {\rm Herr}(ad(M_0)).
\]

By \cite[Proposition 4.5]{pham2023moduli}, we have $H^i({\rm Herr}(ad(M_C)))$ is naturally isomorphic to the extension group $\Ext^i(A_{K,C},ad(M_C))$ in the exact category of finite projective \'etale $(\varphi,\Gamma)$-modules over $A_{K,C}$. The same proof actually applies to the exact category of finite projective \'etale $(\varphi,\Gamma)$-modules over $A_{\cyc,C}[\frac{1}{\xi}]$. The only place that requires attention is the proof of \cite[Lemma 4.10]{pham2023moduli}: one needs to use that $A_{\cyc,C}$ is $\varphi$-stable and any $A_{\cyc,C}$-lattice in a finite projective \'etale $(\varphi,\Gamma)$-module over $A_{\cyc,C}[\frac{1}{\xi}]$ is closed and complete. Hence we also get a natural isomorphism $H^i({\rm Herr}(ad(M_0))\cong \Ext^i(A_{\cyc,C}[\frac{1}{\xi}],ad(M_0))$. Now as the natural map $A_{K,C}\to A_{\cyc,C}[\frac{1}{\xi}]$ is faithfully flat by \cite[Proposition 2.2.12]{EG22}, the equivalence between the category of finite projective \'etale $(\varphi,\Gamma)$-modules over $A_{K,C}$ and the category of finite projective \'etale $(\varphi,\Gamma)$-modules over $A_{\cyc,C}[\frac{1}{\xi}]$ in \cite[Proposition 2.7.8]{EG22} is an exact equivalence. This then implies the natural map ${\rm Herr}(ad(M_C))\to {\rm Herr}(ad(M_0))$ is a quasi-isomorphism.

By \cite[Theorem 5.1.22]{EG22}, we see that ${\rm Herr}(ad(M_0))$ is indeed a perfect complex of $C$-modules. Then we have 
\[
\Map(T^*_e(\calX^{\nil}),M))\simeq \tau_{\geq 0}({\rm Herr}(ad(M_0))[1]\otimes_C M)\simeq \Map({\rm Herr}(ad(M_0))^{\vee}[-1],M).
\]

As ${\rm Herr}(ad(M_0))^{\vee}[-1]$ itself is a perfect complex, the functor $\Map({\rm Herr}(ad(M_0))^{\vee}[-1],-)$ commutes with filtered colimit. Then we have
\[
\Map(T^*_e(\calX^{\nil}),M))\simeq  \Map({\rm Herr}(ad(M_0))^{\vee}[-1],M)
\]
holds for any $M\in \Mod(C)^{\leq 0}$ since $M\simeq \varprojlim_n\tau^{\geq -n}(M)$ and any $N\in \Mod(C)^{[-n,0]}$ can be written as a filtered colimit $\varinjlim_jN_j$ such that $N_j$ is the truncation $\tau^{\geq -n}\tilde N_j$ for some perfect complex $N_j$. By the uniqueness of the pro-cotangent complex, we see 
\[T^*_e(\calX^{\nil})\simeq {\rm Herr}(ad(M_0))^{\vee}[-1]\simeq {\rm Herr}(ad(M_C))^{\vee}[-1].\]
In particular, $T^*_e(\calX^{\nil})$ is a perfect complex over $C$.
    
\end{proof}

Once we know the claim about the pro-cotangent complexes, we can be further reduced to considering only the classical points valued in finite fields.

\begin{lem}\label{reduce to finite field}
 Let $x$ be any finite-field point of $\calX_{\rm EG}$, i.e. $x\in \calX_{\rm EG}(k_f)$ where $k_f$ is a finite field. If the map
 \[
 f_x:T^*_x(\calX^{\rm nil})\to T^*_x((\Lan\calX_{\rm EG})^{\#,\rm nil})
 \]
 is an isomorphism for all finite-field points, then the natural map $(\Lan\calX_{\rm EG})^{\#,\rm nil}\to \calX^{\rm nil}$ is an isomorphism.
\end{lem}

\begin{proof}
    For any classical ring $R\in {\rm Nilp}_{\bZ_p}$, write $R=\varinjlim_iR_i$ as filtered colimit of all finite type subalgebras $R_i$. Since $\calX^{\rm nil}$ is locally almost of finite type, we have $\calX(R)=\varinjlim_i\calX(R_i)$. So any point $x\in \calX(R)$ is induced by a point $x\in \calX(R_i)$. The same is true for $(\Lan\calX_{\rm EG})^{\#,\rm nil}$. As both of them admit pro-cotangent complexes, that $f_x$ is an isomorphism for any discrete ring $R$ is equivalent to that $f_x$ is an isomorphism for finite type algebras (the former is the base change of the latter).

    Now let's assume $R$ is a finite type algebra over $\bZ_p/p^a$ for some $a$. Then there is a faithfully flat covering $\bigsqcup_i{\rm Spec}(\hat R_{m_i})\to {\rm Spec}(R)$ given by the disjoint union of the spectra of the completion of localizations of $R$ at its maximal ideals. This enables us to reduce further to the points valued in complete Noetherian local rings with finite residue fields.

    Now let's assume $R$ is a complete Noetherian lcoal ring with maximal ideal $m$ and residue field $k_f$ and $x\in (\Lan\calX_{\rm EG})^{\#,\rm nil}(R)=\calX_{\rm EG}(R)$. By Proposition \ref{key-a} and Proposition \ref{key-b}, we put $T^*_x((\Lan\calX_{\rm EG})^{\#,\rm nil})=H\in\Mod(R)$ and $T^*_x(\calX^{\rm nil})=K\in \Mod(R)$. Then by Lemma \ref{laft} and Definition \ref{cotangent complex laft}, to prove $f_x$ is an isomorphism, we just need to compare the cotangent spaces as filtered colimit preserving functors ${\Mod}(R)^{[-n,0]}\to \textbf{Ani}$ for each $n\geq 0$. Note every object $M\in {\Mod}(R)^{[-n,0]}$ can be written as a filtered colimit of perfect complexes $M\simeq \varinjlim_iK_i$. Then we see that $M\simeq \varinjlim_i\tau_{\leq n}(\tau_{\geq 0}K_i)$ (we still use the homological convention here) as $\tau_{\leq n}$ is a left adjoint and filtered colimit is exact. In particular, each $\tau_{\leq n}(\tau_{\geq 0}K_i)$ is derived $m$-complete. So we just need to compare $H^{\wedge}_{m}$ and $K^{\wedge}_{m}$. By derived Nakayama lemma, we are reduced further to comparing $H\otimes k_f$ and $K\otimes k_f$, i.e. we just need to prove $f_x$ is an isomorphism for all finite-field points.
\end{proof}

\begin{rmk}
    That the pro-cotangent complexes of the two derived stacks at classical points valued in finite type algebras are indeed complexes is crucial in the proof of Lemma \ref{reduce to finite field}. The point is that pro-complexes do not satisfy faithfully flat descent in general and the derived Nakayama lemma can not apply to pro-complexes.
\end{rmk}

Now by Lemma \ref{reduce to finite field}, proving Theorem \ref{main2} can be reduced to comparing the infinitesimal deformation theory of $(\Lan\calX_{\rm EG})^{\#,\rm nil}$ and $\calX^{\rm nil}$, which will imply the pro-cotangent complexes are isomorphic. More precisely, fixing a finite-field point $\bar \rho\in \calX_{\rm EG}(k_f)$, i.e. a residual representation of the absolute Galois group $G_K$, we need to investigate the functor $(\Lan\calX_{\rm EG})^{\#,\rm nil}_{\bar \rho}: {{\textbf{Art}}_{/k_f}}\to \textbf{Ani}$ defined as
\[
R\mapsto {\rm Fib}_{\bar \rho}((\Lan\calX_{\rm EG})^{\#,\rm nil}(R)\to (\Lan\calX_{\rm EG})^{\#,\rm nil}(k_f)),
\]
and the functor $(\calX)_{\bar \rho}^{\nil}: {{\textbf{Art}}_{/k_f}}\to \textbf{Ani}$ defined as
\[
R\mapsto {\rm Fib}_{\bar \rho}(\calX^{\nil}(R)\to \calX^{\nil}(k_f)),
\]
where ${{\textbf{Art}}_{/k_f}}$ is the $\infty$-category of Artinian local animated ring $R$ with an identification of $k_f$ with the residue field of $R$. Recall that an animated ring $R$ is called Artinian local if $\pi_0(R)$ is a classical Artinian local ring and $\pi_*(R)$ is a finitely generated $\pi_0(R)$-module. In particular all Artinian local animated rings over $\bZ_p$ are truncated. For any $M\in {\rm Perf}(k_f)^{[n,0]}$, we have $k_f\oplus M$ is an Artinian local animated ring. So if $(\Lan\calX_{\rm EG})^{\#}_{\bar \rho}=(\Lan\calX_{\rm EG})^{\#,\rm nil}_{\bar \rho}\simeq (\calX)^{\nil}_{\bar \rho}=(\calX)_{\bar \rho}$, we can deduce $T^*_{\bar \rho}(\calX^{\rm nil})\to T^*_{\bar \rho}((\Lan\calX_{\rm EG})^{\#,\rm nil})$ is an isomorphism by Definition \ref{cotangent complex laft}.

In order to compare $(\Lan\calX_{\rm EG})^{\#}_{\bar \rho}$ and $\calX_{\bar \rho}$ (we can omit the superscript $``\nil"$), we will go back to derived Galois representations. Our proof will be divided into two parts:
\begin{enumerate}
    \item compare $\calX$ with the derived prestack of Laurent $F$-crystals on the absolute perfect prismatic site;
    \item compare directly Laurent $F$-crystals on the absolute perfect prismatic site with derived representations.
\end{enumerate}

\subsubsection{Laurent $F$-crystals on perfect site}
In this subsection, we investigate the first part of our strategy. The main result is the following.

\begin{prop}\label{derivedperf}
    Let $R$ be an animated ring in $\textbf{Nilp}_{\bZ_p,\rm ft}^{\leq n}$ for some $n$. Then there is an equivalence 
    \[
\Vect((\calO_K)_{\Prism},\calO_{\Prism,R}[\frac{1}{\calI_{\Prism}}])^{\varphi=1,\simeq}\simeq \Vect((\calO_K)^{\perf}_{\Prism},\calO_{\Prism,R}[\frac{1}{\calI_{\Prism}}])^{\varphi=1,\simeq}.
    \]
\end{prop}

\begin{proof}
   As usual, using the Breuil--Kisin prism $(\frakS,(E))$ and its perfection $(A_{\infty},(E))$, we have 
   \[
   \Vect((\calO_K)_{\Prism},\calO_{\Prism,R}[\frac{1}{\calI_{\Prism}}])^{\varphi=1,\simeq}\simeq \varprojlim \Vect(\frakS_{R}[\frac{1}{E}])^{\varphi=1,\simeq}\rightrightarrows\Vect(\frakS_{R}^1[\frac{1}{E}])^{\varphi=1,\simeq}\cdots\Vect(\frakS_{R}^{n+2}[\frac{1}{E}])^{\varphi=1,\simeq},
   \]
   and
   \[
   \Vect((\calO_K)^{\perf}_{\Prism},\calO_{\Prism,R}[\frac{1}{\calI_{\Prism}}])^{\varphi=1,\simeq}\simeq \varprojlim \Vect(A_{\infty,R}[\frac{1}{E}])^{\varphi=1,\simeq}\rightrightarrows\Vect(A_{\infty,R}^1[\frac{1}{E}])^{\varphi=1,\simeq}\cdots\Vect(A_{\infty,R}^{n+2}[\frac{1}{E}])^{\varphi=1,\simeq}.
   \]

   By using the same arguments as in the proof of Theorem \ref{finitetype}, we see that \[\Vect(\frakS^i_{R}[\frac{1}{E}])^{\varphi=1,\simeq}\to \Vect(A^i_{\infty,R}[\frac{1}{E}])^{\varphi=1,\simeq}\] is fully faithful for every $0\leq i\leq n+2$. We now want to prove $\Vect(\frakS_{R}[\frac{1}{E}])^{\varphi=1,\simeq}\to \Vect(A_{\infty,R}[\frac{1}{E}])^{\varphi=1,\simeq}$ is moreover essentially surjective.

Note that we have a commutative diagram
\begin{equation*}
    \xymatrix@=0.6cm{
    {\rm Ho}(\Vect(\frakS_{R}[\frac{1}{E}])^{\varphi=1,\simeq})\ar[rr]^{l_1}\ar[d]^{l_2}&&{\rm Ho}(\Vect(A_{\infty,R}[\frac{1}{E}])^{\varphi=1,\simeq})\ar[d]^{l_3}\\
    \Vect(\pi_0(\frakS_{R})[\frac{1}{E}])^{\varphi=1,\simeq}\ar[rr]^{l_4}&& \Vect(\pi_0(A_{\infty,R})[\frac{1}{E}])^{\varphi=1,\simeq}.
    }
\end{equation*}
  We know that $l_1$ is fully faithful and $l_2,l_3$ are essentially surjective (by the same arguments as in the proof of Proposition \ref{geometrically finite type}). Moreover by \cite[Proposition 2.6.12]{EG22}, $l_4$ is an equivalence. Then by Lemma \ref{lift lemma}, we see that $l_1$ is also essentially surjective. Finally by Lemma \ref{tot}, we are done.
\end{proof}
\begin{rmk}
    Proposition \ref{derivedperf} recovers Proposition \ref{finitetype} as finite type $\bZ_p$-algebras are discrete Noetherian animated rings over $\bZ_p$.
\end{rmk}

\subsubsection{Derived representations with finite coefficients}
In this subsection, we will relate derived $F$-crystals on the perfect site to derived Galois representations. Let us first recall the construction of derived local Galois deformation functor in \cite[Definition 5.4]{GV18}.

\begin{cons}[Unframed derived local Galois deformation functor]
    Let $\bar \rho:G_K\to\GL_d(k_f)$ be a residual representation. We can define a functor $\calF_{K,\GL_d}:{\textbf{Art}}\to \textbf{Ani}$ as
    \[
    \calF_{K,\GL_d}(R):=\Map_{\textbf{Ani}}(|G_K|,|\GL_d(R)|):=\varinjlim_i\Map_{\textbf{Ani}}(|G_i|,|\GL_d(R)|)
    \]
    where $G_K=\varprojlim_iG_i$ is the inverse limit of finite quotient groups and $|\cdot|$ means geometric realization of the corresponding simplicial anima, i.e. $|G_i|=\varinjlim_jG_i^j$ and $|\GL_d(R)|=\varinjlim_j\GL_d^j(R)=\varinjlim_j\Map(\calO(\GL_d)^{\otimes j},R))$ . Then the unframed derived Galois deformation functor is $\calF_{K,\GL_d,\bar\rho}:{\textbf{Art}_{/k_f}}\to \textbf{Ani}$ is defined as
    \[
   \calF_{K,\GL_d,\bar\rho}(R):={\rm Fib}_{\bar\rho}( \calF_{K,\GL_d}(R)\to  \calF_{K,\GL_d}(k_f))
    \]
\end{cons}

\begin{prop}\label{derivedrep-laurent}
    Let $R$ be an Artinian local ring in $\textbf{Nilp}_{\bZ_p}^{\leq n}$ for some $n$. There is an equivalence
    \[
    \Vect((\calO_K)_{\Prism}^{\perf},\calO_{\Prism,R}[\frac{1}{\calI_{\Prism}}])^{\varphi=1,\simeq}\simeq \calF_{K,\GL_d}(R). 
    \]
\end{prop}

\begin{proof}
Assume $R$ is a $\bZ/p^a$-algebra for some $a$. We first define an anima $\Vect(R,G_K)^{\simeq}$ as follows 
\[
\Vect(R,G_K)^{\simeq}:=\varprojlim_j\Vect(R)^{\simeq}\rightrightarrows\Vect(C(G_K,\bZ/p^a)\otimes R)^{\simeq}\rightthreearrow\cdots\Vect(C(G_K^{n+1},\bZ/p^a)\otimes R)^{\simeq}
\]
where the cosimplicial animated rings $C(G_K^{\bullet},\bZ/p^a)\otimes R$ is defined using the trivial $G_K$-action on $\bZ/p^a$.
We will then compare both $\Vect((\calO_K)_{\Prism}^{\perf},\calO_{\Prism,R}[\frac{1}{\calI_{\Prism}}])^{\varphi=1,\simeq}$ and $ \calF_{K,\GL_d}(R)$ with $\Vect(R,G_K)^{\simeq}$ . 

Let us start with $ \calF_{K,\GL_d}(R)$. Note that any finite projective module over $R$ is finite free as $R$ is an Artinian local animated ring (see the first paragraph of the proof of \cite[Proposition 2.5.3]{Lur09}). So $\pi_0(\Vect(R)^{\simeq})=\pi_0(B\GL_d(R))=\{*\}$, which is represented by $R^d$. Now we claim $\Vect(R)^{\simeq}$ is equivalent to $|\GL_d(R)|=\varinjlim_m\GL^m(R)$. 

 As $B\GL_d$ is indeed the sheafification of the functor $|\GL_d(-)|$, we get a sequence $*\to |\GL_d(R)|\to B\GL_d(R)$. Note that $\pi_0(|\GL_d(R)|)=\{*\}$. Then it suffices to compare the mapping space $\Map(*,*)$ in both anima. By \cite[Lemma 5.2]{GV18}, we know that $\Map(*,*)$ in $|\GL_d(R)|$ is $\GL_d(R)$. By the homotopy pullback diagram
\begin{equation*}
    \xymatrix@=0.6cm{
    \GL_d\ar[r]\ar[d]& {*}\ar[d]\\
    {*} \ar[r]& B\GL_d,
    }
\end{equation*}
we see that $\Map(*,*)$ in $B\GL_d(R)$ is also $\GL_d(R)$. So $|\GL_d(R)|\simeq \Vect(R)^{\simeq}$.

As $\Vect(R)^{\simeq}$ has only one object $R^n$, we can rewrite $\Vect(R,G_K)^{\simeq}$ as follows
\[
\Vect(R,G_K)^{\simeq}:=\varprojlim_j{\rm Free}(R)^{\simeq}\rightrightarrows{\rm Free}(C(G_K,\bZ/p^a)\otimes R)^{\simeq}\rightthreearrow\cdots{\rm Free}(C(G_K^{n+1},\bZ/p^a)\otimes R)^{\simeq}
\]
where ${\rm Free}(C(G_K^{\bullet},\bZ/p^a)\otimes R)^{\simeq}$ means the full subcategory of  $\Vect(C(G_K^{\bullet},\bZ/p^a)\otimes R)^{\simeq}$ spanned by $C(G_K^{\bullet},\bZ/p^a)\otimes R^d$. In particular, ${\rm Free}(C(G_K^{j},\bZ/p^a)\otimes R)^{\simeq}\simeq |\GL_d(C(G_K^{j},\bZ/p^a)\otimes R)|$ for each $j$. Then we have 
\[
\Vect(R,G_K)^{\simeq}=\varprojlim_j |\GL_d(C(G_K^j,\bZ/p^a)\otimes R)|
\]
By the definition of $|\cdot|$ as geometric realization, we see that
\[
|\GL_d(C(G_K^j,\bZ/p^a)\otimes R|\simeq \varinjlim_m \GL_d^m(C(G_K^j,\bZ/p^a)\otimes R).
\]
For each $m$, we also have
\[
\GL_d^m(C(G_K^j,\bZ/p^a)\otimes R)=\varinjlim_i\GL_d^m(C(G_i^j,\bZ/p^a)\otimes R),
\]
as $\GL_d^m$ is of finite type. So we get
\[
\Vect(R,G_K)^{\simeq}\simeq \varprojlim_j\varinjlim_{m,i}\GL_d^m(C(G_i^j,\bZ/p^a)\otimes R)\simeq \varprojlim_j\varinjlim_{m,i}\GL_d^m(R^{G^j_i}).
\]
On the other hand, for $\calF_{K,\GL_d}(R)$, we have
\[
\calF_{K,\GL_d}(R)=\varinjlim_i\Map_{\textbf{Ani}}(|G_i|,|\GL_d(R)|)\simeq \varinjlim_i\varprojlim_j|\GL_d(R)|^{G_i^j}\simeq\varinjlim_i\varprojlim_j\varinjlim_m\GL_d^m(R^{G_i^j}).
\]
As the inverse limit $\varprojlim_j$ is in fact a finite limit which commutes with filtered colimit $\varinjlim_i$, we then get
\[
\Vect(R,G_K)^{\simeq}\simeq \calF_{K,\GL_d}(R).
\]

Next we come to compare $\Vect(R,G_K)^{\simeq}$ with  $\Vect((\calO_K)_{\Prism}^{\perf},\calO_{\Prism,R}[\frac{1}{\calI_{\Prism}}])^{\varphi=1,\simeq}$. By choosing the Fontaine prism $(A_{\inf},(\xi))$, we can write $\Vect((\calO_K)_{\Prism}^{\perf},\calO_{\Prism,R}[\frac{1}{\calI_{\Prism}}])^{\varphi=1,\simeq}$ as
\[
\varprojlim_j\Vect(A_{\inf}[\frac{1}{\xi}]/p^a\otimes R)^{\varphi=1,\simeq}\rightrightarrows\Vect(C(G_K,A_{\inf}[\frac{1}{\xi}]/p^a)\otimes R)^{\varphi=1,\simeq}\cdots \Vect(C(G_K^{n+1},A_{\inf}[\frac{1}{\xi}]/p^a)\otimes R)^{\varphi=1,\simeq}.
\]
Then there is a natural base change map 
\[
\Vect(R,G_K)^{\simeq}\to\Vect((\calO_K)_{\Prism}^{\perf},\calO_{\Prism,R}[\frac{1}{\calI_{\Prism}}])^{\varphi=1,\simeq}.
\]

For each $j\in[0,n+1]$, we claim the base change functor $\Vect(C(G_K^j,\bZ/p^a)\otimes R)^{\simeq}\to \Vect(C(G_K^j,A_{\inf}[\frac{1}{\xi}]/p^a)\otimes R)^{\varphi=1,\simeq}$ is fully faithful. This means for every $M\in \Vect(C(G_K^j,\bZ/p^a)\otimes R)^{\simeq}$, we need to prove $M\simeq (M\otimes C(G_K^j,A_{\inf}[\frac{1}{\xi}]/p^a))^{\varphi=1}$. This follows from the short exact sequence
\begin{equation}\label{key-short}
0\to C(G_K^j,\bZ/p^a)\to C(G_K^j,A_{\inf}[\frac{1}{\xi}]/p^a)\xrightarrow{\varphi-1}C(G_K^j,A_{\inf}[\frac{1}{\xi}]/p^a)\to 0.
\end{equation}
It is easy to see the sequence \ref{key-short} is left exact. To see it is right exact, we first prove $C(G_K^j,\calO_C^{\flat})\xrightarrow{\varphi-1}C(G_K^j,\calO_C^{\flat})$ is surjective. Let $\varpi$ be a pesudo-uniformiser of $\calO_C^{\flat}$. As $C^{\flat}$ is algebraically closed, we see that $\calO_C^{\flat}\xrightarrow{\varphi-1}\calO_C^{\flat}$ is surjective. Then so is $\calO_C^{\flat}/{\varpi^n}\xrightarrow{\varphi-1}\calO_C^{\flat}/{\varpi^n}$ for each $n$. This implies $C(G_K^j,\calO_C^{\flat}/{\varpi^n})\xrightarrow{\varphi-1}C(G_K^j,\calO_C^{\flat}/{\varpi^n})$ is surjective. Then by the Artin--Schreier sequence, we have a short exact sequence
\[
0\to \underline{\bF_p}({\rm Spec}(C(G_K^j,\calO_C^{\flat}/{\varpi^n})))\to C(G_K^j,\calO_C^{\flat}/{\varpi^n})\xrightarrow{\varphi-1}C(G_K^j,\calO_C^{\flat}/{\varpi^n})\to 0.
\]
Note that $R^1\varprojlim_n \underline{\bF_p}({\rm Spec}(C(G_K^j,\calO_C^{\flat}/{\varpi^n})))=0$. Then by taking limit with respect to $n$, we get a short exact sequence
\[
0\to C(G_K^j,\bF_p)\to C(G_K^j,\calO_C^{\flat})\xrightarrow{\varphi-1}C(G_K^j,\calO_C^{\flat})\to 0.
\]
So we also have $C(G_K^j,C^{\flat})\xrightarrow{\varphi-1}C(G_K^j,C^{\flat})$ is surjective. Now consider the following exact triangle
\[
C(G_K^j,W(C^{\flat}))\xrightarrow{\varphi-1}C(G_K^j,W(C^{\flat}))\to K.
\]
We have $H^{-1}(K)=C(G_K^j,\bZ_p)$ and $H^{-1}(K\otimes_{\bZ_p} \bF_p)=C(G_K^j,\bF_p)$. Since there is an exact sequence
\[
0\to H^{-1}(K)/p\to H^{-1}(K\otimes_{\bZ_p}\bF_p)\to H^0(K)[p]\to 0,
\]
we see that $H^0(K)[p]=0$. Now considering the following short exact sequence
\[
0\to H^0(K)/p\to H^0(K\otimes_{\bZ_p}\bF_p)\to H^1(K)[p]=0\to 0,
\]
we get $H^0(K)/p=H^0(K\otimes_{\bZ_p}\bF_p)=0$. As $H^0(K)$ is derived $p$-complete, we have $H^0(K)=0$ by derived Nakayama lemma. This means $C(G_K^j,W(C^{\flat}))\xrightarrow{\varphi-1}C(G_K^j,W(C^{\flat}))$ is surjective. Then so is $C(G_K^j,W(C^{\flat})/p^n)\xrightarrow{\varphi-1}C(G_K^j,W(C^{\flat})/p^n)$, which finally shows that the sequence \ref{key-short} is exact.

Now we claim that for $j=0$, the base change functor $\Vect(R)^{\simeq}\to \Vect( A_{\inf}[\frac{1}{\xi}]/p^a\otimes R)^{\varphi=1,\simeq}$ is essentially surjective. Note that we have the following commutative diagram
\begin{equation*}
    \xymatrix@=0.6cm{
    {\rm Ho}(\Vect( R)^{\simeq})\ar[rrrr]^{l_1}\ar[d]^{l_2}&&&&{\rm Ho}(\Vect(A_{\inf}[\frac{1}{\xi}]/p^a\otimes R)^{\varphi=1,\simeq})\ar[d]^{l_3}\\
    \Vect(\pi_0(R))^{\simeq}\ar[rrrr]^{l_4}&&&& \Vect(A_{\inf}[\frac{1}{\xi}]/p^a\otimes \pi_0(R))^{\varphi=1,\simeq}.
    }
\end{equation*}
We know that $l_1$ is fully faithful and $l_2,l_3$ are essentially surjective (by the same arguments as in the proof of Proposition \ref{geometrically finite type}). In fact, $l_2$ is indeed an equivalence. As $R$ is Artinian, the functor $l_4$ is also an equivalence. By Lemma \ref{lift lemma}, we must have $l_1$ is an equivalence. This implies the equivalence
\[
\Vect(R,G_K)^{\simeq}\xrightarrow{\simeq}\Vect((\calO_K)_{\Prism}^{\perf},\calO_{\Prism,R}[\frac{1}{\calI_{\Prism}}])^{\varphi=1,\simeq}.
\]

So we finally get the equivalence
\[
\Vect((\calO_K)_{\Prism}^{\perf},\calO_{\Prism,R}[\frac{1}{\calI_{\Prism}}])^{\varphi=1,\simeq}\simeq \calF_{K,\GL_d}(R).
\]
\end{proof}

\subsubsection{Derived representations with general coefficients}

One consequence of Proposition \ref{derivedperf} and Proposition \ref{derivedrep-laurent} is that we get an equivalence of functors $\calX_{\bar \rho}\simeq \calF_{K,\GL_d,\bar\rho}: {\textbf{Art}}_{/k_f}\to \textbf{Ani}$. Then in order to prove $(\Lan\calX_{\rm EG})^{\#}_{\bar \rho}\simeq \calX_{\bar \rho}$, it remains to prove the equivalence between $(\Lan\calX_{\rm EG})^{\#}_{\bar \rho}$ and $\calF_{K,\GL_d,\bar\rho}$. As a preparation, we recall some results in \cite{Zhu20} concerning derived representations with general coefficients.

\begin{dfn}[{\cite[Definition 2.4.3]{Zhu20}}]
    We define the derived prestack of framed $\GL_d$-valued continuous representations of $G_K$ over $\bZ_p$ as 
    \[
    \calR_{G_K,\GL_d}^{c}:{\textbf{Nilp}_{\bZ_p}}\to {\textbf{Ani}}, \ \ \ A\mapsto \varinjlim_r\Map_{{\textbf{Nilp}_{\bZ_p}^{\Delta}}}(\bZ/p^r[\GL_d^{\bullet}],C_{\rm cts}(G_K^{\bullet},A)).
    \]
    We can also define the derived prestack $ \calR_{G_K,\GL_d/\GL_d}^{c}$\footnote{Different to the definition in \cite{Zhu20}, we do not take sheafification here.} of unframed $\GL_d$-valued continuous representations of $G_K$ over $\bZ_p$ as the geometric realization of 
    \begin{equation}\label{unframed}
        \xymatrix{
        \cdots\ar@<-.9ex>[r]\ar@<-.3ex>[r]\ar@<.3ex>[r]\ar@<.9ex>[r]&\GL_d\times \GL_d\times  \calR_{G_K,\GL_d}^{c}\ar@<-.6ex>[r]\ar@<.0ex>[r]\ar@<.6ex>[r]&\GL_d\times  \calR_{G_K,\GL_d}^{c}\ar@<-.3ex>[r]\ar@<.3ex>[r]&\calR_{G_K,\GL_d}^{c}
        }
    \end{equation}
    in the $\infty$-category $\PreStk$.
\end{dfn}

Still write $G_K=\varprojlim_iG_i$ with $G_i$'s finite quotient groups. 
\begin{prop}[{\cite[Proposition 2.4.7]{Zhu20}}]\label{Zhu-colimit}
    If $A$ is a truncated animated ring in ${\textbf{Nilp}_{\bZ_p}}$, then
    \[
    \calR_{G_K,\GL_d}^{c}(A)\simeq \varinjlim_i\calR_{G_i,\GL_d}^{c}(A)
    \]
\end{prop}
In particular, each $\calR_{G_i,\GL_d}^{c}$\footnote{As $G_i$ is finite, we simply have $\calR_{G_i,\GL_d}^{c}=\calR_{G_i,\GL_d}$, i.e. all representations are already continuous. For the definition of $\calR_{G_i,\GL_d}$, see \cite[Definition 2.2.1]{Zhu20}.} turns out to be well-behaved.
\begin{prop}[{\cite[Proposition 2.2.3 and 2.2.11]{Zhu20}}]\label{Zhu}
    $\calR_{G_i,\GL_d}^{c}$ is a derived affine scheme almost of finite presentation over $\bZ_p$.
\end{prop}

Now we are ready to prove our main result.

\begin{prop}\label{X-F}
    There is an equivalence of functors
    \[
    (\Lan\calX_{\rm EG})^{\#}_{\bar \rho}\simeq\calF_{K,\GL_d,\bar\rho}.
    \]
\end{prop}

\begin{proof}
By Lemma \ref{Keykey-lemma}, we have $(\Lan\calX_{\rm EG})^{\#}_{\bar \rho}\simeq(\Lan\calX_{\rm EG})_{\bar \rho}$. So we just need to focus on proving $(\Lan\calX_{\rm EG})_{\bar \rho}\simeq\calF_{K,\GL_d,\bar\rho}$.
We will proceed in three steps
\begin{enumerate}
    \item prove $(\Lan\calX_{\rm EG})_{\bar \rho}\simeq (\Lan ^{\rm cl}\calR_{G_K,\GL_d/\GL_d}^{c})_{\bar \rho}$;
    \item prove $(\Lan^{\rm cl}\calR_{G_K,\GL_d/\GL_d}^{c})_{\bar \rho}\simeq\calR_{G_K,\GL_d/\GL_d,\bar\rho}^{c}$;
    \item prove $\calR_{G_K,\GL_d/\GL_d,\bar\rho}^{c}\simeq \calF_{K,\GL_d,\bar\rho}$.
\end{enumerate}

\textbf{Step 1}: let us begin with the first step. Let ${\rm Nilp_{\bZ_p,\rm ft}}$ denote the full subcategory of ${\rm Nilp}_{\bZ_p}$ spanned by finite type $\bZ_p$-algebras. Recall that $\calX_{\rm EG}$ is limit-preserving, or geometrically speaking locally of finite type. Then by \cite[Chapter 2, Section 1.6]{GR19} and \cite[Proposition 5.1.2]{CS23} we can write
\[
\calX_{\rm EG}\simeq\varinjlim_{(\Spec(A)\to \calX_{\rm EG})\in{\rm Nilp}_{\bZ_p}^{\rm op}{}_{/{\calX_{\rm EG}}}}h_A\simeq \varinjlim_{(\Spec(A)\to \calX_{\rm EG})\in{\rm Nilp}_{\bZ_p,\rm ft}^{\rm op}{}_{/{\calX_{\rm EG}}}} h_A,
\]
where ${\rm Nilp}_{\bZ_p}^{\rm op}{}_{/{\calX_{\rm EG}}}$, ${\rm Nilp}_{\bZ_p,\rm ft}^{\rm op}{}_{/{\calX_{\rm EG}}}$ are full subcategories of the overcategory ${\rm PreStk}{}_{/{\calX_{\rm EG}}}$,   the prestack $h_A$ is representable by $A$ and the colimit is taken in the $\infty$-category $^{\rm cl}\rm PreStk$. Then we have  
\[
\Lan\calX_{\rm EG}\simeq \varinjlim_{{\rm Nilp}_{\bZ_p,\rm ft}^{\rm op}{}_{/{\calX_{\rm EG}}}} h_A  
\]
where the colimit is taken in the $\infty$-category $\PreStk$ of derived prestacks, as the left Kan extension along ${\rm Nilp}_{\bZ_p}\to \textbf{Nilp}_{\bZ_p}$ sends $h_A$ to $h_A$. Now the functor $(\Lan\calX_{\rm EG})_{\bar \rho}:\textbf{Art}_{/k_f}\to \textbf{Ani}$ is defined as 
\[
(\Lan\calX_{\rm EG})_{\bar \rho}(R\to k_f):=\Lan\calX_{\rm EG}(R)\times_{\calX_{\rm EG}(k_f),\bar\rho}*\simeq (\varinjlim_{{\rm Nilp}_{\bZ_p,\rm ft}^{\rm op}{}_{/{\calX_{\rm EG}}}}h_A(R))\times_{\calX_{\rm EG}(k_f)}*_{\bar \rho}.
\]

Let $H$ be the set of connected components of $\calX_{\rm EG}(k_f)$ and choose a representative $*_{\alpha}$ from each connected component. Then we can write $\calX_{\rm EG}(k_f)\simeq\bigsqcup_{\alpha\in H}[*_{\alpha}/{{\rm Aut}(*_{\alpha})}]$. For any $\bar \rho_A:h_A\to \calX_{\rm EG}$ and $(R\to k_f)\in \textbf{Art}_{/k_f}$, we then can write
\[
h_A(R)\simeq\bigsqcup_{\alpha\in H}h_{A,\alpha}(R),
\]
where $h_{A,\alpha}(R)$ is the preimage of $[*_{\alpha}/{{\rm Aut}(*_{\alpha})}]$. In particular, there is $h_{A,\bar\rho}(R)$ corresponding to the fixed residual representation $\bar\rho$. As a result, we see that
\[
(\Lan\calX_{\rm EG})_{\bar \rho}(R\to k_f)\simeq (\bigsqcup_{\alpha\in H}\varinjlim_{{\rm Nilp}_{\bZ_p,\rm ft}^{\rm op}{}_{/{\calX_{\rm EG}}}}h_{A,\alpha}(R))\times_{\calX_{\rm EG}(k_f)}*_{\bar \rho}\simeq (\varinjlim_{{\rm Nilp}_{\bZ_p,\rm ft}^{\rm op}{}_{/{\calX_{\rm EG}}}}h_{A,\bar\rho}(R))\times_{\calX_{\rm EG}(k_f)}*_{\bar \rho}.
\]

Now we want to further study each $h_{A,\bar\rho}(R)$. We first look at $h_{A,\bar\rho}(k_f)$. Note that each element in $h_{A,\bar\rho}(k_f)$ determines a maximal ideal of $A$ as $k_f$ is a finite field. Let $T_{\bar\rho_A}$ denote the set of maximal ideals appearing  in this way and we can write $h_{A,\bar\rho}(k_f)=\bigsqcup_{m_\beta\in T_{\bar\rho_A}}h_{A,m_\beta}(k_f)$, such that all elements in $h_{A,m_\beta}(k_f)$ induce the same maximal ideal $m_{\beta}$. For each $m_{\beta}$, let $A_{(m_{\beta})}$ be the localization of $A$ at the maximal ideal $m_{\beta}$ and $\hat A_{(m_{\beta})}$ be its completion. Then we have $h_{A,m_\beta}(k_f)=h_{A_{(m_{\beta})},m_\beta}(k_f)=h_{\hat A_{(m_{\beta})},m_\beta}(k_f)$.

Now for general $(R\to k_f)\in \textbf{Art}_{/k_f}$, we have $h_{A,\bar \rho}(R)\simeq \bigsqcup_{m_{\beta}\in T_{\bar\rho_A}}h_{A,m_{\beta}}(R)$, which is induced by the stratification $h_{A,\bar\rho}(k_f)=\bigsqcup_{m_\beta\in T_{\bar\rho_A}}h_{A,m_\beta}(k_f)$ and the natural map $h_{A,\bar\rho}(R)\to h_{A,\bar\rho}(k_f)$. Now we claim that $h_{A,m_{\beta}}(R)\simeq h_{A_{(m_{\beta})},m_{\beta}}(R)\simeq h_{\hat A_{(m_{\beta})},m_{\beta}}(R)$. Let $\gamma:A\to k_f$ be a map in $h_{A,m_{\beta}}(k_f)$. Then we can define a functor $h_{A,\gamma}:{\textbf{Art}}_{/k_f}\to \textbf{Ani}$ by
\[
h_{A,\gamma}(R):={\rm Fib}_{\gamma}(h_A(R)\to h_A(k_f)).
\]
Similarly, we can define $h_{A_{(m_{\beta})},\gamma}$ and $h_{\hat A_{(m_{\beta})},\gamma}$. By \cite[Remark 6.2.9]{Lur09}, if we want to prove $h_{\hat A_{(m_{\beta})},\gamma}\simeq h_{A_{(m_{\beta})},\gamma}\simeq h_{A,\gamma}$, we just need to compare their tangent complexes. By \cite[Proposition 6.2.10 (2)]{Lur09}, the tangent complexes are the duals of $L_{\hat A_{(m_{\beta})}/\bZ}\otimes^{\bL}_{\hat A_{(m_{\beta})}}k_f$, $L_{ A_{(m_{\beta})}/\bZ}\otimes^{\bL}_{A_{(m_{\beta})}}k_f$, $L_{A/\bZ}\otimes^{\bL}_{A}k_f$ respectively. Now consider the natural maps $\bZ\to A\to A_{(m_{\beta})}$. There is an associated exact triangle
\[
L_{A/\bZ}\otimes^{\bL}_AA_{(m_{\beta})}\to L_{A_{(m_{\beta})}/\bZ}\to L_{A_{(m_{\beta})}/A},
\]
which induces an exact triangle
\[
L_{A/\bZ}\otimes^{\bL}_Ak_f\to L_{A_{(m_{\beta})}/\bZ}\otimes^{\bL}_{A_{(m_{\beta})}}k_f\to L_{A_{(m_{\beta})}/A}\otimes^{\bL}_{A_{(m_{\beta})}}k_f.
\]
Now we consider the exact triangle induced by $A\to A_{(m_{\beta})}\to k_f$
\[
L_{A_{(m_{\beta})}/A}\otimes^{\bL}_{A_{(m_{\beta})}}k_f\to L_{k_f/A}\to L_{k_f/A_{(m_{\beta})}}.
\]
As $A\to A_{(m_{\beta})}$ is flat, we then have $L_{k_f/A}\otimes^{\bL}_AA_{(m_{\beta})}\simeq L_{k_f\otimes_AA_{(m_{\beta})}/A_{(m_{\beta})}}$. Note that $k_f\otimes_AA_{(m_{\beta})}\simeq k_f$ and $L_{k_f/A}\otimes^{\bL}_AA_{(m_{\beta})}\simeq L_{k_f/A}\otimes^{\bL}_{k_f}k_f\otimes^{\bL}_AA_{(m_{\beta})}\simeq L_{k_f/A}$. Then we have $L_{k_f/A}\simeq L_{k_f/A_{(m_{\beta})}}$, which implies $L_{A_{(m_{\beta})}/A}\otimes^{\bL}_{A_{(m_{\beta})}}k_f\simeq 0$. So we get $L_{A/\bZ}\otimes^{\bL}_Ak_f\simeq L_{A_{(m_{\beta})}/\bZ}\otimes^{\bL}_{A_{(m_{\beta})}}k_f$. By using the facts that $A_{(m_{\beta})}\to \hat A_{(m_{\beta})}$ is flat and $k_f\otimes_{A_{(m_{\beta})}}\hat A_{(m_{\beta})}\simeq k_f$, we can run the same argument to deduce $L_{ A_{(m_{\beta})}/\bZ}\otimes^{\bL}_{A_{(m_{\beta})}}k_f\simeq L_{\hat A_{(m_{\beta})}/\bZ}\otimes^{\bL}_{\hat A_{(m_{\beta})}}k_f$. So we have proved $h_{\hat A_{(m_{\beta})},\gamma}\simeq h_{A_{(m_{\beta})},\gamma}\simeq h_{A,\gamma}$. This implies $h_{\hat A_{(m_{\beta})},m_{\beta}}(R)\simeq h_{A_{(m_{\beta})},m_{\beta}}(R)\simeq h_{A,m_{\beta}}(R)$.

Now we have a decomposition
\[
h_{A,\bar\rho}(R)\simeq \bigsqcup_{m_{\beta}\in T_{\bar \rho_A}}h_{\hat A_{(m_{\beta})},m_{\beta}}(R).
\]
Note that $\hat A_{(m_{\beta})}$ is a local ring. So we can also write
\[
h_{A,\bar\rho}(R)\simeq \bigsqcup_{m_{\beta}\in T_{\bar \rho_A}}h_{\hat A_{(m_{\beta})},\bar \rho}(R).
\]

Let $\hat A_{(m_{\beta}),n}$ be the quotient $\hat A_{(m_{\beta})}/m_{\beta}^n$ for any $n\geq 1$. We then have
\[
h_{\hat A_{(m_{\beta})},\bar \rho}(R)\simeq \varinjlim_nh_{\hat A_{(m_{\beta}),n},\bar \rho}(R) \ \ \text{and} \ \ h_{A,\bar \rho}(R)\simeq \bigsqcup_{m_{\beta}\in T_{\bar \rho_A}}\varinjlim_nh_{\hat A_{(m_{\beta}),n},\bar \rho}(R) .
\]

Now we want to investigate $\varinjlim_{{\rm Nilp}_{\bZ_p,\rm ft}^{\rm op}{}_{/{\calX_{\rm EG}}}}h_{A,\bar\rho}(R)$. Let $\bar \rho_A,\bar\rho_B\in {\rm Nilp}_{\bZ_p,\rm ft}^{\rm op}{}_{/{\calX_{\rm EG}}}$ be two elements in the overcategory and $f:\bar\rho_A\to \bar\rho_B$ be a 1-morphism. Then we have an induced map $f: h_{A,\bar\rho}(k_f)\to h_{B,\bar\rho}(k_f)$. In particular, this gives a map of sets of maximal ideals $f^{-1}: T_{\bar\rho_A}\to T_{\bar\rho_B}$. For each $m_{\beta,A}\in T_{\bar\rho_A}$, we have induced maps $f_{m_{\beta,A},n}:\hat B_{f^{-1}(m_{\beta,A}),n}\to \hat A_{m_{\beta,A},n}$. So the map $f:h_{A,\bar\rho}(R)\to h_{B,\bar\rho}(R)$ is actually induced by each $f_{m_{\beta,A},n}$, i.e.
\begin{equation}\label{decomposition of maps}
f=\bigsqcup_{m_{\beta,A}\in T_{\bar\rho_A}}\varinjlim_n f_{m_{\beta,A},n}:\bigsqcup_{m_{\beta,A}\in T_{\bar\rho_A}}\varinjlim_nh_{\hat A_{(m_{\beta,A}),n},\bar\rho}(R)\to \bigsqcup_{m_{\beta,B}\in T_{\bar\rho_B}}\varinjlim_nh_{\hat B_{(m_{\beta,B}),n},\bar \rho}(R). 
\end{equation}

Let ${\rm Art}^{\rm op}_{\bZ_p}{}_{/{\calX_{\rm EG}}}$ be the subcategory of ${\rm Nilp}_{\bZ_p,\rm ft}^{\rm op}{}_{/{\calX_{\rm EG}}}$ spanned by finite Artinian local $\bZ_p$-algebras. So we have a natural map
\[
\iota:\varinjlim_{(\Spec(C)\to\calX_{\rm EG})\in{\rm Art}^{\rm op}_{\bZ_p}{}_{/{\calX_{\rm EG}}}}h_{C,\bar\rho}(R)\to \varinjlim_{(\Spec(A)\to\calX_{\rm EG})\in{\rm Nilp}_{\bZ_p,\rm ft}^{\rm op}{}_{/{\calX_{\rm EG}}}}h_{A,\bar\rho}(R).
\]

We claim that $\iota$ is an equivalence. We prove this by constructing its quasi-inverse. For any $\bar \rho_A:h_A\to \calX_{\rm EG}$ in ${\rm Nilp}_{\bZ_p,\rm ft}^{\rm op}{}_{/{\calX_{\rm EG}}}$, we know that $h_{A,\bar\rho}(R)\simeq \bigsqcup_{m_{\beta}\in T_{\bar\rho_A}}\varinjlim_nh_{\hat A_{(m_{\beta}),n},\bar\rho}(R)$. Moreover as each $\hat A_{(m_{\beta}),n}$ is a finite Artinian local $\bZ_p$-algebra, there is a natural map 
\[
\bigsqcup_{m_{\beta}\in T_{\bar\rho_A}}\varinjlim_nh_{\hat A_{(m_{\beta}),n},\bar\rho}(R)\to \varinjlim_{(\Spec(C)\to\calX_{\rm EG})\in{\rm Art}^{\rm op}_{\bZ_p}{}_{/{\calX_{\rm EG}}}}h_{C,\bar\rho}(R)
\]
which is unique up to a contractible space of choices. Then we get a natural map $i_A: h_{A,\bar\rho}(R)\to \varinjlim_{{\rm Art}^{\rm op}_{\bZ_p}{}_{/{\calX_{\rm EG}}}}h_{C,\bar\rho}(R)$. By \ref{decomposition of maps}, this fits into a well-defined map
\[
\tilde\iota:\varinjlim_{(\Spec(A)\to\calX_{\rm EG})\in{\rm Nilp}_{\bZ_p,\rm ft}^{\rm op}{}_{/{\calX_{\rm EG}}}}h_{A,\bar\rho}(R)\to \varinjlim_{(\Spec(C)\to\calX_{\rm EG})\in{\rm Art}^{\rm op}_{\bZ_p}{}_{/{\calX_{\rm EG}}}}h_{C,\bar\rho}(R).
\]
It is easy to see that $\tilde\iota\circ \iota\simeq id$ by the construction of $\tilde\iota$. Now we consider the composite $\iota\circ \tilde\iota$. Let $j_A$ denote the natural map $h_{A,\bar\rho}(R)\to \varinjlim_{{\rm Nilp}_{\bZ_p,\rm ft}^{\rm op}{}_{/{\calX_{\rm EG}}}}h_{A,\bar\rho}(R)$. Then we need to see that $\iota\circ\tilde\iota\circ j_A\simeq j_A$ functorially. Note that $\tilde\iota\circ j_A$ is just $\bigsqcup_{m_{\beta}\in T_{\bar\rho_A}}\varinjlim_n i_{\hat A_{(m_{\beta}),n}}$. Then we see that $\iota\circ\tilde\iota\circ j_A$ is just $\bigsqcup_{m_{\beta}\in T_{\bar\rho_A}}\varinjlim_n j_{\hat A_{(m_{\beta}),n}}$, which is the same as $j_A$ up to a contractible space of choices. We then obtain $\iota\circ\tilde\iota\simeq id$. So the natural map $\iota$ is an equivalence.

Now go back to the goal of Step 1. As $^{\rm cl}\calR_{G_K,\GL_d/\GL_d}^{c}$ is also limit preserving, for any $(R\to k_f)\in \textbf{Art}_{/k_f}$, we also have
\[
(\Lan ^{\rm cl}\calR_{G_K,\GL_d/\GL_d}^{c})_{\bar \rho}(R\to k_f)=(\varinjlim_{{\rm Art}^{\rm op}_{\bZ_p}{}_{/{^{\rm cl}\calR_{G_K,\GL_d/\GL_d}^{c}}}}h_{C,\bar\rho}(R))\times_{^{\rm cl}\calR_{G_K,\GL_d/\GL_d}^{c}(k_f)}*_{\bar \rho}
\]
As for finite coefficient rings, the \'etale $(\varphi,\Gamma)$-modules are the same as Galois representations, we have 
\[
{\rm Art}^{\rm op}_{\bZ_p}{}_{/{^{\rm cl}\calR^{c}_{G_K,\GL_d/\GL_d}}}\simeq {\rm Art}^{\rm op}_{\bZ_p}{}_{/{\calX_{\rm EG}}}.
\]
by \cite[Lemma 2.4.4]{Zhu20}. So we are done.

\textbf{Step 2}: now we carry out the second step. We first claim that $(\Lan ^{\rm cl}\calR_{G_K,\GL_d}^{c})_{\bar \rho}\simeq\calR_{G_K,\GL_d,\bar\rho}^{c}$. By Proposition \ref{Zhu-colimit} and \ref{Zhu}, we see that $^{\rm cl}\calR_{G_K,\GL_d}^{c}\simeq \varinjlim_ih_{R_i}$ where each $R_i$ is a finite type $\bZ_p$-algebra. So $\Lan ^{\rm cl}\calR_{G_K,\GL_d}^{c}\simeq \varinjlim_ih_{R_i}$ where the colimit is taken in $\PreStk$. For any $(R\to k_f)\in \textbf{Art}_{/k_f}$, we then have
\[
(\Lan ^{\rm cl}\calR_{G_K,\GL_d}^{c})_{\bar \rho}(R\to k_f)\simeq \varinjlim_ih_{R_i}(R)\times_{\varinjlim_ih_{R_i}(k_f)}*_{\bar \rho}\simeq \varinjlim_ih_{R_i,\bar \rho}(R).
\]
By the same argument as in Step 1, we have $h_{R_i,\bar\rho}(R)\simeq h_{\hat R_{i,(m_i)},\bar\rho}(R)$ where $\hat R_{i,(m_i)}$ is the completion of the localization of $R_i$ at the maximal ideal determined by $\bar\rho$. The functor $h_{\hat R_{i,(m_i)},\bar\rho}$ is then pro-corepresentable by $\{\hat R_{i,(m_i),n}:=\hat R_{i,(m_i)}/m_i^n\}_n$. So we get the functor $(\Lan ^{\rm cl}\calR_{G_K,\GL_d}^{c})_{\bar \rho}$ is pro-corepresentable by $\{\hat R_{i,(m_i),n}\}_{i,n}$. In particular, $(\Lan ^{\rm cl}\calR_{G_K,\GL_d}^{c})_{\bar \rho}$ is the left Kan extension of $\varinjlim_{i,n}h_{\hat R_{i,(m_i),n}}: {\rm Art}_{/k_f}\to \textbf{Ani}$ along the inclusion ${\rm Art}_{/k_f}\to \textbf{Art}_{/k_f}$. Note that $\varinjlim_{i,n}h_{\hat R_{i,(m_i),n}}: {\rm Art}_{/k_f}\to \textbf{Ani}$ is just the classical framed deformation functor by \cite[Lemma 2.4.4]{Zhu20}, which is then equivalent to $^{\rm cl}\calF_{K,\GL_d,\bar\rho}^{\square}$. By Theorem \ref{BIP} and \cite[Lemma 7.5]{GV18}, we see that $\calF_{K,\GL_d,\bar\rho}^{\square}$ is also the left Kan extension of $^{\rm cl}\calF_{K,\GL_d,\bar\rho}^{\square}$ along the inclusion ${\rm Art}_{/k_f}\to \textbf{Art}_{/k_f}$. This implies we have $(\Lan ^{\rm cl}\calR_{G_K,\GL_d}^{c})_{\bar \rho}$ is equivalent to $\calF_{K,\GL_d,\bar\rho}^{\square}$. The latter by definition is just $\calR_{G_K,\GL_d,\bar\rho}^{c}$ (cf. \cite[Remark 2.2.2]{Zhu20}). By Lemma \ref{Keykey-lemma} and (the same proof of ) Lemma \ref{reduce to finite field}, we have $(\Lan ^{\rm cl}\calR_{G_K,\GL_d}^{c})^{\#,\nil}\simeq \calR_{G_K,\GL_d}^{c,\nil}$.

Now we are going to prove $(\Lan ^{\rm cl}\calR_{G_K,\GL_d/\GL_d}^{c})_{\bar \rho}\simeq \calR_{G_K,\GL_d/\GL_d,\bar \rho}^{c}$. By the fact that left Kan extension commutes with colimits and the definition of $\calR_{G_K,\GL_d/\GL_d}^{c}$ as the geometric realization of the following simplicial prestacks

\begin{equation}\label{classical unframed}
        \xymatrix{
        \cdots\ar@<-.9ex>[r]\ar@<-.3ex>[r]\ar@<.3ex>[r]\ar@<.9ex>[r]&\GL_d\times \GL_d\times  \calR_{G_K,\GL_d}^{c}\ar@<-.6ex>[r]\ar@<.0ex>[r]\ar@<.6ex>[r]&\GL_d\times \calR_{G_K,\GL_d}^{c}\ar@<-.3ex>[r]\ar@<.3ex>[r]&\calR_{G_K,\GL_d}^{c},
        }
    \end{equation}
It suffices to prove for any $(R\to k_f)\in \textbf{Art}_{/k_f}$, we have $(\Lan ^{\cl}\calR_{G_K,\GL_d}^{c})(R)\simeq \calR_{G_K,\GL_d}^{c}(R)$. But this follows from $(\Lan ^{\cl}\calR_{G_K,\GL_d}^{c})_{\bar \rho}\simeq \calR_{G_K,\GL_d,\bar \rho}^{c}$ and the same arguments in the first two paragraphs of Lemma \ref{Keykey-lemma}.
    


\textbf{Step 3}. Now it remains to prove the third step: $\calR_{G_K,\GL_d/\GL_d,\bar\rho}^{c}\simeq \calF_{K,\GL_d,\bar\rho}$. This follows directly from the definition of both sides. More precisely, the framed deformation functor $\calF^{\square}_{K,\GL_d}$ (cf. \cite[Definition 5.4(i)]{GV18}) is equivalent to the framed deformation functor $\calR^{c}_{G_K,\GL_d}$ by \cite[Remark 2.2.2]{Zhu20}. By \cite[Definition 5.4(i)]{GV18}, for any animated ring $R$, we have the pullback diagram
\begin{equation*}
    \xymatrix@=0.6cm{
    \calF^{\square}_{K,\GL_d}(R)\ar[r]^{}\ar[d]^{}&{*}\ar[d]^{}\\
    \calF_{K,\GL_d}(R)\ar[r]^{}&|\GL_d(R)|.
    }
\end{equation*}
The base change of the simplicial anima $\GL_d^{\bullet}(R)$ along $\calF_{K,\GL_d}\to|\GL_d(R)|$ gives a simplial anima, which is the same as the simplicial anima in \cite[Definition 2.2.14]{Zhu20}. The homotopy colimit of this simplicial anima is just $\calF_{K,\GL_d}(R)$ by the fact that the colimit in the $\infty$-topos $\textbf{Ani}$ is universal. This shows $\calF_{K,\GL_d}(R)$ is equivalent to $\calR^{c}_{G_K,\GL_d/\GL_d}(R)$ (see also \cite[Remark 2.2.15]{Zhu20}.

\end{proof}

The following lemma has been used above and enables us to deal with sheafification.

\begin{lem}\label{Keykey-lemma}
    Let $\calY$ be a derived prestack which satisifies the \'etale sheaf property when restricted to $\textbf{Nilp}_{\bZ_p}^{<\infty}
    $. Assume $\calY$ admits a deformation theory and for all points $\bar \rho:\Spec(k_f)\to \calY$ with $k_f$ a finite field, there is an equivalence of functors $(\Lan^{\rm cl}\calY)_{\bar\rho}\simeq \calY_{\bar \rho}$. Then we have $(\Lan^{\rm cl}\calY)^{\#}_{\bar\rho}\simeq \calY_{\bar \rho}$ for all $\bar \rho$.
\end{lem}

\begin{proof}
    Write $\calY(k_f)$ as $\bigsqcup_{\alpha\in H}[*_{\alpha}/({\rm Aut}(*_{\alpha})]$ where $H$ is the set of connected components of $\calY(k_f)$. By the assumption, for any $(R\to k_f)\in {{\textbf{Art}}_{/k_f}}$, we have $(\Lan^{\rm cl}\calY)_{*_{\alpha}}(R\to k_f)\simeq \calY_{*_\alpha}(R\to k_f)$. Consider the following pullback diagram
    \begin{equation*}
    \xymatrix@=0.6cm{
    \calY_{*\alpha}(R)\ar[r]^{}\ar[d]^{}& [*_{\alpha}/({\rm Aut}(*_{\alpha})]\ar[d]^{}\\
    \calY(R)\ar[r]^{}& \calY(k_f).
    }
\end{equation*}

Note that $[*_{\alpha}/({\rm Aut}(*_{\alpha})]$ is the geometric realization (i.e. the homotopy colimit) of simplicial anima corresponding to the automorphism group ${\rm Aut}(*_{\alpha})$. We write $[*_{\alpha}/({\rm Aut}(*_{\alpha})]=\varinjlim{\rm Aut}(*_{\alpha})^{\bullet}$. Base changing along $\calY_{*\alpha}(R)\to [*_{\alpha}/({\rm Aut}(*_{\alpha})]$, we get another simplicial anima $\calY_{*_\alpha}(R)^{\bullet}$. As the colimit in $\infty$-topos $\textbf{Ani}$ is universal, i.e. pullback preserves colimits, we have $\calY_{*_{\alpha}}(R)\simeq \varinjlim \calY_{*\alpha}(R)^{\bullet}$. The same is true for $\Lan^{\rm cl}\calY$. In particular, we get $\Lan^{\rm cl}\calY(R)\simeq \calY(R)$ for any Artinian local ring with finite residue field. 

Now consider any \'etale map $R_1\to R_2$ with $R_1$ an Artinian local ring with finite residue field $k_1$. Then $\pi_0(R_2)$ is also an Artinian ring and $\pi_i(R_2)=\pi_0(R_2)\otimes_{\pi_0(R_1)}\pi_i(R_1)$ is a finite $\pi_0(R_2)$-module for each $i$. By definition, we know $R_2$ is also Artinian. In particular, $R_2$ can be written uniquely as a finite product $\prod_jR_{2j}$ of Artinian local rings with finite residue fields (see  \cite[Section 6.2]{Lur09}). As $\calY$ is an \'etale stack when restricted to truncated animated rings, we have $\calY(R_2)=\prod_j\calY(R_{2j})$. We claim that this is also true for $\Lan^{\rm cl}\calY$. In fact, as $^{\rm cl}\calY$ is an \'etale stack, we have $^{\rm cl}\calY:{\rm Nilp}_{\bZ_p}\to \textbf{Ani}$ preserves finite products. This is equivalent to that $^{\rm cl}\calY$ can be written as a sifted colimit of representable sheaves, i.e. $^{\rm cl}\calY\simeq \varinjlim_ih_{A_i}$. As left Kan extension preserves representable sheaves, we have $\Lan^{\rm cl}\calY\simeq \varinjlim_ih_{A_i}$ where the colimit is taken in $\PreStk$. Again, by using that a functor preserving finite products is equivalent to that it is a sifted colimit of representables, we conclude $\Lan^{\rm cl}\calY$ preserves finite products. Hence we have $\Lan^{\rm cl}\calY(R_2)\simeq \calY(R_2)$.

Now by the description of the sheafification functor \cite[Remark 6.2.2.12]{Lur09a} (and the proof of \cite[Proposition 6.2.2.7]{Lur09a}), we have $(\Lan^{\rm cl}\calY)^{\#}(R)\simeq (\Lan^{\rm cl}\calY)(R)$ for any Artinian local rings $R$ with finite residue fields as the sheafification process only involves \'etale coverings of such rings (see also \cite[Chapter 2, Section 2.5.2]{GR19}). In particular, this means $(\Lan^{\rm cl}\calY)^{\#}_{\bar\rho}\simeq \calY_{\bar \rho}$.
\end{proof}

Finally, we can prove Theorem \ref{main2}.

\begin{proof}[Proof of Theorem \ref{main2}]
    Just combine Proposition \ref{derivedperf}, Proposition \ref{derivedrep-laurent}, Proposition \ref{X-F}, Lemma \ref{Keykey-lemma} and Lemma \ref{reduce to finite field}.
\end{proof}

\subsection{Derived \'etale $(\varphi,\Gamma)$-modules}\label{3.4}

Throughout this section, we assume $K$ is absolutely unramified. We will give an ad-hoc definition of derived \'etale $(\varphi,\Gamma)$-modules and prove they are equivalent to derived Laurent $F$-crystals.

Recall there is the cyclotomic prism $(\bar A_K^+,([p]_q)$ (cf. Definition \ref{cyl-prism}) in $(\calO_K)_{\Prism}$, which is equipped with compatible $\varphi$-action and $\Gamma$-action. For each $n\geq 0$, we define $\bar A_K^{+,n}/p^a:=C(\Gamma^n,\bar A_K^+/p^a)$. In particular, we get a natural cosimplicial ring $\bar A_K^{+,\bullet}/p^a$ by using the $\Gamma$-action on $\bar A_K^{+}/p^a$.

\begin{dfn}
   Let $R$ be a truncated animated ring which is of finite type over $\bZ/p^a$. We define the $\infty$-category $\Mod^{\varphi,\Gamma}(A_{K,R})$ of \'etale $(\varphi,\Gamma)$-modules with coefficient in $R$ to be the limit of the following cosimplicial $\infty$-category
   
  \begin{equation*}
      \xymatrix{
  \Vect(\bar A_K^{+}/p^a\widehat\otimes R[\frac{1}{[p]_q}])^{\varphi=1}\ar@<-.4ex>[r]\ar@<.4ex>[r]&\Vect(\bar A_K^{+,1}/p^a\widehat\otimes R[\frac{1}{[p]_q}])^{\varphi=1}\cdots\Vect(\bar A_K^{+,i}/p^a\widehat\otimes R[\frac{1}{[p]_q}])^{\varphi=1}\cdots
      }
  \end{equation*}
\end{dfn}

\begin{dfn}
    Let $\tilde\calX:{\textbf{ Nilp}_{\bZ_p\rm ft}^{<\infty}}\to \textbf{Ani}$ be the functor sending each $R\in {\textbf{Nilp}_{\bZ_p,\rm ft}^{<\infty}}$ to the anima $\Mod^{\varphi,\Gamma}(A_{K,R})^{\simeq}$. By abuse of notation, we still use $\tilde\calX$ to denote the nilcompletion of its left Kan extension along ${\textbf{ Nilp}_{\bZ_p,\rm ft}^{<\infty}}\to {\textbf{ Nilp}_{\bZ_p}^{<\infty}}$, which is called the derived prestack of \'etale $(\varphi,\Gamma)$-modules.
\end{dfn}

Now we proceed to prove $\tilde\calX\simeq \calX^{\nil}$. As both of them are locally almost of finite type, we just need to compare their evaulation on ${\textbf{ Nilp}_{\bZ_p,\rm ft}^{<\infty}}$. 

\begin{lem}\label{derived-cyclo}
    Let $R$ be an animated ring over $\bZ/p^a$. Then for each $n$, we have
    \[
    (\varinjlim_{\varphi}\bar A_{K}^{+,n}\widehat\otimes R)^{\wedge}_{[p]_q,\rm derived}\simeq C(\Gamma^n,A_{\cyc}/p^a)\widehat\otimes R.
    \]
\end{lem}

\begin{proof}
    It suffices to prove
    \[
    (\varinjlim_{\varphi}\bar A_{K}^{+,n})/[p]_q\cong C(\Gamma^n,A_{\cyc}/p^a)/[p]_q.
    \]
    This follows from $(\varinjlim_{\varphi}\bar A_K^+)/(p^a,[p]_q)\cong A_{\cyc}/(p^a,[p]_q)$ and that $\Gamma^n$ is compact.
\end{proof}
\begin{prop}\label{fully-faithful-cylco}
    Let $R$ be a $n$-truncated animated ring which is finite type over $\bZ/p^a$.  Let $\calF_R:\Delta_{\leq n+2}\to \textbf{Ani}$ be the natural functor sending each $[i]$ to the anima $\Vect(C(\Gamma^i,A_{\cyc}/p^a)\widehat\otimes R[\frac{1}{[p]_q}])^{\varphi=1,\simeq}$ and $\tilde\calX^{\rm perf}(R)=\varprojlim_{\Delta_{\leq n+1}}\calF_R$. Then the natural functor $\tilde\calX(R)\to \tilde\calX^{\rm perf}(R)$ by base change is fully faithful.
\end{prop}

\begin{proof}
    The argument is the same as that in the proof of Proposition \ref{finitetype}. In particular, we need to use Lemma \ref{derived-cyclo} and Remark \ref{cyclo} so that we can apply Lemma \ref{F-invariant} to $\bar A_K^+$.
\end{proof}

Now in order to prove $\tilde\calX(R)\simeq \tilde\calX^{\rm perf}(R)$, we need to prove the following result.

\begin{prop}\label{essential-surj}
    Let $R$ be a $n$-truncated animated ring which is finite type over $\bZ/p^a$. Then the following base change functor is an equivalence
    \[
    \Vect(\bar A_K^+\widehat\otimes R[\frac{1}{[p]_q}])^{\varphi=1}\xrightarrow{\simeq} \Vect(A_{\cyc}\widehat\otimes R[\frac{1}{[p]_q}])^{\varphi=1}.
    \]
\end{prop}

Before proving Proposition \ref{essential-surj}, we need to make some preparations when $R$ is discrete.

\begin{lem}\label{prep-discrete}
   Let $R$ be a discrete finite type $\bZ/p^a$-algebra for some $a$. Then we have the following results.
   \begin{enumerate}
       \item The natural map $\bar A_{K}^+/[p]_q^n\otimes R\to A_{\cyc}/[p]_q^n\otimes R$ is faithfully flat for each $n$. In particular, $\bar A_{K,R}^+:=\bar A_{K}^+\widehat\otimes R\to A_{\cyc,R}$ is faithfully flat.

       \item $A_{\cyc,R}$ is exactly the (derived) $[p]_q$-adic completion of $\varinjlim_{\varphi}\bar A_{K,R}^+$.
      
   \end{enumerate}
\end{lem}

\begin{proof}
    For item (1), note that $\bar A_K^+/p\cong E_K^+\to A_{\cyc}/p$ is faithfully flat as $E_K^+$ is a complete discrete valuation ring (cf. \cite[Lemma 2.5]{Wu21}) and $A_{\cyc}/p$ is an integral domain. Then item (1) follows from the proof of \cite[Proposition 2.2.12]{EG22}.

    For item (2), we first note that $A_{\cyc,R}[[p]_q]=0$. In fact, this follows from $A_{\inf,R}[[p]_q]=0$ as the natural map $A_{\cyc,R}\to A_{\inf,R}$ is faithfully flat. To see $A_{\inf,R}[[p]_q]=0$, we use  the faithfully flat map $\frakS_R=(W(k)\otimes R)[[T]]\to A_{\inf,R}$ and the fact that $A_{\inf,R}[[p]_q]=0$ if and only if $A_{\inf,R}[T]=0$. Now by item (1), we know that $\bar A_{K,R}^+\to A_{\cyc,R}$ is injective and so is $\varinjlim_{\varphi}\bar A_{K,R}^+\to A_{\cyc,R}$ as $\varphi$-action on $A_{\cyc,R}$ is an isomorphism. In particular, $(\varinjlim_{\varphi}\bar A_{K,R}^+)[[p]_q]=0$ and the derived $[p]_q$-adic completion of $\varinjlim_{\varphi}\bar A_{K,R}^+$ is the same as the classical $[p]_q$-adic completion. By derived Nakayama lemma, it suffices to prove $(\varinjlim_{\varphi}\bar A_{K,R}^+)/[p]_q\to A_{\cyc,R}/[p]_q$ is an isomorphism, which follows from $(\varinjlim_{\varphi}\bar A_K^+)/(p^a,[p]_q)\cong A_{\cyc}/(p^a,[p]_q)$.

\end{proof}

\begin{proof}[Proof of Proposition \ref{essential-surj}]
    By Proposition \ref{fully-faithful-cylco}, we only need to prove this functor is essentially surjective.

    We first deal with the case that $R$ is discrete. In fact, with the help of Lemma \ref{prep-discrete}, we can argue in the same way as \cite[Proposition 2.6.9]{EG22}. In particular, as $\bar A_K^+$ is $\varphi$-stable, our situation is less involved. We could simply take $T$ in loc.cit to be $[p]_q^m$ for some large enough $m$ as $\varphi([p]_q)=[p]_q^p+p\delta([p]_q)$ and $p^a=0$.

    In general, we have the following commutative diagram\begin{equation*}
        \xymatrix{
        {\rm Ho}(\Vect(\bar A_K^+\widehat\otimes R[\frac{1}{[p]_q}])^{\varphi=1})\ar[r]^{l_1}\ar[d]^{l_2}&{\rm Ho}(\Vect(A_{\cyc}\widehat\otimes R[\frac{1}{[p]_q}])^{\varphi=1})\ar[d]^{l_3}\\
        \Vect(\bar A_K^+\widehat\otimes \pi_0(R)[\frac{1}{[p]_q}])^{\varphi=1}\ar[r]^{l_4}&\Vect(A_{\cyc}\widehat\otimes \pi_0(R)[\frac{1}{[p]_q}])^{\varphi=1}
        }
    \end{equation*}

    Note that $\pi_0(A\widehat \otimes R)=A\widehat\otimes \pi_0(R)$ for $A=\bar A_K^+, A_{\cyc}$ by Lemma \ref{pi-complete}. Now by the discrete case, we know the functor $l_4$ is an equivalence. Also by the proof of Proposition \ref{geometrically finite type}, we see that $l_2$ and $l_3$ are both essentially surjective. Then by Lemma \ref{lift lemma}, we can conclude $l_1$ is essentially surjective. So we are done.
\end{proof}

\begin{cor}\label{varphi-gamma-perf}
    Let $R$ be a $n$-truncated animated ring which is finite type over $\bZ/p^a$. Then we have an equivalence $\tilde\calX(R)\simeq \tilde\calX^{\rm perf}(R)$.
\end{cor}
\begin{proof}
    Just combine Proposition \ref{fully-faithful-cylco}, Proposition \ref{essential-surj} and Lemma \ref{tot}.
\end{proof}

Next we will compare $\tilde\calX^{\rm perf}(R)$ with $\calX(R)$.
\begin{lem}\label{continuous function-ani}
    Let $R$ be a $n$-truncated animated ring which is finite type over $\bZ/p^a$. Then we have 
    \[
    A_{\cyc,R}^n[\frac{1}{[p]_q}]\simeq C(\Gamma^n,A_{\cyc}/p^a)\widehat\otimes R[\frac{1}{[p]_q}].
    \]
\end{lem}

\begin{proof}
    Note that by assumption, $\pi_*(R)=\oplus_n\pi_n(R)$ is a finite type $\bZ_p$-algebra. In the setting of Lemma \ref{Continuous functions}, we have $K\otimes^{\bL}_{\bZ/p^a}\pi_*(R)$ is killed by some bounded power of $\xi$. This means for each $n$, $K\otimes^{\bL}_{\bZ/p^a}\pi_n(R)$ is also killed by some bounded power of $\xi$. In particular, we have
     \[
    A_{\cyc,\pi_n(R)}^n[\frac{1}{[p]_q}]\simeq C(\Gamma^n,A_{\cyc}/p^a)\widehat\otimes \pi_n(R)[\frac{1}{[p]_q}].
    \]

    By Lemma \ref{pi-complete}, we know that $\pi_n(A_{\cyc,R}^n[\frac{1}{[p]_q}])\cong A_{\cyc,\pi_n(R)}^n[\frac{1}{[p]_q}]$ and \[
    \pi_n(C(\Gamma^n,A_{\cyc}/p^a)\widehat\otimes R[\frac{1}{[p]_q}])\cong C(\Gamma^n,A_{\cyc}/p^a)\widehat\otimes \pi_n(R)[\frac{1}{[p]_q}]\]. So we are done.
\end{proof}

As an immediate corollary, we can compare $\tilde\calX^{\rm perf}(R)$ with $\calX(R)$ now.
\begin{cor}\label{perf-crystal}
    Let $R$ be a $n$-truncated animated ring which is finite type over $\bZ/p^a$. Then we have 
    \[
    \tilde\calX^{\rm perf}(R)\simeq \calX(R)
    \]
\end{cor}
\begin{proof}
    This follows from Lemma \ref{continuous function-ani} and Proposition \ref{derivedperf}.
\end{proof}

Finally we come to the main theorem of this subsection.

\begin{thm}
    There is an equivalence of derived prestacks $\calX^{\rm nil}\simeq\tilde\calX$.
\end{thm}
\begin{proof}
    As both of them are locally almost of finite type and nilcomplete, this theorem follows from Corollary \ref{varphi-gamma-perf} and Corollary \ref{perf-crystal}.
\end{proof}



\begin{appendices}
\renewcommand{\thesection}{Appendix \Roman{section}}

\section{A lemma on algebraic cotangent complex}\label{appendix}
\renewcommand{\thesection}{\Roman{section}}
The derived deformation theory in \cite[Chapter 1]{GR17} is only established in the characteristic $0$ case. However, almost all results in \cite[Chapter 1]{GR17} that we need apply to animated rings straightforwardly. The only place one needs to be careful about is \cite[Chapter 1, Lemma 5.4.3]{GR17}, which concerns the topological cotangent complex of $\bE_{\infty}$-rings. The goal of this appendix is to establish this result in the setting of animated rings, i.e. consider the algebraic cotangent complex of animated rings.

\begin{lem}\label{lemma-animated}
    Let $i:A\to B$ be a morphism of animiated rings such that $H^0(A)\to H^0(B)$ is surjective. Consider the corresponding fiber sequence
    \[
    J\to A\to B.
    \]
    Then
    \begin{enumerate}
        \item $H^0(L_{B/A})=0$ and $H^{-1}(L_{B/A})=H^0(B\otimes_AJ)$;
        \item For $n\geq 0$, we have
        \[
        \tau^{\geq -n}(J)=0\Longrightarrow \tau^{\geq-n-1}(L_{B/A})=0.\footnote{\text{To keep compatible with \cite[Chapter 1]{GR17}, we use the cohomological convention here.}}
        \]
        In the latter case, $H^0(A)=H^0(B)$ and $H^{-n-2}(L_{B/A})=H^{-n-1}(J)$.
    \end{enumerate}
\end{lem}

To prove this lemma, let's recall the following result concerning the Hurewicz map associated with $i:A\to B$.
\begin{prop}[{\cite[Proposition 25.3.6.1]{Lur18}}]\label{Hure}
    Let $i:A\to B$ be a morphism of animated rings. Then the Hurewicz map
    \[
    \epsilon_i: B\otimes_A {\rm cofib}(i)\to L_{B/A}
    \]
    is surjective on $H^0$. If the fiber of $i$ is connective, then the fiber of $\epsilon_i$ is $2$-connective. If the fiber of $i$ is $m$-connective for $m>0$, then the fiber of $\epsilon_i$ is $m+3$-connective.
\end{prop}

\begin{proof}[Proof of Lemma \ref{lemma-animated}]
\begin{enumerate}
    \item As $H^0(A)\to H^0(B)$ is surjective, we see that $H^0({\rm cofib}(i))=0$. Then $H^0(B\otimes_A{\rm cofib}(i))=0$. By Proposition \ref{Hure}, we see that $H^0(L_{B/A})=0$. As $J$ is connective, we have $H^{-1}(B\otimes_A{\rm cofib}(i))\cong H^{-1}(L_{B/A})$ by Proposition \ref{Hure}. Since $H^{-1}(B\otimes_A{\rm cofib}(i))$ is just $H^0(B\otimes_AJ)$, we get $H^0(B\otimes_AJ)\cong H^{-1}(L_{B/A})$.
    \item Suppose $\tau^{\geq -n}(J)=0$, i.e. $J$ is $(n+1)$-connective. Then by Proposition \ref{Hure}, we see that ${\rm fib}(\epsilon_i)$ is $(n+4)$-connective. Hence we have $H^i(B\otimes_A{\rm cofib}(i))=H^i(L_{B/A})=0$ for all $i\geq -(n+1)$. In particular, as $H^0(J)=0$, we see $H^0(A)=H^0(B)$ and \[H^{-n-2}(L_{B/A})=H^{-n-1}(B\otimes_AJ)=H^0(B\otimes_AJ[-n-1])=H^0(B)\otimes_{H^0(A)}H^0(J[-n-1])=H^{-n-1}(J).\]
\end{enumerate}

\end{proof}
\end{appendices}

\addcontentsline{toc}{section}{References}
\bibliographystyle{alpha}

\bibliography{main.bib}

\end{document}